\documentclass[12pt]{amsart}
\usepackage[applemac]{inputenc}

\usepackage[T1]{fontenc}

\usepackage{ulem}
\usepackage{amsmath,esint,color}
\usepackage{amsthm,latexsym,epsfig,graphicx}
\usepackage{amsmath}
\usepackage{amsfonts}
\usepackage{amssymb}
\usepackage{bigints}
\usepackage{fullpage}
\usepackage[colorlinks=true, pdfstartview=FitV, linkcolor=blue, citecolor=blue, urlcolor=blue]{hyperref}

\date{}
\definecolor{sah}{rgb}{0.66,0.33, 0.04}
\definecolor{adel4}{cmyk}{1,0,0,0}
\definecolor{adel3}{rgb}{0.66,0.33, 0.04}
\definecolor{adel1}{cmyk}{0,0.20,1,0}
\definecolor{adel2}{cmyk}{0,0.40,1,0.30}
\definecolor{adel0}{rgb}{0.99,0.60, 0.30}
\definecolor{trut}{rgb}{0.99,0.80, 0.00}
\definecolor{trus}{rgb}{0.00, 0.50, 0.00}
 \definecolor{trust}{rgb}{0.99, 0.99, 0.80}
\definecolor{MaCouleur}{rgb}{0,0.9,0.3}

\newcommand{\CC}{\mathbb{C}}
\newcommand{\NN}{\mathbb{N}}
\newcommand{\RR}{\mathbb{R}}

\newcommand{\C}{\mathbb{C}}

\theoremstyle{plain}
\newtheorem{definition}{Definition}
\newtheorem{theorem}{Theorem}
\newtheorem{proposition}{Proposition}
\newtheorem{lemma}{Lemma}
\newtheorem{remark}{Remark}

\def\virgp{\raise 2pt\hbox{,}}

\def\Xint#1{\mathchoice
 {\XXint\displaystyle\textstyle{#1}}%
 {\XXint\textstyle\scriptstyle{#1}}%
 {\XXint\scriptstyle\scriptscriptstyle{#1}}%
 {\XXint\scriptscriptstyle\scriptscriptstyle{#1}}%
 \!\int}
\def\XXint#1#2#3{{\setbox0=\hbox{$#1{#2#3}{\int}$}
 \vcenter{\hbox{$#2#3$}}\kern-.5\wd0}}

\def\av_#1{\Xint-_{#1}}

\usepackage[textmath,displaymath,floats,graphics,delayed]{preview}
\title[]{An analytical and numerical study of steady patches in the disc}
\author{}
\author[F. de la Hoz]{\tiny{ Francisco de la Hoz}}
\address{ Department of Applied Mathematics and Statistics and Operations Research, Faculty of Science and Technology \\
University of the Basque Country UPV/EHU, Barrio Sarriena S/N\\
 48940 Leioa, Spain
 }
\email{francisco.delahoz@ehu.eus}
\author[]{ZINEB HASSAINIA}
\address{The Courant Institute for Mathematical Sciences, New York University\\
251 Mercer Street
New York, N.Y. 10012-1185}
 \email{ zineb.hassainia@cims.nyu.edu}
 \author[T. Hmidi]{Taoufik Hmidi}
\address{IRMAR, Universit\'e de Rennes 1\\ Campus de
Beaulieu\\ 35~042 Rennes cedex\\ France}
\email{thmidi@univ-rennes1.fr}
\author[J. Mateu]{Joan Mateu}
\address{Departament de Matem\`{a}tiques\\
Universitat Aut\`{o}noma de Barcelona\\
08193 Bellaterra, Barcelona, Catalonia}
\email{mateu@mat.uab.cat}
\begin{document}
\begin{abstract}
In this paper, we prove the existence of $m$-fold rotating patches for the Euler equations in the disc, for both simply-connected and doubly-connected cases. Compared to the planar case, the rigid boundary introduces rich dynamics for the lowest symmetries $m=1$ and $m=2$. We also discuss some numerical experiments highlighting the interaction between the boundary of the patch and the rigid one.
 \end{abstract}
\maketitle{}
\tableofcontents
\section{Introduction}

In this paper, we shall discuss some aspects of the vortex motion for the Euler system in
the unit disc $\mathbb{D}$ of the Euclidean space $\mathbb{R}^2$. That system is described by the following equations:
\begin{equation}
\label{E1}
\begin{cases}
\partial_{t}v+v\cdot\nabla v+\nabla p=0,\quad(t,x)\in\RR_+\times\mathbb{D}, \\
\textnormal{div}\,v=0,\\
v\cdot \nu=0,\quad \textnormal{on}\quad \partial \mathbb{D}, \\
v_{|t=0}=v_0.
\end{cases}
\end{equation}
Here, $v=(v^1,v^2)$ is the velocity field, and the pressure $p$ is a scalar potential that can be related to the velocity using the incompressibility condition. The boundary equation means that there is no matter flow through the rigid boundary $\partial \mathbb{D}=\mathbb{T}$;
the vector $\nu$ is the outer unitary vector orthogonal to the boundary.
The main feature of two-dimensional flows is that they can be illustrated through their vorticity structure; this can be identified with the scalar function $\omega=\partial_1v_2-\partial_2 v_1$, and its evolution is governed by the nonlinear transport equation:
\begin{equation}\label{vorticity}
\partial_{t}\omega +v\cdot \nabla\, \omega=0.
\end{equation}
To recover the velocity from the vorticity, we use the stream function $\Psi$, which is defined as the unique solution of the Dirichlet problem on the unit disc:
\begin{equation*}
\begin{cases}
\Delta\Psi=\omega,
 \cr
\psi|_{\partial\mathbb{D}}=0.
\end{cases}
\end{equation*}
Therefore, the velocity is given by
\begin{equation*}
v=\nabla^\perp\Psi,\quad \nabla^\perp=(-\partial_2,\partial_1).
\end{equation*}
By using the Green function of the unit disc, we get the expression
\begin{equation}\label{Integ71}
\Psi(z)=\frac{1}{4\pi}\int_{\mathbb{D}}\log\Big|\frac{z-\xi}{1-z\overline{\xi}}\Big|^2\omega(\xi)dA(\xi),
\end{equation}
with $dA$ being the planar Lebesgue measure. In what follows, we shall identify the Euclidean and the complex planes; so the velocity field is identified with the complex function
$$v(z)=v_1(x_1,x_2)+i\, v_2(x_1,x_2),\quad\hbox{with}\quad z=x_1+ix_2.$$
Therefore, we get the compact formula
\begin{align}
\label{vcomplex}
v(t,z)&=2i\,\partial_{\overline{z}}\,\Psi(t,z)
 \cr
& = \frac{i}{2\pi} \int_{\mathbb{D}}
\frac{|\xi|^2-1}{\big(\overline{z}-\overline{\xi}\big)\big(\xi\overline{z}-1\big)}\omega(t,\xi)\,dA
(\xi)
 \cr
&=\frac{i}{2\pi} \int_{\mathbb{D}}
\frac{\omega(t,\xi)}{\overline{z}-\overline{\xi}}\,dA(\xi)+\frac{i}{2\pi} \int_{\mathbb{D}}
\frac{\xi}{1-\xi\overline{z}}\omega(t,\xi)\,dA(\xi).
\end{align}
We recognize in the first part of the last formula the structure of the Biot-Savart law in the plane $\RR^2$, which is given by
\begin{equation}
v(t,z)=\frac{i}{2\pi} \int_\C
\frac{\omega(t,\xi)}{\overline{z}-\overline{\xi}}\,dA(\xi),\quad z \in \C.
\end{equation}
The second term of \eqref{vcomplex} is absent in the planar case. It describes the contribution of the rigid boundary $\mathbb{T}$, and our main task is to investigate the boundary effects on the dynamics of special long-lived vortex structures. Before going further into details, we recall first that, from the equivalent formulation \eqref{vorticity}-\eqref{vcomplex} of the Euler system \eqref{E1}, Yudovich was able in \cite{Y1} to construct a unique global solution in the weak sense, provided that the initial vorticity
$\omega_0$ is compactly supported and bounded. This result is very important, because it allows to deal rigorously with vortex patches, which are vortices uniformly distributed in a bounded region $D$, i.e., $\omega_0=\chi_{D}$. These structures are preserved by the evolution and, at each time $t$, the vorticity is given by $\chi_{D_t}$, with $D_t=\psi(t,D)$ being the image of $D$ by the flow. As we shall see later in \eqref{vort-dy}, the contour dynamics equation of the boundary $\partial D_t$ is described by the following nonlinear integral equation. Let $\gamma_t :\mathbb{T}\to \partial D_t$ be the Lagrangian parametrization of the boundary, then
\begin{equation*}
\partial_t\gamma_t=-\frac{1}{2\pi}\int_{\partial D_t}\log|\gamma_t-\xi|d\xi+\frac{1}{4\pi}\int_{\partial D_t}\frac{|\xi|^2}{1-\overline{\gamma_t}\,\xi}d\xi.
\end{equation*}
We point out  that, when the initial boundary is smooth enough, roughly speaking more regular than $C^1$, then the regularity is propagated for long times without any loss. This was first achieved by Chemin \cite{Ch} in the plane, and extended in bounded domains by \mbox{Depauw \cite{Dep}.} Note also that we can find in \cite{B-C} another proof of Chemin's result.
It appears that the boundary dynamics of the patch is very complicate to tackle and, up to our knowledge, the only known explicit example is the stationary one given by a small disc centered at the origin. Even though explicit solutions form a poor class, one can try to find implicit patches with prescribed dynamics, such as rotating patches, also known as $V$-states. These patches are subject to perpetual rotation around some fixed point that we can assume to be the origin, and with uniform angular velocity $\Omega$; this means that $D_t=e^{it\Omega} D$.
We shall see in \mbox{Section \ref{bound135}} that the $V$-states equation, when $D$ is symmetric with respect to the real axis, is given by,
 \begin{equation}\label{rotsssq}
\,\textnormal{Re}\Big\{ \Big(2\Omega \overline{z}+\fint_\Gamma\frac{\overline{z}-\overline{\xi}}{z-\xi}d\xi-\fint_\Gamma\frac{|\xi |^2}{1-z\xi}d\xi\Big)\, z^\prime\Big\}=0, \quad z\in \Gamma\triangleq\partial D,
\end{equation}
with $z^\prime$ being a tangent vector to the boundary $\partial D_0$ at the point $z$; remark that we have used the notation $\fint_\Gamma\equiv\frac{1}{2i\pi}\int_\Gamma$.
In the flat case, the boundary equation \eqref{rotsssq} becomes
\begin{equation}\label{rotsssqq}
\,\textnormal{Re}\Big\{ \Big(2\Omega \overline{z}+\fint_\Gamma\frac{\overline{z}-\overline{\xi}}{z-\xi}d\xi\Big)\, z^\prime\Big\}=0, \quad z\in \Gamma.
\end{equation}
Note that circular patches are stationary solutions for \eqref{rotsssqq}; however, elliptical vortex patches perform a steady rotation about their centers without changing shape. This latter fact was discovered by Kirchhoff \cite{Kirc}, who proved that, when $D$ is an ellipse centered at zero, then $D_t = e^{i t \Omega}\,D$, where the angular
velocity $\Omega$ is determined by the semi-axes $a$ and $b$ through the formula $\Omega = ab/(a+b)^2$. These ellipses are often referred to in the literature as the Kirchhoff elliptic vortices; see for instance \cite[p. 304]{BM} or \cite [p. 232]{L}.

One century later, several examples of rotating patches were obtained by Deem and Zabusky \cite{DZ}, using contour dynamics simulations. Few years later, Burbea gave an analytical proof and showed the existence of
 $V$-states with $m$-fold symmetry for each integer $m \geq 2$. In this countable family, the case $m=2$ corresponds to the Kirchhoff elliptic vortices.
Burbea's approach consists in using complex analysis tools, combined with bifurcation theory. It should be noted that, from this standpoint, the rotating patches are arranged in a collection of countable curves bifurcating from Rankine vortices (trivial solution) at the discrete angular velocities set $\big\{\frac{m-1}{2m}, m\geq 2\big\}.$ The numerical analysis of limiting $V$-states which are the ends of each branch is done in \cite{Over,WOZ} and reveals interesting behavior: the boundary develops corners at right angles. Recently, the $C^\infty$ regularity and the convexity of the patches near the trivial solutions have been investigated in \cite{HMV}. More recently, this result has been improved by Castro, C\'ordoba and G\'omez-Serrano in \cite{Cor22}, who showed the analyticity of the $V$-states close to the disc. We point out that similar research has been carried out in the past few years for more singular nonlinear transport equations arising in geophysical flows, such as the surface quasi-geostrophic equations, or the quasi-geostrophic shallow-water equations; see for instance \cite{Cor1,Cor22,H-H, Hana}. It should be noted that the angular velocities of the bifurcating $V$-states for \eqref{rotsssqq} are contained in the interval $]0,\frac12[$. However, it is not clear whether we can find a $V$-state when $\Omega$ does not lie in this range. In \cite{Fran}, Fraenkel proved, always in the simply-connected case, that the solutions associated to $\Omega=0$ are trivial and reduced to Rankine patches. This was established by using the moving plane method, which seems to be flexible and has been recently adapted in \cite{Hm} to $\Omega<0$, but with a convexity restriction. The case $\Omega=\frac12$ was also solved in that paper, using the maximum principle for harmonic functions.

Another related subject is to see whether or not a second bifurcation occurs at the branches discovered by Deem and Zabusky. This has been explored for the branch of the ellipses corresponding to $m=2$. Kamm gave in \cite{Kam} numerical evidence of the existence of some branches bifurcating from the ellipses, see also \cite{Saf}. In the paper \cite{Luz} by Luzzatto-Fegiz and Willimason, one can find more details about the diagram for the first bifurcations, and some illustrations of the limiting $V$-states. The proof of the existence and analyticity of the boundary has been recently investigated in \cite{Cor22, HM}. Another interesting topic which has been studied since the pioneering work of Love \cite{Love} is the linear and nonlinear stability of the $m$-folds. For the ellipses, we mention \cite{Guo,Tang}, and for the general case of the $m$-fold symmetric $V$-states, we refer to \cite{Landau,Wan}. For further numerical discussions, see also \cite{Cerr,DR,Mit}.

Recently in \cite{HMV2,H-F-M-V}, a special interest has been devoted to the study of doubly-connected $V$-states which are bounded patches and delimited by two disjoint Jordan curves. For example, an annulus is doubly-connected and, by rotation invariance, it is a stationary $V$-state. No other explicit doubly-connected $V$-state is known in the literature. In
\cite{HMV2}, a full characterization of the $V$-states (with nonzero magnitude in the interior domain) with at least one elliptical interface has been achieved, complementing the results of Flierl and Polvani \cite{Flierl}. As a by-product, it is shown that the domain between two
ellipses is a $V$-state only if it is an annulus. The existence of nonradial doubly-connected $V$-states has been achieved very recently in \cite{H-F-M-V} by using bifurcation theory. More precisely, we get the following result.
Let $0 <b <1$ and $m\geq3$, such that
$$
 1+{b}^{{m}}-\frac{1-{b}^2}{2} {m}<0.
$$
Then, there exists two curves of $m$-fold symmetric doubly-connected $V$-states bifurcating from the annulus $\{z\in \CC, b < |z| < 1
\}$ at each of the angular velocities
\begin{equation}\label{Nonl1}
\Omega_m^{\pm} = \frac{1-b^2}{4}\pm
\frac{1}{2m}\sqrt{\Big(\frac{m(1-b^2)}{2}-1\Big)^2
-b^{2m}}.
\end{equation}

The main topic of the current paper is to explore the existence of rotating patches \eqref{rotsssq} for Euler equations posed on the unit disc $\partial \mathbb{D}$. We shall focus on the simply-connected and doubly-connected cases, and study the influence of the rigid boundary on these structures.
Before stating our main results, we define the set
$
\mathbb{D}_b=\big\{z\in\CC,\, |z|<b\big\}.
$
Our first result dealing with the simply-connected $V$-states reads as follows.
\begin{theorem}\label{thmV1}
Let $b\in]0,1[$ and $m\in \NN^\star$. Then, there exists a family of $m$-fold symmetric $V$-states $(V_m)_{m\geq1}$ for the equation \eqref{rotsssq} bifurcating from the trivial solution $\omega_0=\chi_{\mathbb{D}_b}$ at the angular velocity
$$
\Omega_m\triangleq\frac{m-1+b^{2m}}{2m}\cdot
$$
\end{theorem}
The proof of this theorem is done in the spirit of the works \cite{B,H-F-M-V}, using the conformal mapping parametrization $\phi:\mathbb{T}\to \partial D$ of the $V$-states, combined with bifurcation theory. As we shall see later in \eqref{F}, the function $\phi$ satisfies the following nonlinear equation, for all $w\in \mathbb{T}$:
\begin{equation*}
 \textnormal{Im}\Bigg\{ \bigg[2\Omega\, \overline{\phi(w)}+\fint_{\mathbb{T}}\frac{\overline{\phi(w)}-\overline{\phi(\tau)}}{\phi(w)-\phi(\tau)}\phi'(\tau)d\tau-\fint_{\mathbb{T}}\frac{|\phi(\tau) |^2\phi'(\tau)}{1-\phi(w)\phi(\tau)}d\tau\bigg]\, {w}\, {{\phi'(w)}}\Bigg\}= 0.
\end{equation*}
Denote by $F(\Omega,\phi)$ the left term of the preceding equality. Then, the linearized operator around the trivial solution $\phi=b\operatorname{Id}$ can be explicitly computed, and is given by the following Fourier multiplier: For $\displaystyle{h(w)=\sum_{n\in\NN} a_n \overline{w}^n}$,
$$
\partial_\phi F(\Omega,b\, \operatorname{Id })h(w)=
 {b}\sum_{n\geq 1}n\Big(\frac{n-1+b^{2n}}{n}-2\Omega\Big) a_{n-1}e_n,\quad e_n=\frac{1}{2i}(\overline{w}^n-w^n).
$$
Therefore the {\it nonlinear eigenvalues} leading to nontrivial kernels with one dimension are explicitly described by the quantity $\Omega_m$ appearing in Theorem \ref{thmV1}. Later on, we check that all the assumptions of the Crandall-Rabinowitz theorem stated in Subsection \ref{Sub34} are satisfied, and our result follows easily. In Subsection \ref{numsc} we implement some numerical experiments concerning the limiting $V$-states. We observe two regimes depending on the size of $b$: $b$ small and $b$ close to $1$. In the first case, as it is expected, corners do appear as in the planar case. However, for $b$ close to $1$, the effect of the rigid boundary is not negligible. We observe that the limiting $V$-states are touching tangentially the unit circle, see Figure \ref{f:limiting0holem1234}.
Now we shall give some remarks.

\begin{remark}
For the Euler equations in the plane, there are no curves of $1$-fold $V$-states close to Rankine vortices. However, we deduce from our main theorem that this mode appears for spherical bounded domains. Its existence is the fruit of the interaction between the patch and the rigid boundary $\mathbb{T}$. Moreover, according to the numerical experiments, these $V$-states are not necessary centered at the origin and this fact is completely new. For the symmetry $m\geq2$, all the discovered $V$-states are necessarily centered at zero, because they have at least two axes of symmetry passing through zero.
\end{remark}
\begin{remark}\label{rmq78}
By a scaling argument, when the domain of the fluid is the ball $B(0,R)$, with $R>1$, then, from the preceding theorem, the bifurcation from the unit disc occurs at the angular velocities
$$
\Omega_{m,R}\triangleq\frac{m-1+R^{-2m}}{2m}\cdot
$$
Therefore we retrieve Burbea's result \cite{B} by letting $R$ tend to $+\infty$.
\end{remark}

\begin{remark}
From the numerical experiments done in \cite{H-F-M-V}, we note that, in the plane, the bifurcation is pitchfork and occurs to the left of $\Omega_m$. Furthermore, the branches of bifurcation are ``monotonic'' with respect to the angular velocity. In particular, this means that, for each value of $\Omega$, we have at most only one $V$-state with that angular velocity. This behavior is no longer true in the disc, as it will be discussed later in the numerical experiments, see \mbox{Figure $\ref{f:bifurcationm3b0_89}$.}
\end{remark}
\begin{remark}
Due to the boundary effects, the ellipses are no longer solutions for the rotating patch equations \eqref{rotsssq}. Whether or not explicit solutions can be found for this model  is an interesting problem. However, we believe that the conformal mapping of any non trivial $V$-state has necessary an infinite expansion. Note that  Burbea proved in \cite{Burb22} that in  the planar  case  when the conformal mapping associated to the $V$-state has a finite expansion then necessary it is of  an ellipse. His approach is based on Faber polynomials and this could  give an insight on  solving the same problem in the disc.
\end{remark}
The second part of this paper deals with the existence of doubly-connected $V$-states for the system \eqref{E1}, and governed by the system \eqref{rotsssq}. Note that the annular patches centered at zero, which are given by
$$
\mathbb{A}_{b_1,b_2}=\big\{z\in\CC; b_1<|z|<b_2\}, \quad\hbox{with}\quad b_1<b_2<1,
$$
are in fact stationary solutions. Our main task is to study the bifurcation of the $V$-states from these trivial solutions in the spirit of the recent works \cite{H-H-H,H-F-M-V}. We shall first start with studying the existence with the symmetry $m\geq2$, followed by the special case $m=1$.

\begin{theorem}\label{thmV02}
Let $0 <b_2<b_1 <1$, and set $b\triangleq\frac{b_2}{b_1}$. Let $m\geq2$, such that
$$
m>\frac{2+2b^m-(b_1^m+b_2^m)^2}{1-b^2}\cdot
$$
Then, there exist two curves of $m$-fold symmetric doubly-connected $V$-states bifurcating from the annulus $\mathbb{A}_{b_1,b_2}$ at the angular velocities
$$
\Omega_m^{\pm} = \frac{1-b^2}{4}+\frac{b_1^{2m}-b_2^{2m}}{4m}\pm
\frac{1}{2}\sqrt{\Delta_m}\,,$$
with
$$
\Delta_m=\Big(\frac{1-b^2}{2}-\frac{2-b_1^{2m}-b_2^{2m}}{2m}\Big)^2-b^{2m}\Big(\frac{1-b_1^{2m}}{m}\Big)^2.
$$
\end{theorem}
Before outlining the ideas of the proof, a few remarks are necessary.
\begin{remark}
As it was discussed in Remark $\ref{rmq78}$, one can use a scaling argument and obtain the result previously established in \cite{H-F-M-V} for the planar case. Indeed, when the domain of the fluid is the ball $B(0,R)$, with $R>1$, then the bifurcation from the annulus $\mathbb{A}_{b,1}$ amounts to make the changes $b_1=\frac{1}{R}$ and $b_2=\frac{b}{R}$ in Theorem $\ref{thmV02}$. Thus, by letting $R$ tend to infinity, we get exactly the nonlinear eigenvalues of the Euler equations in the plane \eqref{Nonl1}.\end{remark}
\begin{remark}
Unlike in the plane, where the frequency $m$ is assumed to be larger than $3$, we can reach $m=2$ in the case of the disc. This can be checked for $b_2$ small with respect to $b_1$. This illustrates once again the fruitful interaction between the rigid boundary and the $V$-states.
\end{remark}
Now, we shall sketch the proof of Theorem \ref{thmV02}, which follows the same lines of \cite{H-F-M-V}, and stems from bifurcation theory. The first step is to write down the analytical equations of the boundaries of the $V$-states. This can be done for example through the conformal parametrization of the domains $D_1$ and $D_2$, which are close to the discs $b_1\mathbb{D}$ and $b_2\mathbb{D}$, respectively. Set $\phi_j:\mathbb{D}^c\to D_j^c$, the conformal mappings which have the following expansions:
$$
\forall \,|w|\geq1,\quad \phi_1(w)=b_1\,w+\sum_{n\in\NN} \frac{a_{1,n}}{w^n},\quad \phi_2(w)= b_2\, w+\sum_{n\in\NN} \frac{a_{2,n}}{w^n}.
$$
In addition, we assume that the Fourier coefficients are real, which means that we are looking only for $V$-states that are symmetric with respect to the real axis. As we shall see later in Section \ref{bound548}, the conformal mappings are subject to two coupled nonlinear equations defined as follows: for $j\in\{1,2\}$ and $w\in \mathbb{T}$,
\begin{equation*}
F_j(\lambda, \phi_1,\phi_2)(w) \triangleq \textnormal{Im}\bigg\{ \Big((1-\lambda)\overline{\phi_j(w)}+I(\phi_j(w))-J(\phi_j(w))\Big)\, {w}\, {{\phi_j'(w)}}\bigg\}=0,
\end{equation*}
with
$$
I(z)=\fint_{\mathbb{T}}\frac{\overline{z}-\overline{\phi_1(\xi)}}{z-\phi_1(\xi)}\phi_1^\prime(\xi)d\xi-\fint_{\mathbb{T}}\frac{\overline{z}-\overline{\phi_2(\xi)}}{z-\phi_2(\xi)}\phi_2^\prime(\xi)d\xi,
$$
and
$$
J(z)=\fint_{\mathbb{T}}\frac{|\phi_1(\xi) |^2}{1-z\phi_1(\xi)}\phi_1^\prime(\xi)d\xi-\fint_{\mathbb{T}}\frac{|\phi_2(\xi) |^2}{1-z\phi_2(\xi)}\phi_2^\prime(\xi)d\xi,\quad \lambda\triangleq1-2\Omega.
$$
In order to apply bifurcation theory, we should understand the structure of the linearized operator around the trivial solution $(\phi_1,\phi_2)=(b_1\operatorname{Id},b_2\,\operatorname{Id})$ corresponding to the annulus with radii $b_1$ and $b_2$, and identify the range of $\Omega $ where this operator has a one-dimensional kernel. The computations of the linear operator $DF(\Omega,b_1\operatorname{Id},b_2\,\operatorname{Id})$ with $F=(F_1,F_2)$ in terms of the Fourier coefficients are fairly long, and we find that it acts as Fourier multiplier matrix. More precisely, for
$$
{ h_1(w)=\sum_{n\geq1}\frac{a_{1,n}}{w^n}}, \quad \displaystyle h_2(w)=\sum_{n\geq1}\frac{a_{2,n}}{w^n},
$$
we obtain the formula
\begin{equation*}
DF(\lambda,b_1\,\operatorname{Id},b_2\,\operatorname{Id})\big(h_1,h_2\big)=\sum_{n\geq 1}M_{n}(\lambda)
\begin{pmatrix}
a_{1,n-1} \\
a_{2,n-1}
\end{pmatrix}e_{n}, \quad e_n(w)\triangleq \frac{1}{2i}(\overline{w}^n-w^n),
\end{equation*}
where the matrix $M_n$ is given by
\begin{equation*}
M_{n}(\lambda)=\begin{pmatrix}
b_1\Big[n\lambda-1+b_1^{2n}-n\Big(\dfrac{b_2}{b_1}\Big)^2\Big] & b_2\Big[\Big(\dfrac{b_2}{b_1}\Big)^{n}-(b_1b_2)^{n}\Big] \\
 -b_1\Big[\Big(\dfrac{b_2}{b_1}\Big)^{n}-(b_1b_2)^{n}\Big] & b_2\Big[n\lambda-n+1-{b_2^{2n}}\Big]
\end{pmatrix}.
\end{equation*}
Therefore, the values of $\Omega$ associated to nontrivial kernels are the solutions of a second-degree polynomial in $\lambda$,
\begin{equation}\label{det227}
P_n(\lambda)\triangleq\hbox{det }M_n(\lambda)=0.
\end{equation}
The polynomial $P_n$ has real roots when the discriminant $\Delta_n(\alpha,b)$ introduced in Theorem \ref{thmV02} is positive. The calculation of the dimension of the kernel is rather more complicated than the cases raised before in the references \cite{B,H-F-M-V}. The matter reduces to count, for a given $\lambda$, the following discrete set:
$$
\big\{n\geq2, P_n(\lambda)=0\big\}.
$$
Note that, in \cite{B,H-F-M-V}, this set has only one element and, therefore, the kernel is one-dimensional. This follows from the monotonicity of the roots of $P_n$ with respect to $n$. In the current situation, we get similar results, but with a more refined analysis, see Lemma \ref{lem-mono} and Proposition \ref{asympbeh}.

Now, we shall move to the existence of $1$-fold symmetries, which is completely absent in the plane. The study in the general case is slightly more subtle, and we have only carried out partial results, so some other cases are left open and deserve to be explored. Before stating our main result, we need to make some preparation. As we shall see in Section \ref{first11}, the equation $P_1(\lambda)=0$ admits exactly two solutions given by
$$
\lambda_1^-=(b_2/b_1)^2\quad\textnormal{or}\quad\lambda_1^+=1+b_2^2-b_1^2.
$$
Similarly to the planar case \cite{H-F-M-V}, there is no hope to bifurcate from the first eigenvalue $\lambda_1^-$, because the range of the linearized operator around the trivial solution has an infinite co-dimension and, thus, C-R theorem stated in Subsection \ref{Sub34} is useless. However, for the second eigenvalue $\lambda_1^+$, the range is at most of co-dimension two and, in order to bifurcate, we should avoid a special set of $b_1$ and $b_2$ that we shall describe now. Fix $b_1$ in $]0,1[$, and set
$$
\mathcal{E}_{b_1}\triangleq \Big\{ b_2\in]0,b_1[; \exists n\geq2\,\, s.t.\,\, P_n\Big({1+b_2^2-b_1^2}\Big)=0\Big\}.
$$
 where $P_n$ is defined in \eqref{det227}.
As we shall see in Proposition \ref{prop-one}, this set is countable and composed of a strictly increasing sequence $(x_m)_{m\geq1}$ converging to $b_1$. Now, we state our result.
\begin{theorem}\label{thm-1fold}
Given $b_1\in]0,1[$, then, for any $b_2 \notin\mathcal{E}_{b_1}$, there exists a curve of nontrivial $1$-fold doubly
connected $V$-states bifurcating from the annulus $\mathbb{A}_{b_1,b_2}$ at the angular velocity
$$
\Omega_1 =\frac{b_1^2-b_2^2}{2}.
$$
\end{theorem}
The proof is done in the spirit of Theorem \ref{thmV02}. When $b_2 \notin\mathcal{E}_{b_1}$, then all the conditions of C-R theorem are satisfied. However, when $b_2 \in\mathcal{E}_{b_1}$, then the range of the linearized operator has co-dimension two. Whether or not the bifurcation occurs in this special case is an interesting problem which is left open here.

{\bf {Notation.}}
We need to collect some useful notation that will be frequently used along this paper. We shall use the symbol $\triangleq$ to define an object. Crandall-Rabinowitz theorem is sometimes shorten to CR theorem.
The unit disc is denoted by $\mathbb{D}$, and its boundary, the unit circle, by $\mathbb{T}$.
For a given continuous complex function $f:\mathbb{T}\to \CC$, we define its mean value by
$$
\fint_{\mathbb{T}} f(\tau)d\tau\triangleq \frac{1}{2i\pi}\int_{\mathbb{T}} f(\tau)d\tau,
$$
where $d\tau$ stands for complex integration.

Let $X$ and $Y$ be two normed spaces. We denote by $\mathcal{L}(X,Y)$ the space of all continuous linear maps $T: X\to Y$ endowed with its usual strong topology.
We denote by $\operatorname{Ker} T$ and $R(T)$ the null space and the range of $T$, respectively. Finally, if $F$ is a subspace of $Y$, then $Y/ F$ denotes the quotient space.

\section{Preliminaries and background}
 In this introductory section we shall collect some basic facts on H\"older spaces, bifurcation theory and see how to use the conformal mappings in the equations of the $V$-states.
\subsection{Function spaces}
In this paper as well as in the preceding ones \cite{HMV, H-F-M-V} we find more convenient   to think of  $2\pi$-periodic function $g:\RR\to\CC$ as a function of the complex variable $w=e^{i\theta}$. To be more precise, let  $f:\mathbb{T}\to \RR^2$ be a smooth  function, then it  can be assimilated to  a $2\pi-$periodic function $g:\RR\to\RR^2$ via the relation
$$
f(w)=g(\eta),\quad w=e^{i\eta}.
$$
By Fourier expansion there exist complex numbers $(c_n)_{n\in\mathbb{Z}}$ such that
$$
f(w)=\sum_{n\in \mathbb{Z}} c_n w^n
$$
and the differentiation with respect to $w$ is understood in the complex sense.
Now we shall introduce  H\"older spaces  on the unit circle $\mathbb{T}$.
\begin{definition}
Let 
$0<\gamma<1$. We denote by $C^\gamma(\mathbb{T}) $  the space of continuous functions $f$ such that
$$
\Vert f\Vert_{C^\gamma(\mathbb{T})}\triangleq\Vert f\Vert_{L^\infty(\mathbb{T})}+\sup_{\tau\neq w\in \mathbb{T}}\frac{\vert f(\tau)-f(w)\vert}{\vert \tau-w\vert^\alpha}<\infty.
$$
For any nonnegative integer $n$, the space $C^{n+\gamma}(\mathbb{T})$ stands for the set of functions $f$ of class $C^n$ whose $n-$th order derivatives are H\"older continuous  with exponent $\gamma$. It is equipped with the usual  norm,
$$
\Vert f\Vert_{C^{n+\gamma}(\mathbb{T})}\triangleq  \sum_{k=0}^n\Big\Vert \frac{d^k f}{d^kw}\Big\Vert_{L^\infty(\mathbb{T})}+\Big\Vert \frac{d^n f}{d^nw}\Big\Vert_{C^\gamma(\mathbb{T})}.
$$ 
\end{definition}
Recall that the Lipschitz semi-norm is defined by,
$$
\|f\|_{\textnormal{Lip}(\mathbb{T})}=\sup_{\tau\neq w\in\mathbb{T}}\frac{|f(\tau)-f(w)|}{|\tau-w|}\cdot
$$
Now we list some classical properties that will be useful later.
\begin{enumerate}
\item For $n\in \mathbb{N}, \gamma\in ]0,1[$ the space $C^{n+\gamma}(\mathbb{T})$ is an algebra.
\item For $K\in L^1(\mathbb{T})$ and $f\in C^{n+\gamma}(\mathbb{T})$ we have the convolution law,
$$
\|K*f\|_{C^{n+\gamma}(\mathbb{T})}\le \|K\|_{L^1(\mathbb{T})}\|f\|_{C^{n+\gamma}(\mathbb{T})}.
$$
\end{enumerate}

 \subsection{Elements of the bifurcation theory}\label{Sub34}
 We  shall now recall an important theorem of the bifurcation theory which plays a central role in the proofs of our main results. This theorem was established by Crandall and Rabinowitz in \cite{CR} and sometimes will be referred to as   C-R theorem for abbreviation.  Consider a continuous function $F:\RR\times X\to Y$  with $X$ and $Y$ being   two Banach spaces. Assume that $F(\lambda, 0)=0$ for any $\lambda$ belonging to  non trivial interval $I.$ C-R theorem gives sufficient conditions for the existence of branches of non trivial solutions to  the equation  $
 F(\lambda,x)=0
 $ bifurcating at some point $(\lambda_0, 0)$. For more general results  we refer the reader to  the book of Kielh\"{o}fer \cite{Kil}. 
 
\begin{theorem}\label{C-R} Let $X, Y$ be two Banach spaces, $V$ a neighborhood of $0$ in $X$ and let 
$
F : \RR \times V \to Y.
$ Set $ \mathcal{L}_0\triangleq \partial_x F(0,0)$ then  the following  properties are satisfied.
\begin{enumerate}
\item $F (\lambda, 0) = 0$ for any $\lambda\in \RR$.
\item The partial derivatives $F_\lambda$, $F_x$ and $F_{\lambda x}$ exist and are continuous.
\item The spaces $N(\mathcal{L}_0)$ and $Y/R(\mathcal{L}_0)$ are one-dimensional. 
\item {\it Transversality assumption}: $\partial_\lambda\partial_xF(0, 0)x_0 \not\in R(\mathcal{L}_0)$, where
$$
N(\mathcal{L}_0) = span\{x_0\}.
$$
\end{enumerate}
If $Z$ is any complement of $N(\mathcal{L}_0)$ in $X$, then there is a neighborhood $U$ of $(0,0)$ in $\RR \times X$, an interval $(-a,a)$, and continuous functions $\varphi: (-a,a) \to \RR$, $\psi: (-a,a) \to Z$ such that $\varphi(0) = 0$, $\psi(0) = 0$ and
$$
F^{-1}(0)\cap U=\Big\{\big(\varphi(\xi), \xi x_0+\xi\psi(\xi)\big)\,;\,\vert \xi\vert<a\Big\}\cup\Big\{(\lambda,0)\,;\, (\lambda,0)\in U\Big\}.
$$
\end{theorem}

Before proceeding further with the consideration of the $V$-states,  we shall recall Riemann mapping theorem which  is one of the most important results in complex analysis. To restate this result we need to recall the definition of {\it simply-connected} domains. Let $\widehat{\mathbb{C}}\triangleq \mathbb{C}\cup\{\infty\}$ denote the Riemann sphere. We  say that a domain $\Omega\subset \widehat{\mathbb{C}}$ is {\it simply-connected} if the set $ \widehat{\mathbb{C}}\backslash \Omega$ is connected.
Let $\mathbb{D}$ denote the unit open ball and  $\Omega\subset \CC$ be a simply-connected bounded domain. Then according to Riemann Mapping Theorem there is a unique bi-holomorphic map called also conformal map,  $\Phi: \CC\backslash\overline{\mathbb{D}}\to  \CC\backslash\overline{\Omega}$ taking the form
$$
\Phi(z)= az+\sum_{n\in\NN} \frac{a_n}{z^n} \quad \textnormal{with}\quad a>0.
$$
 In this theorem the regularity of the boundary has no effect regarding the existence of the conformal mapping  but  it contributes in the boundary behavior of the conformal mapping, see for instance \cite{P,WS}. Here, we shall recall the following result. 
\vspace{0,2cm}


{\it  Kellogg-Warschawski's theorem.} It can be found in  \cite{WS} or in  \cite[Theorem 3.6]{P}.  It asserts that if   the conformal map
$\Phi:\CC\backslash\overline{\mathbb{D}}\to  \CC\backslash\overline{\Omega}$ has a continuous
extension to $\CC\backslash\mathbb{D}$  which is of \mbox{class
$C^{n+\beta},$} with
$n\in \NN$  and $ 0<\beta<1$,  then the
boundary $\Phi(\mathbb{T})$ is  of \mbox{class $C^{n+\beta}.$}

\subsection{Boundary equations}\label{bound135}
Our next task is to write down the equations of the rotating patches using the conformal parametrization. First  recall that the vorticity  $\omega=\partial_1v_2-\partial_2 v_1$  satisfies the transport equation,\begin{equation*}
\partial_{t}\omega+v\cdot\nabla\omega=0
\end{equation*}
and the associated velocity is related to the vorticity through  the stream function $\Psi$  as follows, 
\begin{equation*}
v=2i\,\partial_{\overline z}\Psi.
\end{equation*}
with  \begin{equation*}
\Psi(z)=\frac{1}{4\pi}\int_{\mathbb{D}}\log\Big|\frac{z-\xi}{1-z\overline{\xi}}\Big|^2\omega(\xi)dA(\xi).
\end{equation*}
When the vorticity is a patch of the form $\omega=\chi_{D}$ with $D$ a bounded domain strictly contained in $\mathbb{D}$, then
 \begin{equation*}
\Psi(z)=\frac{1}{4\pi}\int_{{D}}\log\Big|\frac{z-\xi}{1-z\overline{\xi}}\Big|^2dA(\xi).
\end{equation*}
For  a complex function $\varphi: \CC\to\CC$   of class $C^1$ in the Euclidean  variables (as a function of $\RR^2$) we define 
$$
{\partial_z \varphi}=\frac{1}{2}\Big(\frac{\partial \varphi}{\partial x}-i\frac{\partial \varphi}{\partial y}\Big)\quad\textnormal{and}\quad{\partial_{\overline{z}}}\varphi=\frac{1}{2}\Big(\frac{\partial \varphi}{\partial x}+i\frac{\partial \varphi}{\partial y}\Big).
$$

As we have previously seen in the Introduction a rotating patch is a special solution of the vorticity equation \eqref{vorticity} with   initial data $\omega_0=\chi_D$ and such that $$
\omega(t)=\chi_{D_t},\quad \hbox{with}\quad D_t=e^{it\Omega}D.
$$
In this definition and for the simplicity  we have only considered patches rotating around zero. 
According to  \cite{B, HMV, H-F-M-V}  the boundary equation of the rotating patches  is given by 
\begin{equation}\label{rotsq1}
 \textnormal{Re}\Big\{ \Big(\Omega\overline{z}-2\partial_z\Psi\Big)\, z^\prime\Big\}=0\quad z\in\Gamma \triangleq \partial D,
\end{equation}
where $z'$ denotes a tangent vector to the boundary at the point $z$. We point out  that the existence of rigid boundary  does not alter this equation which in fact was established in the planar case. The purpose now is to transform the equation \eqref{rotsq1} into an equation involving  only on the boundary $\partial D$ of the $V$-state.
To do so,  we need  to write $\partial_z\Psi$ as an integral on the boundary $\partial D$ based on the use of   Cauchy-Pompeiu's formula : Consider a finitely-connected domain $D$ bounded by finitely many smooth Jordan curves and let $\Gamma$ be the boundary $\partial D$ endowed with the positive orientation, then 
\begin{equation}\label{CP}
\forall \, z\in \CC,\quad \fint_\Gamma\frac{\varphi(z)-\varphi(\xi)}{z-\xi}{d\xi}=-\frac1\pi\int_D\partial_{\overline{\xi}}\varphi(\xi)\frac{dA(\xi)}{z-\xi}.
\end{equation}
Differentiating \eqref{Integ71} with respect to the variable ${z}$ yields
\begin{equation}\label{Integ011}
\partial_{{z}}\Psi(z)=\frac{1}{4\pi}\int_{{D}}\frac{\overline{\xi}}{1-z\overline{\xi}}\,dA(\xi)+\frac{1}{4\pi}\int_{{D}}\frac{1}{z-\xi}dA(\xi).
\end{equation}
Applying  Cauchy-Pompeiu's formula with $\varphi(z)=\overline{z}$ we find
\begin{equation*}
\frac{1}{\pi}\int_{{D}}\frac{1}{z-\xi}dA(\xi)=-\fint_\Gamma\frac{\overline{z}-\overline{\xi}}{z-\xi}d\xi, \quad\forall z\in \overline{D}.
\end{equation*}
Using the change of variable $\xi\to\overline{\xi}$ which keeps invariant the  Lebesgue measure we get
$$
\frac{1}{\pi}\int_{{D}}\frac{\overline{\xi}}{1-z\overline{\xi}}\,dA(\xi)=
\frac{1}{\pi z}\int_{\widetilde{D}}\frac{{\xi}}{1/z-{\xi}}\,dA(\xi).
$$
with  $\widetilde{D}$ being  the image of $D$ by the complex conjugate.
A second application of  the Cauchy-Pompeiu formula using  that $\frac1z\notin \mathbb{D}$ for $z\in D$ yields
\begin{equation*}
\frac{1}{\pi z}\int_{\widetilde{D}}\frac{{\xi}}{1/z-{\xi}}\,dA(\xi)=\fint_{\widetilde{\Gamma}}\frac{|\xi |^2}{1-z\xi}d\xi, \quad\forall z\in \overline{D},\quad \widetilde{\Gamma}=\partial \widetilde{D}.
\end{equation*}
Using more again  the change of variable $\xi\to\overline{\xi}$ which reverses the orientation we get
\begin{equation*}
\fint_{\widetilde{\Gamma}}\frac{|\xi |^2}{1-z\xi}d\xi=-\fint_{{\Gamma}}\frac{|\xi |^2}{1-z\overline \xi}d\overline\xi, \quad\forall z\in \overline{D}.
\end{equation*}

 Therefore we obtain 
 \begin{equation}\label{Integ1}
4\partial_{z}\Psi(z)=-\fint_\Gamma\frac{|\xi |^2}{1-z\overline\xi}d\overline\xi-\fint_\Gamma\frac{\overline{z}-\overline{\xi}}{z-\xi}d\xi.
\end{equation}
Inserting the last identity in \eqref{rotsq1} we get an  equation making  appeal only to the boundary
\begin{equation*}
\,\textnormal{Re}\Bigg\{ \Big(2\Omega \overline{z}+\fint_\Gamma\frac{\overline{z}-\overline{\xi}}{z-\xi}d\xi+\fint_\Gamma\frac{|\xi |^2}{1-z\,\overline \xi}d\overline\xi\Big)\, z^\prime\Bigg\}=0, \quad \forall\, z\in \Gamma.
\end{equation*}
It is more convenient in the formulas  to replace in the preceding equation the angular velocity $\Omega$ by the parameter $\lambda = 1-2\Omega$ leading to the $V$-states equation
\begin{equation}\label{rotsqzw}
\,\textnormal{Re}\Bigg\{ \Big((1-\lambda) \overline{z}+\fint_\Gamma\frac{\overline{z}-\overline{\xi}}{z-\xi}d\xi+\fint_\Gamma\frac{|\xi |^2}{1-z\,\overline\xi}d\overline\xi\Big)\, z^\prime\Bigg\}=0, \quad \forall\, z\in \Gamma.
\end{equation}
It is worthy to point out that the equation \eqref{rotsqzw} characterizes $V$-states among domains with
$C^1$ boundary, regardless of the number of boundary components. If
the domain is simply-connected then there is only one boundary component
and so only one equation. However, if the domain is doubly-connected then
\eqref{rotsqzw} gives rise to  two coupled equations, one for each boundary
component. We note that all the $V$-states that we shall consider admit at least one axis of symmetry passing through zero and without loss of generality it can be supposed to be the real axis. This implies that the boundary $\partial D$ is invariant by the reflection \mbox{symmetry   $\xi\to\overline{\xi}$.} Therefore  using in the last integral term of the equation \eqref{rotsqzw} this   change of variables   which reverses  the orientation  we obtain
\begin{equation}\label{rotsq}
\,\textnormal{Re}\Big\{ \Big((1-\lambda) \overline{z}+\fint_\Gamma\frac{\overline{z}-\overline{\xi}}{z-\xi}d\xi-\fint_\Gamma\frac{|\xi |^2}{1-z\,\xi}d\xi\Big)\, z^\prime\Big\}=0, \quad \forall\, z\in \Gamma.
\end{equation}
To end this section, we mention that in the general framework the dynamics of any vortex patch can be described by its Lagrangian parametrization $\gamma_t:\mathbb{T}\to \partial D_t\triangleq\Gamma_t$ as follows
$$
\partial_t \gamma_t =v(t,\gamma_t).
$$
Since $\Psi$ is real-valued function then
$$
\partial_{\overline z}\Psi=\overline{\partial_z\Psi},
$$
which   implies according to \eqref{Integ1}
\begin{eqnarray*}
v(t,z)&=&2 i\,\partial_{\overline{z}}\Psi(t,z)\\
&=&
-\frac{1}{4\pi}\int_{\Gamma_t}\log|z-\xi|^2 d\xi+\frac{1}{4\pi}\int_{\Gamma_t}\frac{|\xi |^2}{1-\overline z\,\xi}d\xi.
\end{eqnarray*}
Consequently, we find that  the Lagrangian parametrization   satisfies the nonlinear ODE,
\begin{eqnarray}\label{vort-dy}
\partial_t\gamma_t=-\frac{1}{4\pi}\int_{\Gamma_t}\log|\gamma_t-\xi|^2 d\xi+\frac{1}{4\pi}\int_{\Gamma_t}\frac{|\xi |^2}{1-\overline \gamma_t\,\xi}d\xi.
\end{eqnarray}
\\

The ultimate goal of this section is to relate the $V$-states described above  to stationary solutions for Euler equations when the  the rigid boundary rotates at some specific angular velocity. To do so, suppose that the disc $\mathbb{D}$ rotates with a constant angular velocity $\Omega$ then the equations \eqref{E1} written in the frame of the rotating disc   take the form :

$$
\partial_t u+u\cdot \nabla  u-\Omega y^\perp\cdot\nabla u+\Omega u^\perp+\nabla q=0
$$
with 
$$
y=e^{-it\Omega} x,\quad v(t,x)=e^{-it\Omega} u(t,y)\quad\hbox{and}\quad q(t,y)=p(t,x).
$$
For more details about the derivation of this equation we refer the reader for instance  to the paper \cite{To}. Here the variable in the rotating frame is denoted by $y$. Applying the {\it curl} operator to the equation of $u$ we find that the vorticity of $u$ which still denoted by $\omega$ is governed by the transport equation
$$
\partial_t\omega+(u-\Omega y^\perp)\cdot\nabla \omega=0.
$$
Consequently any stationary solution in the patch form is actually a $V$-state rotating with the angular velocity $\Omega$. Relating  this observation to Theorem \ref{thmV1} and  Theorem \ref{thmV02} we deduce that by rotating the disc at some suitable angular velocities creates  stationary patches with $m$-fold symmetry.

\section{Simply-connected $V$-states}
In this section we shall gather all the pieces needed for the proof of  Theorem \ref{thmV1}. The strategy is analogous to \cite{B, HMV, H-F-M-V}. It consists first  in writing down the $V$-states equation through the conformal parametrization and second  to apply C-R theorem. As it can be noted  from Theorem \ref{thmV1} the result is local meaning that we are looking  for $V$-states which are smooth and being small perturbation of the Rankine patch $\chi_{\mathbb{D}_b}$ with $\mathbb{D}_b=b\mathbb{D}.$  We also assume  that the patch is symmetric with respect to the real axis and this fact  has been crucially used to derive the equation \eqref{rotsq}. Note that  as $D\Subset \mathbb{D}$, then  the exterior conformal mapping
$\phi: \mathbb{D}^c\to D^c$  has the expansion
\begin{equation*}
\phi(w)=bw+\sum_{n\geq 0}\frac{b_n}{w^n},\quad b_n\in \RR
\end{equation*}
and satisifies  $0<b<1.$
This latter  fact  follows from Schwarz lemma. Indeed, let 
$$
\psi(z)\triangleq\frac{1}{\phi(1/z)},
$$
then $\psi:\mathbb{D}\to \widehat D$ is conformal, with $\widehat D$ the image of $D$  by the map $z\mapsto \frac1z$. Clearly $\mathbb{D}\subset \widehat D$ and therefore the restriction $\psi^{-1}:\mathbb{D}\to \mathbb{D}$ is well-defined, holomorphic  and satisfies $\psi(0)=0$. From Schwarz lemma we deduce that $ |(\psi^{-1})^\prime(0)|<1$, otherwise $D$ will coincide with $\mathbb{D}$. It suffices now to use that $(\psi^{-1})^\prime(0)=b$.\\
Now we shall transform the equation \eqref{rotsq} into an equation on the unit circle $\mathbb{T}.$ For this purpose we make the change of variables: $z=\phi(w)$ and $\xi=\phi(\tau)$.  Note that  for $w\in \mathbb{T}$ a tangent vector at the point $z=\phi(w)$ is given by
$$
z'=i{w}\, {{\phi'(w)}}
$$
and thus  the  equation \eqref{rotsq} becomes
\begin{equation}\label{F}
 \,\textnormal{Im}\Bigg\{ \bigg[(1-\lambda)\overline{\phi(w)}+\fint_{\mathbb{T}}\frac{\overline{\phi(w)}-\overline{\phi(\tau)}}{\phi(w)-\phi(\tau)}\phi'(\tau)d\tau-\fint_{\mathbb{T}}\frac{|\phi(\tau) |^2\phi'(\tau)}{1-\phi(w)\phi(\tau)}d\tau\bigg]\, {w}\, {{\phi'(w)}}\Bigg\}= 0,
\end{equation}
Set  $\phi\triangleq b\,\hbox{Id}+f$ then the foregoing functional can be split  into three parts :
\begin{eqnarray*}
F_1(f)(w)&\triangleq &\textnormal{Im}\Big\{ \overline{\phi(w)}\, {w}\, {{\phi'(w)}}\Big\},
\end{eqnarray*}
\begin{eqnarray*}
F_2(f)(w)&\triangleq &\textnormal{Im}\bigg\{\fint_{\mathbb{T}}\frac{\overline{\phi(w)}-\overline{\phi(\tau)}}{\phi(w)-\phi(\tau)}\phi'(\tau)d\tau\, {w}\, {{\phi'(w)}}\bigg\},
\end{eqnarray*}
\begin{eqnarray}\label{F3}
F_3(f)(w)&\triangleq &\textnormal{Im}\bigg\{ \fint_{\mathbb{T}}\frac{|\phi(\tau) |^2\phi'(\tau)}{1-\phi(w)\phi(\tau)}d\tau\, {w}\, {{\phi'(w)}}\bigg\}
\end{eqnarray}
and consequently the equation \eqref{F} becomes 
\begin{equation}\label{zqqq1}
F(\lambda,f)=0,\quad\hbox{with}\quad F(\lambda,f)\triangleq (1-\lambda)F_1(f)+F_2(f)-F_3(f).
\end{equation}
Observe that we can decompose $F$ into two parts $F(\lambda,f)=G(\lambda,f)-F_3(f)$ where  $G(\lambda,f)$ is the functional appearing in the flat space $\RR^2$ and the new term $F_3$ describes the interaction between the patch and the rigid  boundary $\mathbb{T}.$
Now it is easy from the complex formulation to check  that the disc $\mathbb{D}_b$ is a rotating patch for any $\Omega\in\RR$. Indeed, as the disc is a trivial solution for the full space $\RR^2$ then  $G(\lambda,0)=0$. Moreover,
\begin{eqnarray*}
F_3(0)(w)&\triangleq &\textnormal{Im}\bigg\{ {b^4w}\fint_{\mathbb{T}}\frac{d\tau}{1-b^2w\tau}\bigg\}=0
\end{eqnarray*} 
because the integrand is analytic in the   open disc $\frac{1}{b^2}\mathbb{D}$ and therefore we apply residue theorem.
\subsection{Regularity of the functional $F$}

This section is devoted to the study  of the regularity assumptions stated in  C-R Theorem for the functional $F$ introduced in \eqref{zqqq1}. The application of this theorem requires at this stage of the presentation to fix  the function spaces $X$ and $Y$. We should look for Banach spaces $X$ and $Y$ of H\"older type in the spirit of the \mbox{papers \cite{HMV, H-F-M-V}} and they  are given by,
$$
X=\Big\{f\in C^{1+\alpha}(\mathbb{T}),\, f(w)=\sum_{n\geq 0}a_n\overline{w}^n, a_n\in \RR,\, w\in\mathbb{T}\Big\}
$$
and
$$
Y=\Big\{g\in C^{\alpha}(\mathbb{T}),\, g(w)=\sum_{n\geq 1}b_ne_n\ b_n\in \RR,\, w\in\mathbb{T}\Big\}, \quad e_n\triangleq\frac{1}{2i}\big(\overline{w}^n-w^n\big),
$$
with $\alpha\in ]0,1[$. For  $r\in (0,1)$ we denote by $B_r$  the open ball of $X$ with center $0$ and radius $r$,
$$
B_r=\Big\{f\in X,\quad \Vert f\Vert_{C^{1+\alpha}}\leq r\Big\}.
$$
It is straightforward that for any $f\in B_r$ the function $w\mapsto \phi(w)=bw+f(w)$ is conformal on  $\CC\backslash\overline{\mathbb{D}}$ provided that $r<b$. Moreover according to Kellog-Warshawski result \cite{WS}, the boundary of $\phi(\CC\backslash\overline{\mathbb{D}})$  is a Jordan curve of \mbox{class $C^{1+\alpha} $.} 
We propose to  prove  the following result concerning the regularity of $F.$
\begin{proposition}\label{reg}
Let $b\in]0,1[$ and $0<r <\min(b,1-b)$, then the following holds true. 
\begin{enumerate}
\item $F: \RR\times B_r\to Y$ is  $C^1$ $($it is in fact $C^\infty)$.
\item The partial derivative  $\partial_\lambda \partial_fF: \mathbb{R}\times B_r\to \mathcal{L}(X, Y)$  exists and is continuous $($it is in fact $C^\infty).$
 \end{enumerate}
\end{proposition}
\begin{proof} 
{{(\hbox{i})}} We shall only sketch the proof because most of the details are done in the papers \cite{HMV, H-F-M-V}. First recall from \eqref{zqqq1} the decomposition
$$
F(\lambda,f)=(1-\lambda)F_1(f)+F_2(f)-F_3(f).
$$
The part $(1-\lambda)F_1(f)+F_2(f)$ coincides with  the nonlinear functional appearing in the plane and its regularity was studied in \cite{HMV, H-F-M-V}. Therefore  it remains to check the regularity assumptions for the term $F_3$ given in \eqref{F3}. Since $C^{\alpha}(\mathbb{T})$ is an algebra then it suffices to prove that the mapping $F_4: \phi \in b\hbox{Id}+B_r\to C^\alpha$ defined  by 
\begin{equation}\label{trs1}
F_4(\phi(w))=\fint_{\mathbb{T}}\frac{|\phi(\tau) |^2\phi'(\tau)}{1-\phi(w)\phi(\tau)}d\tau
\end{equation}
is $C^1$ and admits real Fourier coefficients.
Observe that this functional is well-defined and is given by the series expansion
$$
F_4(\phi(w))=\sum_{n\in \NN}\phi^n(w)\fint_{\mathbb{T}}\phi^n(\tau){|\phi(\tau) |^2\phi'(\tau)}d\tau.
$$
This sum is defined pointwisely because $\|\phi\|_{L^\infty}\le b+r<1$. This series converges absolutely in $C^\alpha(\mathbb{T})$. To get this we use the law product which can be proved by induction
$$
\|\phi^n\|_{C^\alpha}\le n\|\phi\|_{L^\infty}^{n-1}\|\phi\|_{C^\alpha}
$$
and therefore we obtain
\begin{eqnarray*}
\|F_4(\phi)\|_{C^\alpha}&\le &\sum_{n\in \NN}n\|\phi\|_{L^\infty}^{n-1}\|\phi\|_{C^\alpha}\Big\vert\fint_{\mathbb{T}}\phi^n(\tau){|\phi(\tau) |^2\phi'(\tau)}d\tau\Big\vert\\
&\le &\|\phi^\prime\|_{L^\infty}\|\phi\|_{C^\alpha}\sum_{n\in \NN}n\|\phi\|_{L^\infty}^{2n+1}\\
&\le &\|\phi^\prime\|_{L^\infty}\|\phi\|_{C^\alpha}\sum_{n\in \NN}n(b+r)^{2n+1}<\infty.
\end{eqnarray*}
From the completeness of $C^\alpha(\mathbb{T})$ we obtain that $F_4(\phi)$ belongs to this space. Again from the series expansion we can check that $\phi\mapsto F_4(\phi)$ is not only $C^1$ but also $C^\infty$. To end the proof we need to check that all the Fourier coefficients of $F_4(\phi)$ are real and this fact is equivalent to show that 
$$
\overline{F_4(\phi(w))}=F_4(\phi(\overline{w})),\quad \forall w\in \mathbb{T}.
$$
As $\overline{\phi(w)}=\phi(\overline{w})$ and $\overline{\phi^\prime(w)}=\phi^\prime(\overline{w})$ then we may  write successively 
\begin{eqnarray*}
\overline{F_4(\phi(w))}&=&-\fint_{\mathbb{T}}\frac{|\phi(\overline{\tau}) |^2\phi'(\overline{\tau})}{1-\phi(\overline{w})\phi(\overline{\tau})}d\overline{\tau}\\
&=&\fint_{\mathbb{T}}\frac{|\phi(\tau) |^2\phi'(\tau)}{1-\phi(w)\phi(\tau)}d\tau
\end{eqnarray*}
where in the last equality we have used the change of variables $\tau\mapsto \overline{\tau}$.
\\

{(\hbox{ii})} Following the arguments developed in   \cite{HMV, H-F-M-V} we get what is expected formally, that is
\begin{eqnarray*}
\partial_\lambda\partial_fF(\lambda, f)h&=&-\partial_fF_1(f)\\
&=&\textnormal{Im}\Big\{ \overline{\phi(w)}\, {w}\, {{h'(w)}}+\overline{h(w)}\, {w}\, {{\phi'(w)}}\Big\}.
\end{eqnarray*}
From which  we deduce that $\partial_\lambda\partial_fF(\lambda, f)\in \mathcal{L}(X,Y)$ and the mapping $f\mapsto  \partial_\lambda\partial_fF(\lambda, f)$ is in fact  $C^\infty$ which is clearly more better than the statement of the proposition.
\end{proof}

\subsection{Spectral study}
This part is crucial for implementing C-R theorem. We shall in particular compute the linearized operator $\partial_f F(\lambda,0)$ around the trivial solution  and look for the values of $\lambda $ associated to non trivial kernel. For these values of $\lambda$ we shall see that  the linearized operator has a one-dimensional kernel and is in fact of  Fredholm type with zero index. Before giving the main result of this paragraph we recall  the notation
$
e_n=\frac{1}{2i}(\overline{w}^n-w^n).
$

\begin{proposition}\label{propspec01}
Let $h\in X$ taking the form $\displaystyle{h(w)=\sum_{n\geq0}\frac{a_n}{{w}^n}}.$ Then the following holds true.
\begin{enumerate}
\item Structure of $\partial_f F(\lambda,0) :$
$$
\partial_f F(\lambda,0)h(w)=
 {b}\sum_{n\geq 1}n\Big(\lambda-\frac{1-b^{2n}}{n}\Big) a_{n-1}e_n. 
$$
\item The kernel of $\partial_f F(\lambda,0)$ is non trivial if and only if there exists $m\in \NN^\star$ such that
$$
\lambda=\lambda_m\triangleq\frac{1-b^{2m}}{m}, m\in \NN^\star
$$
and in this case the kernel  is one-dimensional generated by $v_m(w)= \overline{w}^{m-1}$.
\item The range of $\partial_f F(\lambda_m,0)$ is of co-dimension one 

\item Transversality condition $:$ for $m\in \NN^\star$
$$
\partial_\lambda\partial_f F(\lambda_m,0) v_m\notin R\partial_f F(\lambda_m,0) 
$$

\end{enumerate}
\end{proposition}
\begin{proof} 
(\hbox{i}) The computations of the terms $\partial_f F_i(\lambda,0)h$ were almost done in \cite{H-F-M-V} and we shall only give  some  details. By straightforward computations we obtain
\begin{eqnarray}\label{pf1}
\partial_f F_1(0,0)h(w)&=& \textnormal{Im}\Big\{ b\overline{h(w)} {w}+bh'(w)\Big\}\notag\\ &=&  b\textnormal{Im}\Big\{\sum_{n\geq 0}a_nw^{n+1}-\sum_{n\geq 1}na_n\overline{w}^{n+1}\Big\}\notag\\ &=& -\frac{b}{2i}\sum_{n\geq 0}(n+1)a_n(\overline{w}^{n+1}-w^{n+1})\notag\\
&=&-{b}\sum_{n\geq 0}(n+1)a_{n} \, e_{n+1}.
\end{eqnarray}
Concerning  $\partial_f F_2(0,0)$ one may write
\begin{eqnarray*}
\partial_fF_2(0,0)h(w)&=&\textnormal{Im}\Big\{ b w\fint_{\mathbb{T}}\frac{\overline{h(\tau)}-\overline{h(w)}}{\tau-w}d\tau+b \fint_{\mathbb{T}}\frac{{h(\tau)}-{h(w)}}{\tau-w}\overline{\tau}d\tau\\
&-&b\fint_{\mathbb{T}}h^\prime(\tau)\overline{\tau}d\tau-bh^\prime(w)\Big\}.
\end{eqnarray*}
Therefore using residue theorem at infinity we get
\begin{eqnarray*}
\partial_fF_2(0,0)h(w)&=&\textnormal{Im}\Big\{ b w\fint_{\mathbb{T}}\frac{\overline{h(\tau)}-\overline{h(w)}}{\tau-w}d\tau-bh^\prime(w)\Big\}\\
&=&-\textnormal{Im}\big\{ bh^\prime(w)\big\}
\end{eqnarray*}
and where we have used in the last line the fact
\begin{eqnarray*}
\fint_{\mathbb{T}}\frac{\overline{h(\tau)}-\overline{h(w)}}{\tau-w}d\tau&=&\sum_{n\in \NN} a_n\fint_{\mathbb{T}}\frac{w^n-\tau^n}{\tau-w}d\tau\\
&=&0.\end{eqnarray*}
Consequently we obtain
\begin{equation}\label{pf3}
\partial_f F_2(0,0)h(w)={b}\sum_{n\geq 1}na_ne_{n+1}.
\end{equation}


As  to the third term $\partial_f F_3(0,0)h$ we get by plain computations,
\begin{eqnarray}\label{pf2}
\nonumber\partial_f F_3(0,0)h(w)&=&\, \textnormal{Im}\,\bigg\{{b^3w}\fint_{\mathbb{T}}\frac{d\tau}{1-b^2w\tau}h'(w)+{b^3w}\fint_{\mathbb{T}}\frac{h'(\tau)d\tau}{1-b^2w\tau}+{2b^3w}\fint_{\mathbb{T}}\frac{\textnormal{Re}\{h(\tau)\overline{\tau}\}}{1-b^2w\tau}d\tau\\
&+&{b^5w}\fint_{\mathbb{T}}\frac{wh(\tau)+\tau h(w)}{(1-b^2w\tau)^2}d\tau\bigg\}\notag\\ &\triangleq&\,  \textnormal{Im}\,\Big\{ I_1(w)+I_2(w)+I_3(w)+I_4(w)\Big\}.
\end{eqnarray}
By invoking once again residue theorem we get
\begin{eqnarray}\label{id1}
I_1(w)=0.
\end{eqnarray}
To compute the second term $I_2(w)$ we  use the Taylor series of $\frac{1}{1-\zeta}$ leading to
\begin{eqnarray*}
I_2(w)&=&{b^3w}\fint_{\mathbb{T}}\frac{h'(\tau)d\tau}{1-b^2w\tau}\notag\\ &=&\sum_{n\geq 0}b^{2n+3}w^{n+1}\fint_{\mathbb{T}}\tau^nh'(\tau)d\tau.
\notag\end{eqnarray*}
From the Fourier expansions of $h$  we infer that
$$
\fint_{\mathbb{T}}\tau^nh'(\tau)d\tau=-n\,a_n
$$
which implies that
\begin{eqnarray}\label{id2}
I_2(w)  &=&-\sum_{n\geq 1}na_nb^{2n+3}w^{n+1}.
\end{eqnarray}
In regard to the third term $I_3(w)$ it may be written in the form
\begin{eqnarray*}
I_3(w)={b^3w}\fint_{\mathbb{T}}\frac{\tau\overline{h(\tau)}}{1-b^2w\tau}d\tau+{b^3w}\fint_{\mathbb{T}}\frac{\overline{\tau}h(\tau)}{1-b^2w\tau}d\tau.
\end{eqnarray*}
The first integral term is zero due to the fact that the integrand is analytic in the open unit disc and continuous up to the boundary. Therefore we get similarly to $I_2(w)$
\begin{eqnarray*}
I_3(w)&=&{b^3w}\fint_{\mathbb{T}}\frac{\overline{\tau}h(\tau)}{1-b^2w\tau}d\tau\notag\\ &=& \sum_{n\geq 0}b^{2n+3}w^{n+1}\fint_{\mathbb{T}}\tau^{n-1}h(\tau)d\tau.
\end{eqnarray*}
Remark that
$$
\fint_{\mathbb{T}}\tau^{n-1}h(\tau)d\tau=a_n
$$
which implies in turn that 
\begin{equation}\label{id3}
I_3(w)=\sum_{n\geq 0}a_nb^{2n+3}w^{n+1}.
\end{equation}
Now we come  back to the last term  $I_4(w)$ and one may write using  again residue theorem
\begin{eqnarray*}
I_4(w)&=&{b^5w^2}\fint_{\mathbb{T}}\frac{h(\tau)d\tau}{(1-b^2w\tau)^2}+{b^5wh(w)}\fint_{\mathbb{T}}\frac{\tau d\tau}{(1-b^2w\tau)^2}\notag\\ &=& {b^5w^2}\fint_{\mathbb{T}}\frac{h(\tau)d\tau}{(1-b^2w\tau)^2}+0.
\end{eqnarray*}
Using Taylor expansion
\begin{equation}\label{Tay11}
\frac{1}{(1-\zeta)^2}=\sum_{n\geq 1}n\zeta^{n-1},\quad |\zeta|<1.
\end{equation}
we deduce that
\begin{eqnarray}\label{id4}
I_4(w)
&=&  \sum_{n\geq 1}nb^{2n+3}w^{n+1}\fint_{\mathbb{T}}\tau^{n-1}h(\tau)d\tau\notag\\ &=& \sum_{n\geq 1} na_{n}b^{2n+3}w^{n+1}.
\end{eqnarray}

Inserting the identities  \eqref{id1},\eqref{id2},\eqref{id3} and \eqref{id4} into \eqref{pf2} we find
\begin{eqnarray}\label{Eqswq}
\nonumber\partial_f F_3(0,0)h(w)&=&
\textnormal{Im}\Big\{\sum_{n\geq 0}a_nb^{2n+3}w^{n+1}\Big\}\\ &=& -\sum_{n\geq 0}a_nb^{2n+3}e_{n+1}. 
\end{eqnarray}
Hence by plugging  \eqref{pf1}, \eqref{pf3}, \eqref{Eqswq} into \eqref{zqqq1} we obtain
\begin{eqnarray}\label{ff}
\nonumber\partial_f F(\lambda,0)h(w)&=&
b\sum_{n\geq 0}(n+1)\Big(\lambda-\frac{1-b^{2n+2}}{n+1}\Big)a_ne_{n+1}\\
&=&b\sum_{n\geq 1} n\Big(\lambda-\frac{1-b^{2n}}{n}\Big)a_{n-1}e_{n}. 
\end{eqnarray}
This achieves the proof of the first part $(i)$.
\\

${(\hbox{ii})}$
From \eqref{ff} we  immediately deduce that the kernel of $\partial_f F(\lambda,0)$ is non trivial if and only if there exists $m\geq 1$ such that
$$
\lambda=\lambda_{m}\triangleq \frac{1-b^{2m}}{m}\cdot
$$
We shall prove that the sequence $n \mapsto \lambda_n$ is strictly decreasing from which  we conclude immediately  that the kernel is one-dimensional. Assume that for two  integers  $n> m\geq 1$ one has
$$
\frac{1-b^{2m}}{m}=\frac{1-b^{2n}}{n}\cdot
$$
This implies that
$$
\frac{1-b^{2n}}{1- b^{2m}}=\frac{n}{m}\cdot
$$

Set $\alpha=\frac{n}{m}$ and $ x=b^{2m}$ then the preceding equality becomes
$$
f(x)\triangleq\frac{1- x^\alpha}{1-x}=\alpha.
$$
If we  prove that this equation has no solution $x\in ]0,1[$ for any $\alpha>1$ then the result follows without difficulty. To do so, we get after  differentiating  $f$
$$
f^\prime(x)=\frac{(\alpha-1) x^\alpha-\alpha x^{\alpha-1}+1}{(1-x)^2}\triangleq\frac{g(x)}{(1-x)^2}\cdot
$$
Now we note that
$$
g^\prime(x)=\alpha(\alpha-1)x^{\alpha-2}(x-1)<0.
$$
As $g(1)=0$ then we deduce
$$
g(x)>0, \quad \forall x\in ]0,1[.
$$
Thus $f$ is strictly increasing. Furthermore
$$
\lim_{x\to 1} f(x)=\alpha.
$$
This implies that
$$
\forall x\in ]0,1[,\quad f(x)<\alpha.
$$
Therefore we get the strict monotonicity of the "eigenvalues" and consequently 
 the kernel of $\partial_f F(\lambda_m,0)$ is one-dimensional vector space generated by the function $v_m(w) = \overline{w}^{m-1}$.
 \\
 
 ${(\hbox{iii})}$ We shall prove that the range of $\partial_f F(\lambda_m,0)$ is described by
 $$
 R\partial_f F(\lambda_m,0)=\Big\{g\in Y; \quad g(w)=\sum_{n\geq1\\
 \atop n\neq m} b_n e_n\Big\}\triangleq\mathcal{Z}.
 $$
 Combining  Proposition \ref{reg} and Proposition \ref{propspec01}-$i)$ we  conclude that the range is contained in the right space. So what is left is to prove the converse. Let $g\in \mathcal{Z}$, we will solve in $X$ the equation
 $$
 \partial_f F(\lambda_m,0)h=g, \quad h=\sum_{n\geq0} a_n \overline{w}^n
 $$
 By virtue of \eqref{ff} this equation is equivalent to
 $$
 a_{n-1}=\frac{b_n}{bn(\lambda_m-\lambda_n)},\quad n\geq1, n\neq m.
 $$
 Thus the problem reduces to showing that
 $$
 h: w\mapsto\sum_{n\geq1\\
 \atop n\neq m}\frac{b_n}{bn(\lambda_m-\lambda_n)} \overline{w}^{n-1}\in C^{1+\alpha}(\mathbb{T}).
 $$
 Observe that
 $$
 \inf_{n\neq m}|\lambda_n-\lambda_m|\triangleq c_0>0
 $$
 and thus we deduce by Cauchy-Schwarz
 \begin{eqnarray*}
 \|h\|_{L^\infty}&\le&\frac1b\sum_{n\geq1,
 \atop n\neq m}\frac{|b_n|}{n|\lambda_m-\lambda_n|}\\
 &\le&\frac{1}{c_0 b}\sum_{n\geq1,
 \atop n\neq m}\frac{|b_n|}{n}\\
 &\lesssim& \|g\|_{L^2}\lesssim\|g\|_{C^\alpha}.
 \end{eqnarray*}
 To achieve the proof we shall check that $h^\prime \in C^\alpha(\mathbb{T})$ or equivalently $(\overline{w} h)^\prime \in C^\alpha(\mathbb{T})$. It is obvious that
  \begin{eqnarray*}
 (\overline{w} h(w))^\prime&=&-\sum_{n\geq1
 \atop n\neq m}\frac{b_n}{b(\lambda_m-\lambda_n)}\overline{w}^{n+1}\\
 &=&-\frac{1}{b\lambda_m}\sum_{n\geq1
 \atop n\neq m}{b_n} \overline{w}^{n+1}+\frac{1}{b\lambda_m}\sum_{n\geq1
 \atop n\neq m}\frac{\lambda_n}{\lambda_n-\lambda_m} b_n\overline{w}^{n+1}.
 \end{eqnarray*}
We shall  write the preceding expression with  Szeg\"o projection 
$$\Pi:\sum_{n\in \mathbb{Z}}a_n w^n\mapsto \sum_{n\in -\mathbb{N}}a_n w^n,
$$
\begin{eqnarray*}
 (\overline{w} h(w))^\prime&=&-\frac{\overline{w}}{2i\,b\lambda_m}\Pi g(w)+\frac{\overline{w}}{2i\,b\lambda_m}(K\star\Pi g)(w), \end{eqnarray*}
 with
 $$
 K(w)\triangleq\sum_{n\geq1
 \atop n\neq m}\frac{\lambda_n}{\lambda_n-\lambda_m} \overline{w}^{n}.
 $$
 Notice that
 $$
\frac{\lambda_n}{|\lambda_n-\lambda_m|}\le c_0^{-1} \frac{1}{n}
 $$
 and therefore $K\in L^2(\mathbb{T})$ which implies  in particular that $K\in L^1(\mathbb{T})$. Now to complete the proof of  $(\overline{w} h)^\prime \in C^\alpha(\mathbb{T})$ it suffices to use the continuity of Szeg\"o projection on $ C^\alpha(\mathbb{T})$ combined with $L^1\star C^\alpha(\mathbb{T})\subset C^\alpha(\mathbb{T})$.
 \\
 
 $({\hbox{iv}})$ To check the  transversality assumption, we differentiate  \eqref{ff} with respect to $\lambda$
 $$
 \partial_\lambda\partial_f F(\lambda_m,0) h=b\sum_{n\geq1} n a_{n-1} e_n
 $$
 and therefore
 $$
  \partial_\lambda\partial_f F(\lambda_m,0)v_m=b m e_m\notin R(\partial_f F(\lambda_m,0)).
 $$
 This completes the proof of the proposition.
\end{proof}
\subsection{Proof of Theorem \ref{thmV1}}
According to   Proposition \ref{propspec1} and Proposition \ref{reg} all the assumptions of Crandall-Rabinowitz theorem are satisfied and therefore we conclude for each $m\geq 1$ the existence of only one  non trivial curve bifurcating from the trivial one at the angular velocity 
$$\Omega_m=\frac{1-\lambda_m}{2}=\frac{m-1+b^{2m}}{2m}\cdot$$
To complete the proof it remains to check the $m$-fold symmetry of the $V$-states. This can be done by including the required symmetry in the function spaces. More precisely, instead of dealing with $X$ and $Y$ we should work with the spaces
$$
X_m=\Big\{f\in C^{1+\alpha}(\mathbb{T}),\, f(w)=\sum_{n=1}^{\infty}a_{n}\overline{w}^{nm-1},\, a_n\in \mathbb{R}  \Big\}
$$
and
$$
Y_m=\Big\{g\in C^{\alpha}(\mathbb{T}),\, g(w)=\sum_{n\geq 1}b_{n} e_{nm},\, b_n\in\RR\Big\}, e_n=\frac{1}{2i}(\overline{w}^n-w^n).
$$
The conformal mapping describing the $V$-state takes the form
$$
\phi(w)=bw+\sum_{n=1}^{\infty}a_{n}\overline{w}^{nm-1}
$$
and the $m$-fold symmetry of the $V$-state  means that
$$
\phi(e^{2i\pi/m}w)=e^{2i\pi/m}\phi(w),\,\forall\, w\in \mathbb{T}.
$$
The ball $B_r$ is changed to  $B_r^m=\big\{f\in X_m, \|f\|_{C^{1+\alpha}}<r\big\}$. Then Proposition \ref{reg} holds true according to this adaptation and the only point that one must check is the stability of the spaces, that is, for $f\in B_r^m$ we have $F(\lambda, f)\in Y_m.$
This result was checked   in the paper \cite{H-F-M-V}  for the terms $F_1$ and $F_2$ and it remains to check that $F_3(f)$ belongs to $Y_m$. Recall that 
$$
F_3(f(w))=\hbox{Im}\big\{ F_4(\phi(w)) w\phi^\prime(w)\big\},\quad  \phi(w)= bw+ f(w)
$$
and where $F_4$ is defined in \eqref{trs1}. By change of variables and using the symmetry of $\phi$ we get
\begin{eqnarray*}
F_4(\phi(e^{i\frac{2\pi}{m}}w)\big)&=&\fint_{\mathbb{T}}\frac{|\phi(\xi) |^2\phi'(\xi)}{1-\phi(e^{i\frac{2\pi}{m}}w)\phi(\xi)}d\xi\\
&=&{e^{-i\frac{2\pi}{m}} }\fint_{\mathbb{T}}\frac{|\phi(e^{-i\frac{2\pi}{m}}\zeta) |^2\phi'(e^{-i\frac{2\pi}{m}}\zeta)}{1-\phi(e^{i\frac{2\pi}{m}}w)\phi(e^{-i\frac{2\pi}{m}}\zeta)}d\zeta\\
&=&{e^{-i\frac{2\pi}{m}}}\fint_{\mathbb{T}}\frac{|\phi(\tau) |^2\phi'(\tau)}{1-\phi(w)\phi(\tau)}d\tau\\
&=&e^{-i\frac{2\pi}{m}} F_4(\phi(w)).
\end{eqnarray*}
Consequently we obtain
$$
F_3(f(e^{i\frac{2\pi}{m}}w))=F_3(f(w))
$$
and this shows the stability result.
\section{Doubly-connected $V$-states}
In this section we shall establish all the ingredients required for the proofs of Theorem \ref{thmV02} and Theorem \ref{thm-1fold} and this will be carried out in several steps. First we shall write the equations governing the  doubly-connected $V$-states which are described by two coupled nonlinear equations. Second we  briefly discuss the regularity of the functionals and compute the linearized operator around the trivial solution. The delicate  part to which we will pay  careful attention is the computation of the kernel dimension. This will be implemented through the study of the  monotonicity  of the nonlinear eigenvalues. As we shall see the fact that we have multiple parameters introduces much more complications to this study compared to the result of \cite{H-F-M-V}.  Finally, we shall achieve  the proof of Theorem \ref{thmV02} in  Subsection \ref{TT23}.
\subsection{Boundary equations}\label{bound548}
Let $D$ be a doubly-connected domain of the form $D=D_1\backslash D_2$ with $D_2\subset D_1$ being  two simply-connected domains. Denote by  $\Gamma_j$ the boundary of the domain  $D_j$.  In this case the $V$-states equation \eqref{rotsq} reduces to two coupled equations, one for each  boundary component $\Gamma_j$. More precisely,
\begin{equation}\label{rotsqz1}
\,\textnormal{Re}\Big\{ \Big((1-\lambda) \overline{z}+ I(z)-J(z)\Big)\, z^\prime\Big\}=0, \quad \forall\, z\in \Gamma_1\cup \Gamma_2,
\end{equation}
with
$$
I(z)=\fint_{\Gamma_1}\frac{\overline{z}-\overline{\xi}}{z-\xi}d\xi-\fint_{\Gamma_2}\frac{\overline{z}-\overline{\xi}}{z-\xi}d\xi
$$
and
$$
J(z)=\fint_{\Gamma_1}\frac{|\xi |^2}{1-z\xi}d\xi-\fint_{\Gamma_2}\frac{|\xi |^2}{1-z\xi}d\xi.
$$
As for the simply-connected case we prefer using the conformal parametrization of the boundaries.  Let  $\phi_j : \mathbb{D}^c \to D_j^c$ satisfying
\begin{equation*}
\phi_j(w)=b_jw+\sum_{n\geq 0}\frac{a_{j,n}}{w^n}.
\end{equation*}
with $0<b_j<1$, $j=1,2$ and $b_2<b_1$. We assume moreover that all the Fourier coefficients are real because we shall look for $V$-states which are symmetric with respect to the real axis. Then  by change of variables we obtain
$$
I(z)=\fint_{\mathbb{T}}\frac{\overline{z}-\overline{\phi_1(\xi)}}{z-\phi_1(\xi)}\phi_1^\prime(\xi)d\xi-\fint_{\mathbb{T}}\frac{\overline{z}-\overline{\phi_2(\xi)}}{z-\phi_2(\xi)}\phi_2^\prime(\xi)d\xi
$$
and
$$
J(z)=\fint_{\mathbb{T}}\frac{|\phi_1(\xi) |^2}{1-z\phi_1(\xi)}\phi_1^\prime(\xi)d\xi-\fint_{\mathbb{T}}\frac{|\phi_2(\xi) |^2}{1-z\phi_2(\xi)}\phi_2^\prime(\xi)d\xi.
$$
Setting $\phi_j=b_j \hbox{Id}+f_j$, the  equation \eqref{rotsqz1} becomes
\begin{equation*}
\forall w\in\mathbb{T}\quad G_j(\lambda,f_1,f_2)(w)=0;\quad j=1,2, 
\end{equation*}
where
\begin{alignat*}{3}
G_{j}(\lambda,f_1,f_2)(w) &\triangleq \textnormal{Im}\bigg\{ \Big((1-\lambda)\overline{\phi_j(w)}+I(\phi_j(w))-J(\phi_j(w))\Big)\, {w}\, {{\phi_j'(w)}}\bigg\}.
\end{alignat*}
Note that one can easily check that
$$
G(\lambda,0,0)=0,\quad \forall \lambda\in \RR.
$$
This is coherent with the fact that  the annulus is a stationary solution and therefore it rotates with any angular velocity since the shape is rotational invariant. 
\subsection{Regularity of the functional $G$}

In this short subsection  we shall quickly state the regularity result  of the  functional $G\triangleq(G_1,G_2)$ needed in CR Theorem.   Following  the simply-connected case the spaces $X$ and $Y$ involved in the bifurcation will be chosen in a similar way : Set
$$
X=\Big\{f\in (C^{1+\alpha}(\mathbb{T}))^2,\, f(w)=\sum_{n\geq 0}A_n\overline{w}^n, A_n\in \RR^2,\, w\in\mathbb{T}\Big\}
$$
and
$$
Y=\Big\{g\in (C^{\alpha}(\mathbb{T}))^2,\, g(w)=\sum_{n\geq 1}B_ne_n\ B_n\in \RR^2,\, w\in\mathbb{T}\Big\}, \quad e_n\triangleq\frac{1}{2i}\big(\overline{w}^n-w^n\big),
$$
with $\alpha\in ]0,1[$. For  $r\in (0,1)$ we denote by $B_r$  the open ball of $X$ with center $0$ and radius $r$,
$$
B_r=\Big\{f\in X,\quad \Vert f\Vert_{C^{1+\alpha}}\leq r\Big\}.
$$
Similarly to Proposition \ref{reg} one can establish the regularity assumptions needed for C-R Theorem. Compared to the simply-connected case, the only terms that one should  care about are those describing the interaction between the boundaries of the patches which  are supposed to be disjoint.  Therefore  the involved kernels are sufficiently  smooth and actually they do not carry significant difficulties in their treatment. For this reason we prefer skip the details and restrict  ourselves to the following statement. 
\begin{proposition}\label{regg1}
Let $b\in]0,1[$ and $0<r <\min(b,1-b)$, then the following holds true. 
\begin{enumerate}
\item $G: \RR\times B_r\to Y$ is  $C^1$ $($it is in fact $C^\infty)$.
\item The partial derivative  $\partial_\lambda \partial_fG: \mathbb{R}\times B_r\to \mathcal{L}(X, Y)$  exists and is continuous $($it is in fact $C^\infty).$
 \end{enumerate}
\end{proposition}

\subsection{Structure of the linearized operator}
In this section we shall compute the linearized operator $\partial_f G(\lambda,0)$ around the annulus $\mathbb{A}_{b_1,b_2}$ of radii $b_1$ and $b_2$. The study of the eigenvalues is postponed to the next subsections.  From the regularity assumptions of $G$ we assert  that  Fr\'echet derivative and G\^ateaux derivatives coincide and
$$
DG(\lambda,0,0)(h_1,h_2)=\frac{d}{dt}G(\lambda, t\,h_1,t\, h_2)_{|t=0}.
$$
Note that  $DG(\lambda,0,0)$ is nothing but  the partial derivative  $\partial_f G(\lambda,0, 0)$.  Our main result reads as follows.
\begin{proposition}\label{propspec1}
Let $h=(h_1,h_2)\in X$ taking the form $\displaystyle{h_j(w)=\sum_{n\geq0}\frac{a_{j,n}}{{w}^n}}.$ Then,
 \begin{eqnarray*}
DG(\lambda,0,0)\big(h_1,h_2\big)=\sum_{n\geq 1}M_{n}(\lambda)\left( \begin{array}{c}
a_{1,n-1} \\
a_{2,n-1}
\end{array} \right)e_{n},
\end{eqnarray*}
where the matrix $M_n$ is given   by
\begin{equation*}
M_{n}(\lambda)=\begin{pmatrix}
b_1\Big[n\lambda-1+b_1^{2n}-n\Big(\dfrac{b_2}{b_1}\Big)^2\Big] & b_2\Big[\Big(\dfrac{b_2}{b_1}\Big)^{n}-(b_1b_2)^{n}\Big] \\
 -b_1\Big[\Big(\dfrac{b_2}{b_1}\Big)^{n}-(b_1b_2)^{n}\Big] & b_2\Big[n\lambda-n+1-{b_2^{2n}}\Big]
\end{pmatrix}, \quad e_n(w)= \frac{1}{2i}(\overline{w}^n-w^n).
\end{equation*}
\end{proposition}
\begin{proof}
 Since $G=(G_1,G_2)$ then for  a given couple of functions $(h_1,h_2) \in X$ we have
\begin{eqnarray*}
DG(\lambda,0,0)(h_1,h_2)&=& \begin{pmatrix} \partial_{f_1}G_1(\lambda,0,0)h_1+\partial_{f_2}
G_1(\lambda,0,0)h_2 \\*[4pt] \partial_{f_1} G_2(\lambda,0,0)h_1+\partial_{f_2}
G_2(\lambda,0,0)h_2
\end{pmatrix}.
\end{eqnarray*}
We shall split  $G_j$ into three terms,
$$
G_j(\lambda,f_1,f_2)=G_j^1(\lambda,f_j)+G_j^2(f_1,f_2)+G_j^3(f_1,f_2),
$$
where
\begin{alignat*}{2}
G_j^1(\lambda,f_j)(w)&\triangleq \textnormal{Im}\Bigg\{ \bigg[(1-\lambda)\overline{\phi_j(w)}&& +(-1)^{j+1}\fint_{\mathbb{T}}\frac{\overline{\phi_j(w)}-\overline{\phi_j(\tau)}}{\phi_j(w)-\phi_j(\tau)}\phi'_j(\tau)d\tau\notag\\  & &&+(-1)^j\fint_{\mathbb{T}}\frac{|\phi_j(\tau) |^2\phi_j'(\tau)}{1-\phi_j(w)\phi_j(\tau)}d\tau\bigg]\, {w}\, {{\phi_j'(w)}}\Bigg\},
\end{alignat*}

\begin{equation*}
G_j^2(f_1,f_2)\triangleq(-1)^j\textnormal{Im}\bigg\{\fint_{\mathbb{T}}\frac{\overline{\phi_j(w)}-\overline{\phi_i(\tau)}}{\phi_j(w)-\phi_i(\tau)}\phi'_i(\tau)d\tau\, {w}\, {{\phi_j'(w)}}\bigg\},\quad i\neq j,
\end{equation*}
and
\begin{equation*}
G_j^3(f_1,f_2)\triangleq (-1)^{j+1}\textnormal{Im}\bigg\{\fint_{\mathbb{T}}\frac{|\phi_i(\tau) |^2\phi_i'(\tau)}{1-\phi_j(w)\phi_i(\tau)}d\tau\, {w}\, {{\phi_j'(w)}}\bigg\},\quad i\neq j,
\end{equation*}
with  $\phi_j=b_j \hbox{Id}+f_j$; $j=1,2$.
\\

$\bullet$ {\it Computation of  $\partial_{f_j}G_j^1(\lambda,0,0)h_j$.} First observe that
\begin{alignat*}{2}
G_1^1(\lambda,f_1)(w)&= \textnormal{Im}\Bigg\{ \bigg[(1-\lambda)\overline{\phi_1(w)}&& +\fint_{\mathbb{T}}\frac{\overline{\phi_1(w)}-\overline{\phi_1(\tau)}}{\phi_1(w)-\phi_1(\tau)}\phi'_1(\tau)d\tau\notag\\  & &&-\fint_{\mathbb{T}}\frac{|\phi_1(\tau) |^2\phi_1'(\tau)}{1-\phi_1(w)\phi_1(\tau)}d\tau\bigg]\, {w}\, {{\phi_1'(w)}}\Bigg\}.
\end{alignat*}
This functional  is exactly the defining function in the simply-connected case and thus using merely  \eqref{ff} we get
\begin{eqnarray}\label{lh1}
\partial_{f_1} G_1^1(\lambda,0)h_1&=&
{b_1}\sum_{n\geq 0}\Big(\lambda(n+1)-{1+b_1^{2n+2}}\Big)\,a_{1,n}\, e_{n+1}. 
\end{eqnarray}
In regard to $G_2^1(\lambda,f_2)$ we get from the definition
\begin{alignat*}{2}
G_2^1(\lambda,f_2)(w)&= \textnormal{Im}\Bigg\{ \bigg[(1-\lambda)\overline{\phi_2(w)}&& -\fint_{\mathbb{T}}\frac{\overline{\phi_2(w)}-\overline{\phi_2(\tau)}}{\phi_2(w)-\phi_2(\tau)}\phi'_2(\tau)d\tau\notag\\  & &&+\fint_{\mathbb{T}}\frac{|\phi_2(\tau) |^2\phi_2'(\tau)}{1-\phi_2(w)\phi_2(\tau)}d\tau\bigg]\, {w}\, {{\phi_2'(w)}}\Bigg\}.
\end{alignat*}
It is easy to check the algebraic relation $G_2^1(\lambda,f_2)=-G_1^1(2-\lambda,f_2)$ and  thus we get by applying \eqref{lh1},
\begin{eqnarray}\label{lh2}
\partial_{f_2} G_2^1(\lambda,0)h_2&=&
{b_2}\sum_{n\geq 0}\Big(\lambda(n+1)-{2n-1-b_2^{2n+2}}\Big)a_{2,n}\,e_{n+1}. 
\end{eqnarray}
\\

$\bullet$ {\it Computation of  $\partial_{f_j}G_j^2(\lambda,0,0)h_j$.}
This quantity is  given by 
$$
\partial_{f_j} G_j^2(0,0)h_j=(-1)^j\frac{d}{dt}_{|t=0}\textnormal{Im}\bigg\{b_i w\fint_{\mathbb{T}}\frac{{b_j \overline{w}-b_i\overline{\tau}+ t \overline{h_j(w)}}}{b_j w-b_i \tau+th_j(w)}d\tau\, \, {{(b_j+ t h_j'(w))}}\bigg\}.
$$
Straightforward computations yield 
\begin{alignat*}{2}
\partial_{f_j} G_j^2(0,0)h_j&=(-1)^j b_i\,\textnormal{Im}\bigg\{  h'_j(w)w\fint_{\mathbb{T}}\frac{b_j\overline{w}-b_i\overline{\tau}}{b_jw-b_i\tau}d\tau&&+b_jw\overline{h_j(w)}\fint_{\mathbb{T}}\frac{d\tau}{b_jw-b_i\tau}\notag\\ & &&-b_jwh_j(w)\fint_{\mathbb{T}}\frac{b_j\overline{w}-b_i\overline{\tau}}{(b_jw-b_i\tau)^2}d\tau\bigg\}.
\end{alignat*}
According to  residue theorem  we get
\begin{equation*}
\int_{\mathbb{T}}\frac{d\tau}{b_1w-b_2\tau}=0\quad \textnormal{and}\quad \int_{\mathbb{T}}\frac{d\tau}{(b_1w-b_2\tau)^2}=0,\quad\forall\, w\in\mathbb{T} ,
\end{equation*}
and therefore 
\begin{eqnarray}\label{lL1}
\partial_{f_1} G_1^2(0,0)h_1(w)&=&- b_2^2\,\textnormal{Im}\bigg\{ -\fint_{\mathbb{T}}\frac{w\,h'_1(w)}{b_1w-b_2\tau}\frac{d\tau}{\tau} +b_1\fint_{\mathbb{T}}\frac{wh_1(w)}{(b_1w-b_2\tau)^2}\frac{d\tau}{\tau}\bigg\}\notag\\ 
&=&-\, b_2^2\,\textnormal{Im}\Big\{  -\frac{1}{b_1}h'_1(w)+\frac{1}{b_1}\overline{w}h_1(w)\Big\}\notag\\ 
&=&-\,\frac{b_2^2}{b_1}\sum_{n\geq 0}(n+1)a_{1,n}e_{n+1}.
\end{eqnarray}
Now using the vanishing integrals
\begin{equation*}
\int_{\mathbb{T}}\frac{\overline{\tau}\,d\tau}{b_2w-b_1\tau}=0,\quad \int_{\mathbb{T}}\frac{\overline{\tau}\,d\tau}{(b_2w-b_1\tau)^2}=0\quad  \textnormal{and}\quad\int_{\mathbb{T}}\frac{d\tau}{(b_2w-b_1\tau)^2}\,d\tau=0
\end{equation*}
we may obtain
\begin{eqnarray}\label{lL2}
\nonumber\partial_{f_2} G_2^2(0,0)h_2(w)&=&
\, b_1\,\textnormal{Im}\bigg\{b_2\, h'_2(w)\fint_{\mathbb{T}}\frac{d\tau}{b_2w-b_1\tau}+b_2w\overline{h_2(w)}\fint_{\mathbb{T}}\frac{d\tau}{b_2w-b_1\tau}\bigg\}\\
\nonumber&=&
  b_1\textnormal{Im}\bigg\{ -\frac{b_2}{b_1} h'_2(w)-\frac{b_2}{b_1}w\overline{h_2(w)}\bigg\}\\
  &=& {b_2}\sum_{n\geq 0}(n+1)a_{2,n} e_{n+1} .
\end{eqnarray}
$\bullet$ {\it Computation of  $\partial_{f_i}G_j^2(\lambda,0,0)h_i,\, i\neq j$.}
By straightforward computations we obtain
\begin{alignat}{2}\label{dfilj}
\partial_{f_i}G_j^2(0,0) h_i(w)&=(-1)^j b_j\textnormal{Im}\bigg\{ w\fint_{\mathbb{T}}\frac{(b_j\overline{w}-b_i\overline{\tau})}{b_jw-b_i\tau}{h_i'(\tau)}d\tau&&-b_iw\fint_{\mathbb{T}}\frac{\overline{h_i(\tau)}}{b_jw-b_i\tau}d\tau\notag\\ & &&+b_iw\fint_{\mathbb{T}}\frac{(b_j\overline{w}-b_i\overline{\tau})h_i(\tau)d\tau}{(b_jw-b_i\tau)^2}\bigg\}.
\end{alignat}
As $\overline{h_i}$ is holomorphic inside the open unit disc then by residue theorem we deduce that
\begin{equation*}
\fint_{\mathbb{T}}\frac{\overline{h_i(\tau)}}{b_1w-b_2\tau}d\tau=0,\quad  w\in \mathbb{T}.
\end{equation*}
It follows that
\begin{alignat}{2}\label{j}
\partial_{f_2} G_1^2(0,0)h_2(w)&=
- b_1\textnormal{Im}\bigg\{&& b_1\fint_{\mathbb{T}}\frac{h_2'(\tau)}{b_1w-b_2\tau}d\tau -b_2w\fint_{\mathbb{T}}\frac{\overline{\tau}{h_2^\prime(\tau)}}{b_1w-b_2\tau}d\tau \notag \\ & && +b_1b_2\fint_{\mathbb{T}}\frac{h_2(\tau)d\tau}{(b_1w-b_2\tau)^2}
  -b_2^2w\fint_{\mathbb{T}}\frac{\overline{\tau}h_2(\tau)d\tau}{(b_1w-b_2\tau)^2}\bigg\}\notag\\ &\triangleq- b_1\textnormal{Im}\Big\{&& J_1+J_2+J_3+J_4\Big\}.
  \end{alignat}
  To compute the first term $J_1(w)$ we write after using the series expansion of  $\frac{1}{1-\frac{b_2}{b_1}\overline{w}\tau}$
\begin{align*}
J_1 &={\overline{w}}\fint_{\mathbb{T}}\frac{h_2^\prime(\tau)}{1-\big(\frac{b_2}{b_1}\big)\overline{w}\tau}d\tau\notag\\ &=\sum_{n\geq 0}\Big(\frac{b_2}{b_1}\Big)^{n}\overline{w}^{n+1} \fint_{\mathbb{T}}\tau^nh_2'(\tau)d\tau.\notag\end{align*}
Note that
$$
\fint_{\mathbb{T}}\tau^nh_2'(\tau)d\tau=-n a_{2,n}
$$
which enables to get
\begin{equation}\label{j1}
J_1 =
-\sum_{n\geq 1}na_{2,n}\Big(\frac{b_2}{b_1}\Big)^n\overline{w}^{n+1}.
\end{equation}
As to the term $J_2(w)$ we write in a similar way
\begin{eqnarray*}
J_2 &=&- \frac{b_2}{b_1}\fint_{\mathbb{T}}\frac{\overline{\tau}h_2^\prime(\tau)}{1-\big(\frac{b_2}{b_1}\big)\overline{w}\tau}d\tau\notag\\
 &=&-\sum_{n\geq 0}\Big(\frac{b_2}{b_1}\Big)^{n+1}\overline{w}^{n} \fint_{\mathbb{T}}\tau^{n-1}h_2^\prime(\tau)d\tau.\notag
\end{eqnarray*}
Since $\displaystyle{\fint_{\mathbb{T}}\tau^{-k}h_2^\prime(\tau)d\tau=0}$ for $k\in\{0,1\}$ then the preceding sum starts at $n=2$ and by shifting the summation index we get
\begin{align}\label{j2}
J_2 
&=-\sum_{n\geq 1}\Big(\frac{b_2}{b_1}\Big)^{n+2}\overline{w}^{n+1} \fint_{\mathbb{T}}\tau^{n}h_2^\prime(\tau)d\tau\notag\\ 
&=\sum_{n\geq 1}na_{2,n}\Big(\frac{b_2}{b_1}\Big)^{n+2}\overline{w}^{n+1}. 
\end{align}
Concerning the third term  $J_3$  we write by virtue  of \eqref{Tay11}
\begin{align*}
J_3 &=\frac{b_2}{b_1}{\overline{w}^2}\fint_{\mathbb{T}}\frac{h_2(\tau)}{\big(1-\frac{b_2}{b_1}\,\overline{w}\tau\big)^2}d\tau\notag\\ &=\sum_{n\geq 1}n\Big(\frac{b_2}{b_1}\Big)^{n}\overline{w}^{n+1} \fint_{\mathbb{T}}\tau^{n-1}h_2(\tau)d\tau.
\end{align*}
Therefore  we  find
\begin{equation}\label{j3}
J_3=\sum_{n\geq 1}na_{2,n}\Big(\frac{b_2}{b_1}\Big)^n\overline{w}^{n+1}.
\end{equation}
Similarly we get
\begin{align}\label{j4}
J_4 &=-\Big(\frac{b_2}{b_1}\Big)^2{\overline{w}}\fint_{\mathbb{T}}\frac{\overline{\tau}h_2(\tau)}{\big(1-\big(\frac{b_2}{b_1}\big)\overline{w}\tau\big)^2}d\tau\notag\\ &=-\sum_{n\geq 1}n\Big(\frac{b_2}{b_1}\Big)^{n+1}\overline{w}^{n} \int_{\mathbb{T}}\tau^{n-2}h_2(\tau)d\tau\notag\\ &=-\sum_{n\geq 0}(n+1)a_{2,n}\Big(\frac{b_2}{b_1}\Big)^{n+2}\overline{w}^{n+1}.
\end{align}
Inserting the identities \eqref{j1},\eqref{j2},\eqref{j3} and \eqref{j4} into \eqref{j} we find
\begin{alignat}{2}\label{df2l1}
\partial_{f_2} G_1^2(0,0)h_2(w)&=
b_1\textnormal{Im}\bigg\{&&\sum_{n\geq 0}a_{2,n}\Big(\frac{b_2}{b_1}\Big)^{n+2}\overline{w}^{n+1}\bigg\}  \notag\\ &= {b_1}\sum_{n\geq 0}&&a_{2,n}\Big(\frac{b_2}{b_1}\Big)^{n+2}\,e_{n+1}(w).
\end{alignat}
Next, we shall move to the computation of $\partial_{f_1} G_2^2(0,0)h_1$. In view of \eqref{dfilj} one has 
\begin{eqnarray*}
\partial_{f_1} G_2^2(0,0)h_1(w)&=& b_2\textnormal{Im}\bigg\{ w\fint_{\mathbb{T}}\frac{(b_2\overline{w}-b_1\overline{\tau})}{b_2w-b_1\tau}{h_1'(\tau)}d\tau-b_1w\fint_{\mathbb{T}}\frac{\overline{h_1(\tau)}}{b_2w-b_1\tau}d\tau\notag\\ 
&+&b_1w\fint_{\mathbb{T}}\frac{(b_2\overline{w}-b_1\overline{\tau})h_1(\tau)d\tau}{(b_2w-b_1\tau)^2}\bigg\}.
\end{eqnarray*}
Residue theorem at infinity enables to  get rid of the first and the third integrals in the right-hand side and thus
\begin{eqnarray*}
\partial_{f_1} G_2^2(0,0)h_1(w)&=&-b_1 b_2\textnormal{Im}\bigg\{ w\fint_{\mathbb{T}}\frac{\overline{h_1(\tau)}}{b_2w-b_1\tau}d\tau\bigg\}.\notag\\ 
\end{eqnarray*}
A second application of residue theorem in the disc yields
\begin{eqnarray}\label{df1l2}
\partial_{f_1} G_2^2(0,0)h_1(w)
 &=&b_2\textnormal{Im}\bigg\{ w\overline{h_1}\Big(\frac{b_2w}{b_1}\Big)\bigg\}\notag\\ 
 &=&-{b_2}\sum_{n\geq 0}a_{1,n}\Big(\frac{b_2}{b_1}\Big)^n\, e_{n+1}(w).
\end{eqnarray}

$\bullet$ {\it Computation of  $\partial_{f_i}G_j^3(\lambda,0,0)h_i$.} The diagonal terms  $i=j$ can be easily computed, 
\begin{align}\label{lNj}
\partial_{f_i} G_i^3(0,0)h_j(w)&= (-1)^{i+1}b_i^3\textnormal{Im}\bigg\{  w\fint_{\mathbb{T}}\frac{h^\prime_i(w)d\tau}{1-b_i^2w\tau}+wh_i(w)\fint_{\mathbb{T}}\frac{\tau d\tau}{(1-b_i^2w\tau)^2}\bigg\}\notag\\ &=0.
\end{align}
Let us now calculate $\partial_{f_i}G_j^3(\lambda,0,0)h_i,$ for $i\neq j$. One can check with difficulty that
\begin{alignat*}{2}
\partial_{f_i} G_j^3(0,0)h_i(w)&=(-1)^{j+1} b_jb_i^2\textnormal{Im}\bigg\{ w\fint_{\mathbb{T}}\frac{{h_i'(\tau)}}{1-b_ib_jw\tau}d\tau&&+2w\fint_{\mathbb{T}}\frac{\textnormal{Re}\{\tau\,\overline{h_i(\tau)}\}}{1-b_ib_jw\tau}d\tau\notag\\ & &&+b_ib_jw^2\fint_{\mathbb{T}}\frac{h_i(\tau)d\tau}{(1-b_ib_jw\tau)^2}\bigg\}.
\end{alignat*}
Invoking once again  residue theorem  we find
\begin{eqnarray}\label{lNi}
\nonumber\partial_{f_i} G_j^3(0,0)h_i(w)&=& (-1)^{j+1} b_jb_i^2\textnormal{Im}\bigg\{ -\sum_{n\geq 0}na_{i,n}(b_jb_i)^nw^{n+1}+\sum_{n\geq 0}a_{i,n}(b_jb_i)^nw^{n+1}
\\&+&\sum_{n\geq 0}na_{i,n}(b_jb_i)^nw^{n+1}\bigg\}\notag\\ &=&(-1)^{j} {b_i}\sum_{n\geq 0}a_{i,n}(b_jb_i)^{n+1}e_{n+1}.
\end{eqnarray}
The details are left to the reader because most of them were done previously. Now putting together the identities \eqref{lh1},\eqref{lL1}and\eqref{lNj} we get
\begin{alignat}{2}\label{dg11}
\partial_{f_1}G_1(\lambda,0,0)h_1&=\sum_{n\geq 0}&&b_1\Big[(n+1)\lambda-1+b_1^{2n+2}-(n+1)\Big(\dfrac{b_2}{b_1}\Big)^2\Big]a_{1,n}\, e_{n+1}.\end{alignat}
From  \eqref{lh2},\eqref{lL2} and \eqref{lNj} one obtains
\begin{alignat}{2}\label{dg22}
\partial_{f_2}G_2(\lambda,0,0)h_2&=\sum_{n\geq 0}&&b_2\Big((n+1)\lambda-n-{b_2^{2n+2}}\Big)a_{2,n}\, e_{n+1}.
\end{alignat}
On the other hand, we observe that 
for $i\neq j$ one has
\begin{equation}\label{lhi}
\partial_{f_i} G_j^1(\lambda,0)h_i(w)=0.
\end{equation}
Gathering the identities \eqref{lhi},\eqref{df2l1} and \eqref{lNi} yields
\begin{alignat*}{2}
\partial_{f_2}G_1(\lambda,0,0)h_2&=\sum_{n\geq 0}&&b_2\Big[\Big(\dfrac{b_2}{b_1}\Big)^{n+1}-(b_1b_2)^{n+1}\Big]a_{2,n}\, e_{n+1}.
\end{alignat*}
Furthermore, combining   \eqref{lhi},\eqref{df1l2} and \eqref{lNi}  we can assert that
\begin{alignat*}{2}
\partial_{f_1}G_2(\lambda,0,0)h_1&=\sum_{n\geq 0}&&b_1\Big[(b_1b_2)^{n+1}-\Big(\dfrac{b_2}{b_1}\Big)^{n+1}\Big]a_{1,n}\, e_{n+1}.
\end{alignat*}
Consequently, we get in view of the  last two expressions combined with \eqref{dg22} and \eqref{lhi}
\begin{eqnarray}\label{dff}
DG(\lambda,0,0)\big(h_1,h_2\big)=\sum_{n\geq 0}M_{n+1}\left( \begin{array}{c}
a_{1,n} \\
a_{2,n}
\end{array} \right)\, e_{n+1},
\end{eqnarray}
where the matrix $M_n$ is given for each $n\geq1$  by
\begin{equation}\label{matrix}
M_{n}\triangleq\begin{pmatrix}
b_1\Big[n\lambda-1+b_1^{2n}-n\Big(\dfrac{b_2}{b_1}\Big)^2\Big] & b_2\Big[\Big(\dfrac{b_2}{b_1}\Big)^{n}-(b_1b_2)^{n}\Big] \\
 -b_1\Big[\Big(\dfrac{b_2}{b_1}\Big)^{n}-(b_1b_2)^{n}\Big] & b_2\Big[n\lambda-n+1-{b_2^{2n}}\Big]
\end{pmatrix}.
\end{equation}
This completes the proof of Proposition \ref{propspec1}.
\end{proof}
\subsection{Eigenvalues study}
The current task  will be  devoted to the study of the structure of the {\it nonlinear eigenvalues} which are the values $\lambda$ such that the linearized operator $DG(\lambda,0,0)$  given by \eqref{dff} has a non trivial kernel. Note that these eigenvalues correspond exactly  to matrices $M_n$ which are  not invertible for some integer $n\geq1.$ In other words, $\lambda$ is an eigenvalue if and only if there exists $n\geq1$ such that $\hbox{det } M_n=0$, that is,
\begin{eqnarray*}
\textnormal{det}\, M_{n}(\lambda)&=&
b_1b_2\Bigg[n^2\lambda^2-n\Big(n+b^{2n}_2-b^{2n}_1+n\Big(\frac{b_2}{b_1}\Big)^2\Big)\lambda+ 
 (n-1)\Big(1-b_1^{2n}+n\Big(\frac{b_2}{b_1}\Big)^{2}\Big)\\ &+&\Big(\frac{b_2}{b_1}\Big)^{2n}+nb_2^{2n}\Big(\frac{b_2}{b_1}\Big)^2-b_2^{2n}\Bigg]\\
&=&0.
\end{eqnarray*}
This is equivalent to 
\begin{eqnarray}\label{discr25}
 P_n(\lambda)&\triangleq& \lambda^2-\Big[1+\Big(\frac{b_2}{b_1}\Big)^2- \Big(\frac{b_1^{2n}-b_2^{2n}}{n}\Big)\Big]\lambda\\
 \nonumber&+&\Big(\frac{b_2}{b_1}\Big)^{2}-\frac{1-\big(\frac{b_2}{b_1}\big)^{2n}}{n^2}+\frac{1-\big(\frac{b_2}{b_1}\big)^{2}}{n}-\frac{b_1^{2n}-b_2^{2n}\big(\frac{b_2}{b_1}\big)^{2}}{n}+\frac{b_1^{2n}-b_2^{2n}}{n^2}\\
 \nonumber&=&0.
\end{eqnarray}

The reduced discriminant of this second-degree polynomial on $\lambda$ is given by
\begin{alignat}{2}
\Delta_n&=&&
 \bigg(\frac{1-\big(\frac{b_2}{b_1}\big)^2}{2}-\frac{2-b_2^{2n}-b_1^{2n}}{2n}\bigg)^2-\Big(\frac{b_2}{b_1}\Big)^{2n}\Big(\frac{1-b_1^{2n}}{n}\Big)^2.
\end{alignat}
Thereby $P_n$ admits two real roots if and only if $\Delta_n\geq0,$ and  they  are given by
$$
\lambda_n^\pm=\frac{1+\big(\frac{b_2}{b_1}\big)^2}{2}- \Big(\frac{b_1^{2n}-b_2^{2n}}{2n}\Big)\pm\sqrt{\Delta_n}\,.
$$
To  understand the structure of the eigenvalues and their dependence with respect to the involved parameters, it would be better to fix the radius $b_1$ and to vary $n$ and $b_2\in]0,b_1[$. We shall distinguish the cases $n\geq2$ from $n=1$ which is very special.  For given $n\geq2$ we wish to draw the curves $b_2\mapsto \lambda_n^{\pm}(b_2)$. As we shall see in Proposition \ref{asympbeh} the maximal domain  of existence of these curves are a common   connected set  of the form $[0, b_n^\star]$ and $b_n^\star$ is defined as the unique $b_2\in ]0,b_1[$ such that $\Delta_n=0.$ 
We introduce the graphs $\mathcal{C}_n^\pm$ of $ \lambda_n^\pm(b_2)$,
\begin{equation}\label{curvn}
\mathcal{C}_n^\pm\triangleq\Big\{ \big(b_2,\lambda_n^\pm(b_2)\big), b_2\in[0, b_n^\star]\Big\}\quad\hbox{and}\quad \mathcal{C}_n=\mathcal{C}_n^-\cup \mathcal{C}_n^+\quad\hbox{with}\quad n\geq2.
\end{equation}
It is not hard to check that  $\mathcal{C}_n^+$ intersects  $\mathcal{C}_n^-$ at only one point whose abscissa is $b_n^\star$, that is, when the discriminant vanishes. Furthermore, and this is not trivial, we shall see that the domain enclosed by the curve $\mathcal{C}_n$ and located in the first quadrant of the plane is a strictly  increasing set on $n.$ This will give in particular the monotonicity of the eigenvalues with respect to $n$. Nevertheless, the dynamics of the first eigenvalues corresponding to  $n=1$ is completely different from the preceding ones. Indeed, according to  Subsection \ref{first11} we find for $n=1$ two eigenvalues given explicitly by
$$
\lambda_1^-=(b_2/b_1)^2\quad\textnormal{or}\quad\lambda_1^+=1+b_2^2-b_1^2. 
$$
It turns out that for  the first one  the range of the linearized operator has an infinite co-dimension and therefore there  is no hope to bifurcate using only the classical results in the bifurcation theory. However for the  second eigenvalue the range is ``almost everywhere'' of co-dimension one  and the bifurcation is likely to happen.   As to the structure of this eigenvalue, it  is strictly increasing with respect to $b_2$ and by working more we prove  that  the curve $\mathcal{C}_1^+$ of $b_2\in]0,b_1[\mapsto \lambda_1^+$   intersects $\mathcal{C}_n$ if and only if  $n\geq b_1^{-2}$. We can now make precise statements  of these results and for the  complete ones  we refer the reader  to Lemma \ref{lem-mono}, Proposition \ref{asympbeh} and \mbox{Proposition \ref{prop-one}.}
\begin{proposition}\label{propxx23}
Let $b_1\in ]0,1[$ then the following holds true. 
\begin{enumerate}
\item The sequence $n\geq2\mapsto b_n^\star$ is strictly increasing.
\item Let $2\le n<m$ and $b_2\in [0, b_n^\star[$, then
$$
\lambda_m^-<\lambda_n^-<\lambda_n^+<\lambda_m^+.
$$
\item The curve $\mathcal{C}_1^+$ intersects $\mathcal{C}_n$ if and only if $n\geq\frac{1}{b_1^2}.$ In this case we have a single point $(x_n,\lambda_1^+(x_n))$, with $x_n\in ]0,b_n^\star]$ being the only solution $b_2$ of the equation
$$
P_n(1+b_2^2-b_1^2)=0.
$$
where $P_n$ is defined in \eqref{discr25}.
\end{enumerate}

\end{proposition}

The properties mentioned in the preceding proposition can be illustrated  by Figure \ref{Figg15}. Further illustrations will be given in Figure \ref{f:lambda0.250.50.750.99}.
\begin{figure}[!htb]
\center


\includegraphics[width=0.5\textwidth, clip=true]{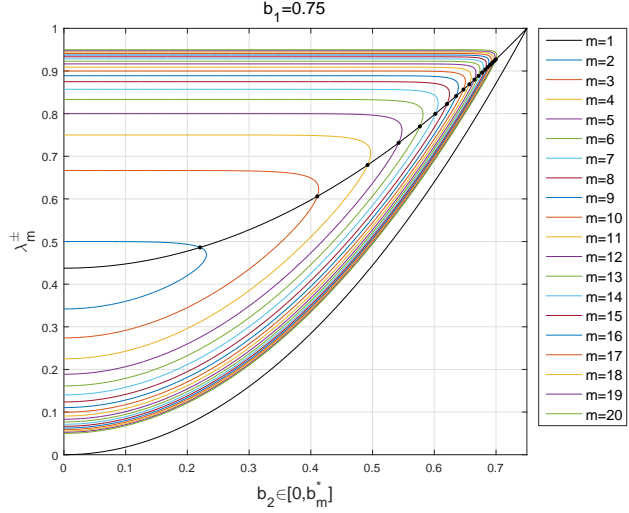}~
\caption{$\lambda_m^\pm$ as a function of $b_2\in[0, b_m^\star]$, for $m = 2, \ldots, 20$, together with the case $m = 1$ (black), for $b_1=0,75$.}
\label{Figg15}
\end{figure}
 
For the proof of Proposition \ref{propxx23} it appears  to be more convenient to work  with a continuous variable instead of the discrete one $n$. This is advantageous especially  in the study of the variations of the eigenvalues with respect to $n$ and the radius $b_2$ for $b_1$ fixed. To do so, we extend  in a natural way  $(\Delta_n)_{n\geq1}$  to a smooth function defined on $[1,+\infty[$   as follows
$$
\Delta_x= \bigg(\frac{1-\big(\frac{b_2}{b_1}\big)^2}{2}-\frac{2-b_2^{2x}-b_1^{2x}}{2x}\bigg)^2-\Big(\frac{b_2}{b_1}\Big)^{2x}\Big(\frac{1-b_1^{2x}}{x}\Big)^2,\quad x\in[1,+\infty[.
$$
It is easy to see that  $\Delta_x$ is  positive if and only if
\begin{equation}\label{det}
\Big(1-\Big(\frac{b_2}{b_1}\Big)^2\Big)x-\Big(2-b_2^{2x}-b_1^{2x}\Big)-2\Big(\frac{b_2}{b_1}\Big)^{x}\Big(1-b_1^{2x}\Big)\geq 0
\end{equation}
or
$$
E_x \triangleq \Big(1-\Big(\frac{b_2}{b_1}\Big)^2\Big)x-\Big(2-b_2^{2x}-b_1^{2x}\Big)+2\Big(\frac{b_2}{b_1}\Big)^{x}\Big(1-b_1^{2x}\Big)<0.
$$
We shall prove  that the last possibility $E_x<0$ is excluded for $x\geq 2$. Indeed,
\begin{alignat}{2}
E_x &= \Big(1-(b_2/b_1)^2\Big)x-2\Big(1-(b_2/b_1)^{x}\Big)+\big(b_2^{x}-b_1^{x}\big)^2\notag\\ &= 2 \Big(1-(b_2/b_1)^2\Big)\Big[\frac x2-\frac{1-\big((b_2/b_1)^2\big)^{\frac x2}}{1-(b_2/b_1)^2}\Big)\Big]+\big(b_2^{x}-b_1^{x}\big)^2\notag\\ &\geq  \big(b_2^{x}-b_1^{x}\big)^2>0,\notag
\end{alignat}
where we have used the classical inequality
$$
\forall b\in (0,1),\,\forall x\geq1\quad \frac{1-b^x}{1-b}\le x.
$$
Thus for $x\geq2$ the condition $\Delta_x\geq0$ is equivalent to  the first one of   \eqref{det} or, in other words,
\begin{equation}\label{det1}
x\geq \frac{2+2\big(\frac{b_2}{b_1}\big)^{x}-{\big(b_1^{x}+b_2^{x}\big)^2}}{1-(b_2/b_1)^2}\triangleq g_x(b_1,b_2).
\end{equation}
In this case the roots of the polynomial  $P_n$ can be also continuously extended as follows
\begin{equation*}
\lambda_x^+=\frac{1+\big(\frac{b_2}{b_1}\big)^2}{2}- \Big(\frac{b_1^{2x}-b_2^{2x}}{2x}\Big)+\sqrt{\Delta_x}
\end{equation*}
and
\begin{equation*}
\lambda_x^-=\frac{1+\big(\frac{b_2}{b_1}\big)^2}{2}- \Big(\frac{b_1^{2x}-b_2^{2x}}{2x}\Big)-\sqrt{\Delta_x}.
\end{equation*}

\subsubsection{Monotonicity for $n\geq2$}\label{secz1}
To settle the proof of the second point $(ii)$ of Proposition \ref{propxx23}  we should   look for the variations of the eigenvalues with respect to $x$ but with fixed radii $b_1$ and $b_2.$ For this purpose we need first to understand  the topological structure of the domain of definition of $x\mapsto\lambda_x^\pm$,
$$
\mathcal{I}_{b_1,b_2}\triangleq\big\{x\geq 2,\,\Delta_x>0\big\}
$$
and to see in particular whether this set is connected or not. We shall establish the following.
\begin{lemma}\label{lemrsq1}
Let $0<b_2<b_1<1$ two fixed numbers, then the following holds true.
\begin{enumerate}
\item The set $\mathcal{I}_{b_1,b_2}$ is connected and of the form $(\mu_{b_1,b_2},\infty[$.
\item The map $x\in \mathcal{I}_{b_1,b_2}\mapsto \Delta_x$ is strictly increasing.

\end{enumerate}
\end{lemma}
\begin{remark}
If the discriminant $\Delta_x$ admits a zero then it will be unique and coincides with the value  $\mu_{b_1,b_2}$. Otherwise  $\mu_{b_1,b_2}$  will be equal to  $2$.
\end{remark}
\begin{proof}
To get this result it suffices to check the following: for any $a\in \mathcal{I}_{b_1,b_2}$ we have
$$[a,+\infty[\subset \mathcal{I}_{b_1,b_2}.
$$
 By the continuity of the discriminant, there exists $\eta>a$ such that  $[a,\eta[\subset\mathcal{I}_{b_1,b_2}$ and let $[a,\eta^\star[$ be the maximal interval contained in $\mathcal{I}_{b_1,b_2}$. If $\eta^\star$ is finite then necessarily $\Delta_{\eta^\star}=0$. If we could show that the discriminant is strictly increasing in this interval then this will contradict the preceding assumption. 
 To see this, observe that   $\Delta_x$ can be rewritten in the form,
 \begin{alignat}{2}\label{exdes}
\Delta_x&=&& \frac14\Big(f_1\Big(\frac{b_2}{b_1}\Big)-{f_x(b_1)-f_x(b_2)}\Big)^2-\Big(\frac{b_2}{b_1}\Big)^{2x}f^2_x(b_1)
\end{alignat}
with the notation
$$
f_x(t)\triangleq\frac{1-t^{2x}}{x}.
$$

Differentiating  $\Delta_x$ with respect to $x $ we get
\begin{alignat}{2}\label{Deltazqq}
\nonumber\partial_x\Delta_x&=&& -\frac12\Big({\partial_x f_x}(b_1)+{\partial_x f_x}(b_2)\Big) \Big(f_1\Big(\frac{b_2}{b_1}\Big)-{f_x(b_1)-f_x(b_2)}\Big)\\
& &&-2f_x(b_1)\Big(\frac{b_2}{b_1}\Big)^{2x}\Big(f_x(b_1)\log \Big(\frac{b_2}{b_1}\Big)+{\partial_x f_x}(b_1)\Big).
\end{alignat}
We shall prove that for all $t\in ]0,1[$, the mapping   $x\in [2,\infty[\mapsto f_x(t)$ is strictly  decreasing. It is clear that
\begin{equation}
{\partial_x f_x(t)}=\frac{t^{2x}\big(1-2x\log t\big)-1}{x^2}\triangleq \frac{g_x(t)}{x^2}.
\end{equation}
To study the variation of $t\mapsto g_x(t)$, note that
$$
g_x'(t)=-4x^2t^{2x-1}\log t>0,\quad\forall t\in]0,1[
$$
and therefore  $g_x$ is strictly increasing which implies that
$$
{\partial_x f_x(t)}<  \frac{g_x(1)}{x^2}=0.
$$
Using this fact we deduce that the last term of \eqref{Deltazqq} is positive and  consequently
$$
\partial_x\Delta_x\geq -\frac12\Big({\partial_x f_x}(b_1)+{\partial_x f_x}(b_2)\Big) \Big(f_1\Big(\frac{b_2}{b_1}\Big)-{f_x(b_1)-f_x(b_2)}\Big).
$$
Hence,  to get $\partial_x\Delta_x>0$ it suffices to establish that 
\begin{equation}\label{pdel}
f_1\Big(\frac{b_2}{b_1}\Big)-{f_x(b_1)-f_x(b_2)}>0,
\end{equation}
which is equivalent to
$$
x> \frac{2-b_1^{2x}-b_2^{2x}}{1-b^2}\quad \hbox{with}\quad b=\frac{b_2}{b_1}\cdot
$$
Note that we have already seen that the positivity of  $\Delta_x$ for $ x\geq2$ is equivalent to  the condition \eqref{det1} which actually implies  the preceding one owing  to the strict inequality 
$$
b^x-(b_1 b_2)^x>0.
$$
 This shows that \eqref{pdel} is true and consequently 
 $$
 \forall x\in [a, \eta^\star[,\quad \partial_x\Delta_x>0.
 $$
 This shows that the discriminant, which is positive,  is strictly increasing in  $[a, \eta^\star[$ and this excludes the fact that $\Delta_{\eta^\star}$ vanishes. Therefore $\eta^\star=\infty$ and the thus the proofs of $(\hbox{i})$ and $(\hbox{ii})$ are simultaneously proved.

\end{proof}

The next goal is to establish the monotonicity of the eigenvalues. 
\begin{lemma} \label{lem-mono}
Let $0<b_2<b_1<1$, then we have: \begin{enumerate}
\item The mapping $x\in\mathcal{I}_{b_1,b_2}\mapsto\lambda_x^+$ is strictly increasing.
\item The mapping $x\in\mathcal{I}_{b_1,b_2}\mapsto\lambda_x^-$ is strictly decreasing.
\item For any $x<y\in\mathcal{I}_{b_1,b_2}$ we have
$$
\lambda_y^-<\lambda_x^-<\lambda_x^+<\lambda_y^+.
$$
\end{enumerate}
\end{lemma}
\begin{proof}
{ (i)}  Note that
\begin{equation*}
\lambda_x^+=\frac{1+b^2}{2}-\frac{b_1^{2x}}{2}f_x\big(b\big)+\sqrt{\Delta_x},\quad b=\frac{b_2}{b_1}\cdot
\end{equation*}
We have already seen in the proof of Lemma \ref{lemrsq1} that for any $t\in ]0,1[$  the mapping   $x\in [2,\infty[\mapsto f_x(t)$ is strictly decreasing and therefore  $x\mapsto {b_1^{2x}}f_x\big(\frac{b_2}{b_1}\big)$ is also strictly decreasing. To get the strict increasing of $x\mapsto \lambda_x^+$ it suffices to  combine this last fact with the increasing property of $ x\mapsto \Delta_x$.
\\

{ (ii)}
It is plain that
\begin{equation*}
\lambda_x^-=\frac{1+b^2}{2}+\frac{f_x(b_1)-f_x(b_2)}{2}-\sqrt{\Delta_x}.
\end{equation*}
The derivative of  $\lambda_x^-$ with respect to $x$  is given by
$$
{\partial_x \lambda_x^-}= \frac12{\partial_x f_x}(b_1)-\frac12\partial_x f_x(b_2)-\frac{\partial_x \Delta_x}{2\sqrt{\Delta_x}}.
$$
By virtue of \eqref{Deltazqq} we can split the preceding function into three parts:
$$
{\partial_x \lambda_x^-}= \textnormal{I}+\textnormal{II}+\textnormal{III},
$$
where
\begin{equation*}
\textnormal{I}\triangleq \frac12\partial_x f_x(b_1)\Big(1+\frac{f_1\big(b\big)-f_x(b_1)-f_x(b_2)}{2\sqrt{\Delta_x}}\Big),
\end{equation*}
\begin{equation*}
\textnormal{II}\triangleq \frac12\partial_x f_x(b_2)\Big(-1+\frac{f_1\big(b\big)-f_x(b_1)-f_x(b_2)}{2\sqrt{\Delta_x}}\Big)
\end{equation*}
and
\begin{equation*}
\textnormal{III}\triangleq \frac{b^{2x}f_x(b_1)\big(f_x(b_1)\log (b)+\partial_x f_x(b_1)\big)}{\sqrt{\Delta_x}}\cdot
\end{equation*}
Have in mind the inequality \eqref{pdel} and $\partial_x f_x(t)<0$ for any $t\in ]0,1[$ we can see that  I is negative. 
To prove that the term II is also negative it suffices to check that
$$
\frac{f_1\big(b\big)-f_x(b_1)-f_x(b_2)}{2\sqrt{\Delta_x}}>1.
$$
From \eqref{pdel} we can deduce  by squaring that the last expression  is actually equivalent to
$$
\frac14\Big(f_1\Big(\frac{b_2}{b_1}\Big)-{f_x(b_1)-f_x(b_2)}\Big)^2>{\Delta_x}.
$$
From \eqref{exdes}  we immediately conclude   that  the last inequality is always verified.\\
In regard the negativity of the third term $\hbox{III}$ we just use the facts $0<b<1$ and the decreasing of the function $x\mapsto f_x(t)$.
\\

{(iii)} This follows easily from $(\hbox{i}), (\hbox{ii})$ and the obvious fact
$$
\forall x\in\mathcal{I}_{b_1,b_2},\quad  \lambda_x^-<\lambda_x^+.
$$
\end{proof}

\subsubsection{Lifespan of the eigenvalues with respect to $b_2$}
We shall study in this section   some properties of the  eigenvalues functions  $b_2\mapsto\lambda_n^{\pm}$ for $n\geq 2$ and   $b_1$  fixed. This will be crucial for studying the dynamics of the first eigenvalue $\lambda_1^+$ and especially in counting the intersection between the curves $\mathcal{C}_1^+$ and $\mathcal{C}_n$ which has been the subject of the part $(\hbox{iii})$ of Proposition \ref{propxx23}.  
Note that in this paragraph  we shall give up using the continuous version $\lambda_x^\pm$ of the roots $\lambda_n^\pm,$ as it has been done in the preceding section. The results that we shall state can actually be proved with the  continuous parameter, however    this does not matter a lot for our final purpose.  We define the following set; for $n\geq 2$ and $b_1\in]0,1[,$ 
$$
\mathcal{J}_{n, b_1}\triangleq\Bigg\{ b_2\in[0, b_1[, n\geq \frac{2+2\big(\frac{b_2}{b_1}\big)^{n}-{\big(b_1^{n}+b_2^{n}\big)^2}}{1-(b_2/b_1)^2}\Bigg\}.
$$
We shall prove the following.
\begin{proposition}\label{asympbeh}
Let $b_1\in ]0,1[$ fixed and  $n\geq2$,  then the following holds true.
\begin{enumerate}
\item The set  $\mathcal{J}_{n, b_1}$ is an interval of the form $[0, b_n^\star]$, with $b_n^\star\in ]0,b_1[$.
\item The eigenvalues $b_2\mapsto \lambda_n^{\pm}$ are together defined in $[0, b_n^\star].$ 
\item The sequence $n\mapsto b_n^\star$ is strictly increasing and we have the asymptotics
$$
b_n^\star= b_1(1-\alpha/n)+o(\frac1n),\quad\hbox{with}\quad  e^{-\alpha}+1=\alpha\quad \hbox{and}\quad \alpha\approx 1,27846.
$$

\item The function  $b_2\in[0, b_n^\star]\mapsto \lambda_n^{-}(b_2)-b_2^2$ is strictly increasing.

\item The function $b_2\in[0, b_n^\star]\mapsto \lambda_n^{+}(b_2)-b_2^2$ is strictly decreasing.
\end{enumerate}

\end{proposition}
\begin{proof}
$(\hbox{i})$ This follows from studying the function $h:[0,b_1]\to \RR,$ defined by
$$
h(x)=n(1-(x/b_1)^2)-2-2\big(\frac{x}{b_1}\big)^{n}+{\big(b_1^{n}+x^{n}\big)^2}.
$$
We claim that $h$ is strictly decreasing. Indeed, by differentiating we get
\begin{eqnarray*}
h^\prime(x)&=&\frac{2n x}{b_1^2}\Big(-1+b_1^2 x^{2n-2}\Big)+\frac{2n x^{n-1}}{b_1^n}\Big(-1+b_1^{2n}\Big)\\
&<&0.
\end{eqnarray*}
As $h(0)=n-2+b_1^{2n}>0 $ and $h(b_1)=4(-1+b_1^{2n})<0$  then we deduce from the intermediate value theorem that the  set $\mathcal{J}_{n, b_1}$ is in fact  an interval of the form $[0, b_n^\star].$ The number $b_n^\star\in [0, b_1[$ is defined by the unique solution of the equation
\begin{equation}\label{ss1}
h(b_n^\star)=0.
\end{equation}
\\
$(\hbox{ii})$ Observe that the domain of definition of the  eigenvalues $\lambda_n^\pm$ coincides with the domain of the discriminant $\Delta_n$, which is in turn given by  $\mathcal{J}_{n, b_1}$ according to \eqref{det1}.  that of the  set $\mathcal{J}_{n, b_1}$. Therefore the equation \eqref{ss1} do imply the vanishing of $\Delta_n$ at the point $b_n^\star$, and consequently both eigenvalues coincide. 
\\

$(\hbox{iii})$  Recall from \eqref{curvn} the definitions of  the curves  $\mathcal{C}_n^\pm$  and $\mathcal{C}_n=\mathcal{C}_n^-\cup \mathcal{C}_n^+$.
Since  the eigenvalues $\lambda_n^+(b_n^\star)$ and $\lambda_n^-(b_n^\star)$ coincide then   curves $\mathcal{C}_n^+$ and $\mathcal{C}_n^-$ end at the same point which is a turning point for $\mathcal{C}_n.$ Furthermore, 
we can see that $\mathcal{C}_n$  lies in the left side of the vertical axis $x=b_n^\star.$ Now let $m>n\geq2$  and we intend to check  by some elementary geometric considerations  that $b_m^\star>b_n^\star.$ From the monotonicity of the eigenvalues $n\mapsto \lambda^\pm_n$ we have
$$
\lambda_m^-(0)<\lambda_n^-(0),\quad\hbox{and}\quad \lambda_m^+(0)>\lambda_n^+(0).
$$ 
If $b_m^\star\le b_n^\star$ then the curve $\mathcal{C}_m$ will intersect $\mathcal{C}_n$ at some point and this contradicts the strict monotonicity  of the eigenvalues  with respect to $n$. Thus we deduce that   $n\mapsto b_n^\star$ is strictly increasing and  therefore it should  converge to some value $b^\star\le b_1.$  Assume that $b^\star<b_1$ then from the equation \eqref{ss1} and the continuity of $h$  we find by letting $n\to+\infty$ that
\begin{eqnarray*}
\lim_{n\to+\infty}h(b_n^\star)&=&0.
\end{eqnarray*}
On the other hand,
\begin{eqnarray*}
\lim_{n\to+\infty}h(b_n^\star)
&=&\lim_{n\to+\infty}n(1-(b_n^\star/b_1)^2)-2\\
&=&+\infty
\end{eqnarray*}
which is clearly a  contradiction and thus $b^\star=b_1.$ For the asymptotic  behavior of $b_n^\star$, which is a marginal part here, we shall settle for a formal reasoning by making a Taylor expansion  at the first order on $\frac1n$. We shall look for $\alpha$ such that,
$$
b_n^\star =b_1(1-\frac\alpha n)+o(1/n).
$$
At the first order of $h$ we get
$$
h(b_n^\star)=\alpha(2-\alpha/n)-2-2(1-\alpha/n)^n+o(1).
$$
By taking the limit as $n\to\infty$ we find that $\alpha$ must satisfy
$$
e^{-\alpha}+1=\alpha.
$$
This equation admits a unique solution lying in the interval $]1,2[$ and can given explicitly by the Lambert W function,
$$
\alpha=W(e^{-1})+1\approx 1,27846.
$$
\vspace{0,2cm}

$(\hbox{iv})$ Set $x=\big(\frac{b_2}{b_1}\big)^2$ and define the functions
$$
f_\pm(x)=\lambda_n^\pm(b_2)=\frac{1+x}{2}+\frac{b_1^{2n}}{2n}\big(x^n-1\big)\pm\sqrt{\Delta_n(x)},\quad  x\in \Big[0, \frac{{b_n^\star}^2}{b_1^2}\Big],
$$
with 
$$
\Delta_n(x)=
 \bigg(\frac{1-x}{2}-\frac{2-b_1^{2n}\big(1+x^n\big)}{2n}\bigg)^2-x^{n}\Big(\frac{1-b_1^{2n}}{n}\Big)^2.
$$
Differentiating with respect to $x$ yields,
$$
\Delta_n^\prime(x)=-\bigg(\frac{1-x}{2}-\frac{2-b_1^{2n}\big(1+x^n\big)}{2n}\bigg)\bigg(1-{b_1^{2n}}x^{n-1}\bigg)-n x^{n-1}\Big(\frac{1-b_1^{2n}}{n}\Big)^2.
$$
Note from the assumption \eqref{det1}, by switching  the parameters $n$ and $x$, that 
$$\frac{1-x}{2}-\frac{2-b_1^{2n}\big(1+x^n\big)}{2n}>0
$$ and therefore we get
$$
\Delta_n^\prime(x)<0,\quad \forall  x\in \Big[0, \frac{{b_n^\star}^2}{b_1^2}\Big]\subset[0,1[.
$$
Coming back to the function $f_\pm$ and taking the derivative we find
$$
{f^\prime_\pm}(x)=\frac12+\frac{b_1^{2n}}{2} x^{n-1}\pm\frac{\Delta_n^\prime(x)}{2\sqrt{\Delta_n(x)}}\cdot
$$
Using the definition of $\Delta_n$ and \eqref{det} one has
$$
\bigg(\frac{1-x}{2}-\frac{2-b_1^{2n}\big(1+x^n\big)}{2n}\bigg)>\sqrt{\Delta_n(x)}
$$
and consequently,
\begin{eqnarray*}
\frac{\Delta_n^\prime(x)}{\sqrt{\Delta_n(x)}}&\le& -\frac{\bigg(\frac{1-x}{2}-\frac{2-b_1^{2n}\big(1+x^n\big)}{2n}\bigg)}{\sqrt{\Delta_n(x)}}\bigg(1-{b_1^{2n}}x^{n-1}\bigg)\\
&<&-\bigg(1-{b_1^{2n}}x^{n-1}\bigg).
\end{eqnarray*}
Therefore  we obtain  that for all $x\in \Big[0, \frac{{b_n^\star}^2}{b_1^2}\Big]$,
$$
f^\prime_{-}(x)> 1,
$$
and
$$
f^\prime_+(x)\leq b_1^{2n} x^{n-1}<b_1^{2}. 
$$
This shows that the function $g_-:x\mapsto f_-(x)-b_1^2x$ is strictly increasing, however  the function $g_+:x\mapsto f_+(x)-b_1^2x$ is strictly decreasing.
This achieves the proof of the desired result.
\end{proof}

\subsubsection{Dynamics of the first eigenvalue}\label{first11}
We shall in this paragraph discuss the behavior of the first eigenvalues corresponding to $n=1.$ Note from the equation \eqref{discr25} that these eigenvalues are in fact the solutions of the polynomial
$$
P_1(\lambda)=\lambda^2-\Big(1+b^{2}_2-b^{2}_1+\big(b_2/b_1)^2\Big)\lambda+(b_2/b_1)^2+b_2^{2}(b_2/b_1)^2-b_2^{2}
$$
which vanishes exactly at the points,
$$
\lambda_1^-=(b_2/b_1)^2\quad\textnormal{or}\quad\lambda_1^+=1+b_2^2-b_1^2.
$$
Recall from the preceding sections the following definition
$$
\mathcal{C}_n^\pm\triangleq\Big\{ (b_2,\lambda_n^\pm(b_2)), b_2\in[0, b_n^\star]\Big\},  \quad \mathcal{C}_n= \mathcal{C}_n^-\cup \mathcal{C}_n^+,
$$
and the graph of  the first eigenvalue $\lambda_1^+$ is given by:
\begin{equation*}
\mathcal{C}_1^+\triangleq\Big\{ (b_2,1+b_2^2-b_1^2), b_2\in[0, b_1]\Big\}.
\end{equation*}
As we have already mentioned it is not clear whether or not the bifurcation occurs with $\lambda_1^-$ because the range of the linearized operator has an infinite co-dimension. The main result reads as follows.
\begin{proposition}\label{prop-one}
Let $b_1\in]0,1[$ and $n\geq2$. Then the following holds true. 
\begin{enumerate}
\item For any $0<b_2<b_1$  we have $\lambda_1^-<\lambda_n^{\pm}.$
\item If $n<b_1^{-2},$ then
$$
\mathcal{C}_n\cap \mathcal{C}_1^+=\varnothing.
$$
\item If  $n\geq b_1^{-2},$ then $\mathcal{C}_n\cap \mathcal{C}_1^+$ is a single point, that is, there exists $x_n\in [0, b_n^\star]$ such that
$$
\mathcal{C}_n\cap \mathcal{C}_1^+=\big\{(x_n,\lambda_1^+(x_n))\big\}.
$$
\item If $b_2\not\in \big\{x_m, m\geq b_1^{-2}\big\}$ then for all $n\geq 2,$ $\lambda_1^+\neq \lambda_n^\pm$.
\item The sequence $\{x_m\}_{m\geq b_1^{-2}}$ is increasing and  converges to $b_1.$

\end{enumerate}

\end{proposition}
\begin{proof}
{ (i)} This follows easily from the monotonicity of the eigenvalue $n\mapsto\lambda_n^-$ and the fact that $\lambda_n^-\leq\lambda_n^+$. Indeed, for all $n\geq 2$ we have 
$$
\lambda_1^-=\Big(\frac{b_2}{b_1}\Big)^2=\lim_{n\to +\infty}\lambda_n^-<\lambda_n^-\leq \lambda_n^+.
$$
 \\
 
{(ii)} In view of  (v) from Proposition \ref{asympbeh} the mapping $b_2\in [0,b_n^\star]\mapsto \lambda_n^+(b_2)-\lambda_1^+(b_2)$ is strictly decreasing, and therefore for $b_2\in]0,b_n^\star]$ we get
\begin{equation*}
\lambda_n^+(b_2)-\lambda_1^+(b_2)<\lambda_n^+(0)-\lambda_1^+(0)=b_1^2-\frac{1}{n}\cdot
\end{equation*}
Therefore for $n<b_1^{-2}$ the last term in the right-hand side  is negative and consequently,
$$
\lambda_n^-(b_2)\leq \lambda_n^+(b_2)<\lambda_1^+(b_2),\quad \forall b_2\in ]0,b_n^\star].
$$
\\
{(iii)} When $n\geq b_1^{-2}$ then  $\lambda_n^+(0)-\lambda_1^+(0)\geq0$ and since  $b_2\in [0,b_n^\star]\mapsto \lambda_n^+(b_2)-\lambda_1^+(b_2)$ is is strictly decreasing then the equation $\lambda_n^+(b_2)-\lambda_1^+(b_2)=0$ has at most one solution in $[0,b_n^\star]$. We shall distinguish three cases: the first one is when $\lambda_n^+(b_n^\star)-\lambda_1^+(b_n^\star)<0$, in which case the foregoing equation  admits a unique  solution denoted by $x_n$. This implies that $\mathcal{C}_n^+\cap \mathcal{C}_1^+$ is a single point  whose abscissa is $x_n$ and  the next step is to check  that $\mathcal{C}_n^-\cap \mathcal{C}_1^+$ is empty. Thus, we get
$$
\lambda_n^+(b_n^\star)-\lambda_1^+(b_n^\star)\leq \lambda_n^+(x_n)-\lambda_1^+(x_n)= 0.
$$
Combining  the last inequality with the fact that $\lambda_n^+(b_n^\star)=\lambda_n^-(b_n^\star)$ and  the monotonicity of the mapping $b_2\in [0,b_n^\star]\mapsto \lambda_n^-(b_2)-\lambda_1^+(b_2)$, which follows from (iv) of Proposition \ref{asympbeh}, we conclude that  for all $b_2\in]0,b_n^\star]$ we have
\begin{eqnarray*}
 \lambda_n^-(b_2)-\lambda_1^+(b_2)&\le&\lambda_n^-(b_n^\star)-\lambda_1^+(b_n^\star)\\
  &\le&\lambda_n^+(b_n^\star)-\lambda_1^+(b_n^\star)\\
  &<& 0.
\end{eqnarray*}
Therefore,
$
\mathcal{C}_n^-\cap \mathcal{C}_1^+=\varnothing
$ and the set $\mathcal{C}_n\cap\, \mathcal{C}_1^+$ reduces to a single point. 
The second case is when $\lambda_n^+(b_n^\star)-\lambda_1^+(b_n^\star)>0$ then $\mathcal{C}_n^+\cap C_1^+$ is empty and we shall prove that $\mathcal{C}_n^-\cap \mathcal{C}_1^+$ is a single point.  Observe first that $$
\lambda_n^-(b_n^\star)-\lambda_1^+(b_n^\star)=\lambda_n^+(b_n^\star)-\lambda_1^+(b_n^\star)>0.
$$
Moreover,
$$
\lambda_n^-(0)-\lambda_1^+(0)=\frac{1-b_1^{2n}}{n}-(1-b_1^2)<0,\quad\forall n\geq2.
$$
Since $b_2\mapsto \lambda_n^-(b_2)-\lambda_1^+(b_2)$ is strictly increasing then  by the intermediate value theorem, there exists only one solution $x_n\in ]0,b_n^\star[$  of the equation $\lambda_n^-(b_2)-\lambda_1^+(b_2)=0$.
The third and last case to analyze  is when $\lambda_n^+(b_n^\star)-\lambda_1^+(b_n^\star)=0$. This means that all the curves $\mathcal{C}_n^+, \mathcal{C}_n^-$ and $\mathcal{C}_1^+$ meet each other at the single point of abscissa $b_n^\star.$
\\

{(iv)} It follows immediately from  (ii) and (iii).  \\
\\

{(v)}  Let $n\geq b_1^{-1}$ and  define the set enclosed by $\mathcal{C}_n$ and located at the first quadrant of the plane,
$$
\widehat{\mathcal{C}}_n\triangleq \Big\{(x,y)\in\RR^2;\, x\in [0,b_n^\star],\, \lambda_n^-(x)\le y \le \lambda_n^+(x)\Big\}.
$$
From the monotonicity of the eigenvalues $n\mapsto \lambda_{n}^\pm$ seen in Lemma \ref{lem-mono}  we note  that
$$
\forall (x,y)\in \widehat{\mathcal{C}}_n, \quad  \lambda_{n+1}^-(x)< \lambda_n^-(x)\le y \le \lambda_n^+(x)< \lambda_{n+1}^+(x).
$$
Hence, we  get
\begin{equation}\label{imb}
\widehat{\mathcal{C}}_n\Subset \widehat{\mathcal{C}}_{n+1}\quad \hbox{and}\quad \mathcal{C}_{n+1}\cap \widehat{\mathcal{C}}_n=\varnothing.
\end{equation}
Now, from the point (iii) and the monotonicity of the mappings $b_2\mapsto \lambda_n^\pm(b_2)- \lambda_1^+(b_2)$ stated in Proposition \ref{asympbeh} we deduce that 
$$
\forall x\in [0,x_n[;\quad  \lambda_n^-(x)< \lambda_1^+(x)< \lambda_n^+(x).
$$
Then we have the inclusion
$$
\mathcal{C}_{1,n}^+\triangleq\Big\{\big(x,\lambda_1^+(x)\big);\, x\in[0, x_n]\Big\}\subset \widehat{\mathcal{C}}_n.
$$
 It follows from \eqref{imb}  that $\mathcal{C}_{n+1}\cap \mathcal{C}_{1,n}^+=\varnothing$ and consequently the abscissa of the single   point intersection  $\mathcal{C}_{n+1}\cap\, \mathcal{C}_1^+$  must satisfy $x_{n+1}> x_n.$
This proves that $\{x_n\}_{n\geq b_1^{-2}}$ is strictly increasing and thereby this sequence converges to some value $x_\star\leq  b_1$. Assume that $x_\star<b_1$ and define the subsequences 
$$
\{x_n^\pm\}_{n\geq b_1^{-2}}\triangleq \Big\{x_n; \lambda_n^\pm(x_n)=\lambda_1^+(x_n)\Big\},
$$
Clearly, one of the two sequences is infinite and assume first   that $\{x_n^+\}$ is infinite and up to an extraction this sequence converges also to $x_\star$ and for the simplicity we still denote this sequence by $\{x_n\}_{n\geq b_1^{-2}}.$
 Then from the definition of $\lambda_n^+$ we can easily check that 
\begin{align*}
\lim_{n\to+\infty}\lambda^+_n(x_n)&=\frac{1+(x_\star/b_1)^2}{2}+\frac{1-(x_\star/b_1)^2}{2}\\
&=1.
\end{align*}
On the other hand
\begin{align*}
\lim_{n\to+\infty}\lambda^+_1(x_n) &=1+x_\star^2-b_1^2.
\end{align*}
This is possible only if $x_\star=b_1$, which is a contradiction and thus $x_\star=b_1.$ Now in the case where only the sequence $\{x_n^-\}$ is infinite then we follow the same reasoning as before. We suppose  that  $x_\star<b_1$ and  one can verify that,
\begin{align*}
\lim_{n\to+\infty}\lambda^-_n(x_n)&=\big(\frac{x_\star}{b_1}\big)^2
\end{align*}
and 
\begin{align*}
\lim_{n\to+\infty}\lambda^+_1(x_n)=1+x_\star^2-b_1^2.
\end{align*}
By equating these numbers we obtain,
$$
(1-b_1^2)(x_\star^2-b_1^2)=0
$$
which is impossible since $b_1<1$ and consequently $x_\star=b_1$.
 Hence the proof of (v) is achieved.

\end{proof}
\subsection{Bifurcation for $m\geq1$}
Now we shall see how to implement the preceding results to prove Theorem \ref{thmV02} and Theorem \ref{thm-1fold}  by using Crandall-Rabinowitz theorem.  The proofs will be broken in several steps. First, we  introduce the spaces of bifurcation which capture  the m-fold symmetry and they  are of H\"{o}lderian  type. Second, we rewrite Proposition  \ref{regg1} dealing with the regularity of the nonlinear functional defining the $V$-states in the new setting. We end this section with the proofs of the properties of the linearized operator around the annulus required by C-R theorem.
\subsubsection{Function spaces} 
 We shall make use of the same spaces of \cite{H-F-M-V}. For  $m\geq1$, we  introduce the spaces $X_m$ and $ Y_m$ as follows:
$$
X_m=C^{1+\alpha}_{m}(\mathbb{T})\times C^{1+\alpha}_{m}(\mathbb{T}),
$$
where $C^{1+\alpha}_{m}(\mathbb{T})$ is the space of the $2\pi-$periodic functions $f\in C^{1+\alpha}(\mathbb{T})$ whose  Fourier series is given by 
$$
f(w)=\sum_{n=1}^{\infty}a_{n}\overline{w}^{nm-1},\quad w\in\mathbb{T},\quad a_n\in \mathbb{R}.
$$
This space is equipped with the usual strong  topology of $C^{1+\alpha}(\mathbb{T})$. We can easily see that $X_m$ is identified to 
\begin{equation}\label{X_m}
X_m=\Big\{f\in (C^{1+\alpha}(\mathbb{T}))^2,\, f(w)=\sum_{n=1}^{\infty}A_{n}\overline{w}^{nm-1},\, A_n\in \mathbb{R}^2  \Big\}.
\end{equation}
We define the ball of radius $r\in(0,1)$ by
$$
B_r^m=\Big\{f\in (C^{1+\alpha}_m(\mathbb{T}))^2, \,\|f\|_{C^{1+\alpha}(\mathbb{T})}< r\Big\}.
$$
 Take  $(f_1,f_2)\in B_r^m$ then  the expansions of the associated conformal mappings $\phi_1, \phi_2$ in the exterior 
 unit disc $\big\{w \in \CC; \, \vert w\vert > 1\big\}$ are given successively  by
$$
\phi_1(w)=b_1w+f_1(w)=w\Big(b_1+\sum_{n=1}^{\infty}\frac{a_{1,n}}{w^{nm}}\Big)
$$
and
$$
\phi_2(w)=b_2w+f_2(w)=w\Big(b_2+\sum_{n=1}^{\infty}\frac{a_{2,n}}{w^{nm}}\Big).
$$
This captures the m$-$fold symmetry of the associated boundaries $\phi_1(\mathbb{T})$ and $\phi_2(\mathbb{T})$ , via the relation
\begin{equation}\label{mfo1}
\phi_j\big(e^{2i\pi/m}w\big)=e^{2i\pi/m}\phi_j(w),\quad j=1,2 \quad\textnormal{and}\quad w\in \mathbb{T}.
\end{equation}
Set
\begin{equation}\label{Y_m}
Y_m=\Big\{g\in \big(C^{\alpha}(\mathbb{T})\big)^2,\, g=\sum_{n\geq 1}C_{n}\,e_{nm},\, C_n\in\RR^2\Big\}.
\end{equation}
With the help of  Proposition \ref{regg1} we deduce that the functional  $G=(G_1,G_2)$ is well-defined and smooth from $\RR\times B_r^m$ to $Y_m$ with $r$ small enough. The only thing that one should care about, which has been already discussed in the simply-connected case, is the persistence  of the symmetry which comes from the  rotational invariance of the functional $G$. As the proofs are very close to the simply-connected case without any substantial difficulties we prefer  to skip them an  state only  the desired results.

\begin{proposition}\label{regg17}
Let $b\in]0,1[$ and $0<r <\min(b,1-b)$, then the following holds true. 
\begin{enumerate}
\item $G: \RR\times B_r^m\to Y_m$ is  $C^1$ $($it is in fact $C^\infty)$.
\item The partial derivative  $\partial_\lambda DG: \mathbb{R}\times B_r^m\to \mathcal{L}(X_m, Y_m)$  exists and is continuous $($it is in fact $C^\infty).$
 \end{enumerate}
\end{proposition}
\vspace{0,3cm}

Now using  \eqref{dff} and \eqref{matrix} we deduce  that the restriction of $DG(\lambda,0)$ to the space $X_m$ leads to a well-defined continuous  operator
$DG(\lambda,0):X_m\to Y_m.$ It  takes the form,

\begin{eqnarray}\label{dff00}
DG(\lambda,0)\big(h_1,h_2\big)=\sum_{n\geq1}M_{nm}(\lambda)\left( \begin{array}{c}
a_{1,n} \\
a_{2,n}
\end{array} \right)\, e_{nm},
\end{eqnarray}
with $(h_1,h_2)\in X_m$ having the expansion
$$
h_j(w)=\sum_{n\geq 1}a_{j,n} \overline{w}^{nm-1}$$
and  the matrix $M_n$ is given for $n\geq1$  by
\begin{equation}\label{matrix00}
M_{n}(\lambda)\triangleq\begin{pmatrix}
b_1\Big[n\lambda-1+b_1^{2n}-n\Big(\dfrac{b_2}{b_1}\Big)^2\Big] & b_2\Big[\Big(\dfrac{b_2}{b_1}\Big)^{n}-(b_1b_2)^{n}\Big] \\
 -b_1\Big[\Big(\dfrac{b_2}{b_1}\Big)^{n}-(b_1b_2)^{n}\Big] & b_2\Big[n\lambda-n+1-{b_2^{2n}}\Big]
\end{pmatrix}.
\end{equation}
\subsubsection{Proof of Theorem $\ref{thmV02}$}\label{TT23}
The main goal of this section is to prove Theorem \ref{thmV02}. This will be an immediate consequence  of  Crandall-Rabinowitz theorem as soon as we check its  conditions which require a careful study. Concerning the regularity assumptions, they  were  discussed in  Proposition \ref{regg17}. As to  the properties required for the linearized operator, they are the object of  following proposition.
\begin{proposition}\label{propspec10}
Let $0 <b_2<b_1 <1$ and set $b\triangleq\frac{b_2}{b_1}$.  Let  $m\geq2$ satisfying
 $$
m\geq\frac{2+2b^m-(b_1^m+b_2^m)^2}{1-b^2}\cdot
$$Then the following results hold true.
\begin{enumerate}
\item 
The kernel of $DG(\lambda_m^\pm,0)$ is one-dimensional  and generated by the vector
$$ v_m(w)=\left( \begin{array}{c}
  b_2\Big[m\lambda_m^\pm-m+1-{b_2^{2m}}\Big] \\
b_1\Big[b^{m}-(b_1b_2)^{m}\Big]
\end{array} \right)\overline{w}^{m-1}.
$$
\item The range of $DG(\lambda_m^\pm,0)$ is closed and  of co-dimension one.

\item Transversality assumption:  the condition
$$
\partial_\lambda DG(\lambda_m^\pm,0) v_m\notin R(DG(\lambda_m^\pm,0))
$$
is satisfied if and only if
$$
m>\frac{2+2b^m-(b_1^m+b_2^m)^2}{1-b^2}\cdot
$$
\end{enumerate}
\end{proposition}
\begin{proof}
${(\hbox{i})}$ According to  \eqref{det1} the positivity of the discriminant $\Delta_n$ that guarantees the existence of real eigenvalues is equivalent for $m\geq2$ to 
$$
m\geq \frac{2+2b^m-(b_1^m+b_2^m)^2}{1-b^2}\cdot
$$
To prove that the kernel of $DG(\lambda_m^\pm,0)$ is one-dimensional it suffices to check that for $n\geq2$ the matrix $M_{nm}(\lambda_m^\pm)$ defined in \eqref{matrix00} is invertible. This follows from Lemma \ref{lem-mono} which asserts that $\lambda_{nm}^\pm\neq\lambda_m^\pm$ for $n\geq2$ and therefore 
$$
\hbox{det } M_{nm}(\lambda_m^\pm)\neq0.
$$
To get  a generator for the kernel it suffices to take an orthogonal vector to the second row of the matrix $M_m(\lambda_m^\pm)$.
\vspace{0,3cm}

${(\hbox{ii})}$ We are going to show that for any $m \geq 2$ the range $R\big(DG(\lambda_m^\pm,0)\big)$ coincides with the following subspace
\begin{equation}\label{g1g2}
\mathcal{Z}_m\triangleq \Big\{ g\in Y_m, g(w)=\sum_{n\geq 1}C_{n}\,e_{nm}, C_1\in R(M_m)\quad \hbox{and}\quad\forall n\geq2\,,C_n\in\RR^2\Big\}.
\end{equation}
Assume for a while this result,  then it is easy to check that  $R\big(DG(\lambda_m^\pm,0)\big)$  is closed in $Y_m$ and is of co-dimension one. Now to get the description of the range  we first observe that from \eqref{dff00} and \eqref{matrix00} the range is  included in the space $\mathcal{Z}_m.$ Therefore what is left is  to check is the inclusion $\mathcal{Z}_m\subset R\big(DG(\lambda_m^\pm,0)\big)$. Take $g=(g_1,g_2)\in \mathcal{Z}_m$ with the form
$$
g_j(w)=\sum_{n\geq1}c_{j,n} e_{nm}
$$
and let us  prove that the equation
$$
 DG(\lambda_m^\pm,0)h=g
$$
admits a solution $h=(h_1,h_2)$ in the space $X_m.$ Note that $h_j$ has the following structure,
$$
h_j(w)=\sum_{n\geq1}a_{j,n} \overline{w}^{nm-1}.
$$
According to \eqref{dff00}, the preceding equation  is equivalent to
$$
M_{mn}\left( \begin{array}{c}
a_{1,n} \\
a_{2,n}
\end{array} \right)=\left( \begin{array}{c}
c_{1,n} \\
c_{2,n}
\end{array} \right),\quad \forall n\geq1.
$$
For $n=1$, this equation is fulfilled because from the definition of $\mathcal{Z}_m$ we assume that the vector  $C_1\triangleq \left( \begin{array}{c}
c_{1,n} \\
c_{2,n}
\end{array} \right)$ belongs to the range of the matrix $M_m.$ With regard to  $n\geq2$, we use the fact that $M_{nm}$ is invertible and therefore   the sequences $(a_{j,n})_{n\geq2}$  are uniquely determined by  the formulae
\begin{equation}\label{m1}
\left( \begin{array}{c}
a_{1,n} \\
a_{2,n}
\end{array} \right)=M_{nm}^{-1}\left( \begin{array}{c}
c_{1,n} \\
c_{2,n}
\end{array} \right),\quad n\geq2.
\end{equation}
By computing the matrix $M_{mn}^{-1}(\lambda_m^\pm)$  we deduce that for all $n\geq2,$  
\begin{eqnarray}\label{invert-coef}
a_{1,n}&=&\frac{b_2\big[nm(\lambda_m^\pm-1)+1-b_2^{2nm}\big]}{\textnormal{det}\big({M}_{nm}(\lambda_m^\pm)\big)}\,c_{1,n}-\frac{b_2\big[(\frac{b_2}{b_1})^{nm}-(b_1 b_2)^{nm}\big]}{\textnormal{det}\big({M}_{nm}(\lambda_m^\pm)\big)}\,c_{2,n}
\end{eqnarray}
and
\begin{eqnarray*}
a_{2,n}&=&\frac{b_1\big[(\frac{b_2}{b_1})^{nm}-(b_1 b_2)^{nm}\big]}{\textnormal{det}\big({M}_{nm}(\lambda_m^\pm)\big)}\,c_{1,n}+\frac{b_1\big[nm(\lambda_m^\pm-(\frac{b_2}{b_1})^2)-1+b_1^{2nm}\big]}{\textnormal{det}\big({M}_{nm}(\lambda_m^\pm)\big)}\,c_{2,n}.
\end{eqnarray*}
Hence  the proof of  $(h_1,h_2)\in X_m$ amounts to showing that  
$$
w\mapsto \left( \begin{array}{c}
h_1(w)-a_{1,1}\overline{w}^{m-1} \\
h_2(w)-a_{2,1}\overline{w}^{m-1}
\end{array} \right)\in C^{1+\alpha}(\mathbb{T})\times C^{1+\alpha}(\mathbb{T}).
$$
We shall develop the computations only for the first component and the second one can be done in a similar way. Notice that $\textnormal{det}\big({M}_{nm}(\lambda_m^\pm))$
 does not vanish for $n\geq2$ and behaves for large $n$ like
 $$
\textnormal{det}\big({M}_{nm}(\lambda_m^\pm))= b_1 b_2m^2(\lambda_m^\pm-1)\big[\lambda_m^\pm-({b_2}/{b_1})^2\big] n^2+b_1 b_2m\big(1-(b_2/b_1)^2\big)n-1+o(1).
$$ 
Since   $\lambda_m^\pm\notin\big\{1, (b_2/b_1)^2\big\}$, then by Taylor expansion we get
$$
 a_{1,n}=\frac{1}{b_1m(\lambda_m^\pm-(\frac{b_2}{b_1})^2)} \frac{c_{1,n}}{n}+\gamma_{1,n}c_{1,n}+ \gamma_{2,n}c_{2,n} $$
 with
 $$
|\gamma_{j,n}|\le \frac{C}{n^2}\cdot
$$
Set ${\displaystyle \widetilde{h}_1(w)=h_1(w)-a_{1,1}\overline{w}^{m-1}}$ and define the functions
$$
K_{j}(w)=\sum_{n\geq2}n\gamma_{j,n} \overline{w}^{nm}, \quad \quad \widetilde{g}_j=\sum_{n\geq2} \frac{c_{j,n}}{n} e_{nm}.
$$
Then one can check that,
\begin{equation}\label{exp11}
\overline{w}\widetilde{h}_1(w)=\frac{1}{ mb_1(\lambda_m^\pm-(\frac{b_2}{b_1})^2)}\sum_{n\geq2}\frac{c_{1,n}}{n}\overline{w}^{nm}+\{K_1\star(\Pi\,\widetilde g_1)\}(w)+\{K_2\star(\Pi\,\widetilde g_2)\}(w).
\end{equation}
The convolution  is understood in the usual one: for two continuous  functions $f,g;\mathbb{T}\to \CC$ we define
$$
\forall w\in \mathbb{T},\quad f*g(w)=\fint_{\mathbb{T}}f(\tau) g({\tau}\overline{w})\frac{d\tau}{\tau}\cdot
$$
The notation $\Pi$ is used for Szeg\"{o} projection defined by
$$
\Pi(\sum_{n\in \mathbb{Z}} c_n w^n)=\sum_{n\in -\mathbb{N}}c_n w^n,
$$
which  acts continuously  on $C^{1+\alpha}(\mathbb{T})$. 
One can easily see that the first term  in the right-hand side of \eqref{exp11} belongs to $C^{1+\alpha}(\mathbb{T})$. With regard to the last two terms,
 note that  $K_j\in L^2(\mathbb{T})\subset  L^1(\mathbb{T})$ and $\widetilde g_j\in C^{1+\alpha}(\mathbb{T})$, then using   the classical convolution law $L^1(\mathbb{T})*C^{1+\alpha}(\mathbb{T})\to C^{1+\alpha}(\mathbb{T})$ combined with  the continuity of  $\Pi$  we deduce that  those terms belong to $C^{1+\alpha}(\mathbb{T})$ and the function $w\mapsto \overline{w}\widetilde{h}_1(w)$ belongs to this space too. This achieves the proof of the range \mbox{of $DG(\lambda_m^\pm,0)$.}
\vspace{0,3cm}

${(\hbox{iii})}$
Recall from the part $(\hbox{i})$  that  the kernel of $DG(\lambda_m^\pm,0)$ is one-dimensional and generated by the vector $v_m$ defined by
$$
w\in \mathbb{T}\mapsto v_m(w)=\left( \begin{array}{c}
  b_2\Big[m\lambda_m^\pm-m+1-{b_2^{2m}}\Big] \\
b_1\Big[\Big(\dfrac{b_2}{b_1}\Big)^{m}-(b_1b_2)^{m}\Big]
\end{array} \right)\overline{w}^{m-1}.
$$
We shall prove that
$$
\partial_\lambda DG(\lambda_m^\pm,0)v_m\not\in R\Big( DG(\lambda_m^\pm,0)\Big)
$$
if and only if $\lambda_m^+\neq\lambda_m^-$,  which is equivalent  to,
$$
m>\frac{2+2b^m-(b_1^m+b_2^m)^2}{1-b^2}\cdot
$$
Let  $(h_1,h_2)\in X_m$ with the following expansion
$$
h_j(w)=\sum_{n\geq1}a_{j,n} \overline{w}^{nm-1}.
$$
Then differentiating \eqref{dff00} with respect to $\lambda$ we get
\begin{eqnarray}\label{ddff}
\partial_\lambda DG(\lambda,0)(h_1,h_2)=m\sum_{n\geq 1}n\left( \begin{array}{c}
b_1a_{1,n} \\
b_2a_{2,n}
\end{array} \right)\,e_{nm}.
\end{eqnarray}
Hence, we get
\begin{eqnarray*}
\partial_\lambda DG(\lambda_m^\pm,0)v_m&=& mb_1b_2\left( \begin{array}{c}
m\lambda_m^\pm-m+1-{b_2^{2m}} \\
\Big(\dfrac{b_2}{b_1}\Big)^{m}-(b_1b_2)^{m}
\end{array} \right)\, e_m\\ &\triangleq&mb_1b_2\,\mathbb{W}_m\, e_m.
\end{eqnarray*}
This pair of functions is in the range of $DG(\lambda_m^\pm,0)$ if and only if the vector $\mathbb{W}_m$ is a scalar multiple of the second  column of the matrix $M_m(\lambda_m^\pm)$  defined by \eqref{matrix00}. This happens if and only if 
\begin{equation}\label{ction}
\Big(m\lambda_m^\pm-m+1-{b_2^{2m}}\Big)^2-\Big(\Big(\dfrac{b_2}{b_1}\Big)^{m}-(b_1b_2)^{m}\Big)^2=0.
\end{equation}
Combining this equation with det $ M_m=0$ we find
$$
\Big(m\lambda_m^\pm-m+1-{b_2^{2m}}\Big)^2+\Big(m\lambda_m^\pm-m+1-{b_2^{2m}}\Big)\Big(m\lambda_m^\pm-1+b_1^{2m}-m\Big(\dfrac{b_2}{b_1}\Big)^2\Big)=0,
$$
which is equivalent to
$$
\Big(m\lambda-m+1-{b_2^{2m}}\Big)\bigg(2m\lambda-m\Big(1+\Big(\dfrac{b_2}{b_1}\Big)^2\Big)-{b_2^{2m}}+b_1^{2m}\bigg)=0.
$$
Thus we find that
$$
m\lambda_m^\pm-m+1-{b_2^{2m}}=0\quad \textnormal{or}\quad 2m\lambda_m^\pm-m\Big(1+\Big(\dfrac{b_2}{b_1}\Big)^2\Big)-{b_2^{2m}}+b_1^{2m}=0.
$$
The first possibility is excluded by \eqref{ction} and the second one corresponds to a multiple eigenvalue condition: $\lambda_m^+=\lambda_m^-,$ that is, $\Delta_m=0.$
This completes the proof of Proposition \ref{propspec10}.

\end{proof}
\subsubsection{Proof of Theorem $\ref{thm-1fold}$}
Our next purpose is to study the bifurcation of $1$-fold rotating patches. 
Recall from Section \ref{first11} that for $m=1$ there are two different  eigenvalues given by
$$
\lambda_1^-=(b_2/b_1)^2\quad \hbox{and}\quad \lambda_1^+=1+b_2^2-b_1^2.
$$
In that section we observed   significant difference in their behaviors and we shall see next how this fact does affect the bifurcation problem. It appears that the bifurcation with  $\lambda_1^-$ is   very complicate due to the range of the linearized operator which is of infinite co-dimension. Nevertheless,  with  $\lambda_1^+$ the situation is actually more tractable  and the bifurcation occurs frequently. Before stating the basic results of this section, we need to make some  notation. 
Let $b_1\in ]0,1[$ being a fixed real number and define the set
$$
\mathcal{E}_{b_1}\triangleq\Big\{ b_2\in]0,b_1[; \exists m\geq2\,\, s.t.\,\, P_m(\lambda_1^+)=0\Big\}.
$$
The polynomial $P_m$ was defined in \eqref{discr25}, which is up to a factor  the characteristic polynomial of the matrix $M_m(\lambda).$
The set $\mathcal{E}_{b_1}$ corresponds to the abscissa of the points of intersection between the collection  of the curves $\{\mathcal{C}_m, m\geq2\}$ and $\mathcal{C}_1^+$ which were defined in \eqref{curvn}. 
Recall from Proposition \ref{prop-one}-(ii)-(iii)that for each $m\geq2$ there is at most   one value $x_m$ of $b_2$ such that   $P_m(\lambda_1^+)=0$. Moreover, the sequence $(x_m)_{m\geq b_1^{-2}}$ is strictly increasing and  converges to $b_1.$
Now we will  prove the following result.
\begin{proposition}\label{propspec11}
The following assertions hold true.
\begin{enumerate}
\item The range  of $DG(\lambda_1^-,0)$ has an infinite co-dimension.
\item  If $b_2\in \mathcal{E}_{b_1}$, then the kernel of  $DG(\lambda_1^+,0)$ is two-dimensional and generated by the vectors $v_1=\left( \begin{array}{c}
 1 \\
1\end{array} \right)$ and $v_m$ of Proposition $\ref{propspec10}$, with  $m\geq2$ being  the only  integer such that $P_m(\lambda_1^+)=0.$  
In addition, the range of $DG(\lambda_1^+,0)$ is closed and has a co-dimension two.
\item If $b_2\notin \mathcal{E}_{b_1},$ then the kernel of $DG(\lambda_1^+,0)$ is one-dimensional  and it is generated by the vector
$ v_1$ seen before.
Furthermore, the range of $DG(\lambda_1^+,0)$ has  a co-dimension one and the transversality assumption is satisfied,
$$
\partial_\lambda DG(\lambda_1^+,0) v_1\notin R(DG(\lambda_1^+,0)).
$$

\end{enumerate}
\end{proposition}
\begin{proof}
$({\hbox{i}})$ According to \eqref{matrix00}, we obtain
\begin{equation*}
M_{n}(\lambda_1^-)\triangleq\begin{pmatrix}
b_1\big[-1+b_1^{2n}\big] & b_2\Big[\Big(\dfrac{b_2}{b_1}\Big)^{n}-(b_1b_2)^{n}\Big] \\
 -b_1\Big[\Big(\dfrac{b_2}{b_1}\Big)^{n}-(b_1b_2)^{n}\Big] & b_2\Big[n\Big(\Big(\dfrac{b_2}{b_1}\Big)^{n}-1\Big)+1-{b_2^{2n}}\Big]
\end{pmatrix}.
\end{equation*}
In this case we get that the determinant of $M_n(\lambda_1^-)$ behaves for large $n$ like $b_1 b_2\,  n$. Consequently we deduce from \eqref{invert-coef} the existence of $\alpha\neq0$ such that 
$$
a_{1,n}=\alpha c_{1,n}+ o(1),
$$
which means that the pre-image of an element of $Y_m$ by $DG(\lambda_1^-,0)$  is not in general more better than $C^\alpha(\mathbb{T})$. This implies that the range of the linearized operator is of infinite co-dimension. It follows that  one important condition of C-R theorem is violated and therefore the  bifurcation in this special case still unsolved.
\vspace{0,3cm}

$({\hbox{ii}})$ Let $b_2\in \mathcal{E}_{b_1}$. Then by definition there exists $m\geq2$ such that $P_m(\lambda_1^+)=0.$ This means that $\lambda_1^+$ coincides with one of the two numbers $\lambda_m^\pm.$ Therefore the kernel of $DG(\lambda_1^+,0)$ is given by the two dimensional vector space
$$
\hbox{Ker} DG(\lambda_1^+,0)= \hbox{Ker} M_1(\lambda_1^+)\oplus\hbox{Ker} M_m(\lambda_1^+)\overline{w}^{m-1}.
$$
Easy computations give that the expression
\begin{equation*}
M_{1}(\lambda_1^+)=b_2(1-b_1^2)\begin{pmatrix}
- \dfrac{b_2}{b_1}& \dfrac{b_2}{b_1} \\
 -1 &1
\end{pmatrix}.
\end{equation*}
Obviously the kernel of $M_1(\lambda_1^+)$ is spanned by the vector $ v_1=\left( \begin{array}{c}
 1 \\
1\end{array} \right).
$
However we know that $\hbox{Ker} M_m(\lambda_1^+)$ is spanned by the vector $v_m$ already seen in Proposition \ref{propspec10}.
To prove that the range is of co-dimension two we follow the same lines of Proposition \ref{propspec10} bearing in mind that  the determinant of  $M_n(\lambda_1^+)$ behaves for large $n$ like $c n^2$ with $c\neq 0$. We skip the details  which are left to the reader.
\vspace{0,3cm}

$({\hbox{iii}})$ Let $b_2\notin \mathcal{E}_{b_1},$ then $P_m(\lambda_1^+)$ does not vanish for any $m\geq2$. This means that the matrix $M_m(\lambda_1^+)$ is invertible  and therefore the kernel of $DG(\lambda_1^+,0)$ is one-dimensional  and given by 
$$
\hbox{Ker} DG(\lambda_1^+,0)= \hbox{Ker} M_1(\lambda_1^+)=\langle v_1\rangle.
$$
Similarly to the  Proposition  \ref{propspec10} we get that the range is of co-dimension one. In addition, the transversality condition is satisfied  since the eigenvalue $\lambda_1^+$ is simple ($\lambda_1^+\neq\lambda_1^-$) as it has been  discussed   in the proof of the point  $({\hbox{iii}})$ of Proposition \ref{propspec10}. The proof of Proposition \ref{propspec11} is now achieved and the result of Theorem \ref{thm-1fold} follows.

\end{proof}





%

%


\section{Numerical experiments} 

In order to obtain the $V$-states, we follow a similar procedure to that in \cite{H-F-M-V} and \cite{H-H-H}; therefore, we shall omit some details, which can be consulted in those references.

\subsection{Simply-connected $V$-states} \label{numsc}

\subsubsection{Numerical obtention}

Given a simply-connected domain $D$ with boundary $z(\theta)$, where $\theta\in[0,2\pi[$ is the Lagrangian parameter, and $z$ is counterclockwise parameterized, the condition of $D$ being a $V$-state rotating  with angular velocity $\Omega$ is given by \eqref{rotsq}, i.e.,
\begin{equation}
\label{e:condition0}
\begin{split}
\operatorname{Re}\Bigg\{\Bigg(2\Omega \overline{z(\theta)} & + \frac{1}{2\pi i}\int_0^{2\pi}\frac{\overline{z(\theta) - z(\phi)}}{z(\theta)-z(\phi)}z_\phi(\phi)d\phi
    \cr
& - \frac{1}{2\pi i}\int_0^{2\pi}\frac{|z(\phi)|^2}{1-z(\theta)z(\phi)}z_\phi(\phi)d\phi\Bigg) z_\theta(\theta)\Bigg\}=0.
\end{split}
\end{equation}

\noindent As in \cite{H-F-M-V} and \cite{H-H-H}, we use a pseudo-spectral method to find $m$-fold $V$-states from \eqref{e:condition0}. We discretize $\theta\in[0,2\pi[$ in $N$ equally spaced nodes $\theta_i = 2\pi i/N$, $i = 0, 1, \ldots, N-1$. Observe that the integrand in the first integral in \eqref{e:condition0} satisfies that
\begin{equation}
\label{e:limit0}
\lim_{\phi\to\theta}\frac{\overline{z(\theta) - z(\phi)}}{z(\theta)-z(\phi)}\Bigg|_{\theta = \phi} = \frac{\overline{z_\theta(\theta)}}{z_\theta(\theta)}.
\end{equation}

\noindent Therefore, bearing in mind \eqref{e:limit0}, we can evaluate numerically with spectral accuracy the integrals in \eqref{e:condition0} at a node $\theta = \theta_i$ by means of the trapezoidal rule, provided that $N$ is large enough:
\begin{equation}
\label{e:trapezoidal}
\begin{split}
\frac{1}{2\pi}\int_0^{2\pi}\frac{\overline{z(\theta_i) - z(\phi_j)}}{z(\theta_i)-z(\phi_j)}z_\phi(\phi_j)d\phi
& \approx \frac{1}{N}\Bigg(\overline{z_\theta(\theta_i)} + \mathop{\sum_{j = 0}^{N-1}}_{j\not=i}\frac{\overline{z(\theta_i) - z(\phi_j)}}{z(\theta_i)-z(\phi_j)}z_\phi(\phi_j)\Bigg),
    \\
\frac{1}{2\pi}\int_0^{2\pi}\frac{|z(\phi)|^2}{1-z(\theta_i)z(\phi)}z_\phi(\phi)d\phi
& \approx \frac{1}{N}\sum_{j = 0}^{N-1}\frac{|z(\phi_j)|^2}{1-z(\theta_i)z(\phi_j)}z_\phi(\phi_j).
\end{split}
\end{equation}

\noindent In order to obtain $m$-fold $V$-states, we approximate the boundary $z$ as
\begin{equation}
\label{e:z0cos}
z(\theta) = e^{i\theta}\left[b + \sum_{k = 1}^M a_{k}\cos(m\,k\,\theta)\right],
\end{equation}

\noindent where the mean radius is $b$; and we are imposing that $z(-\theta) = \bar z(\theta)$, i.e., we are looking for $V$-states symmetric with respect to the $x$-axis. For sampling purposes, $N$ has to be chosen such that $N \ge 2mM+1$; additionally, it is convenient to take $N$ a multiple of $m$, in order to be able to reduce the $N$-element discrete Fourier transforms to $N/m$-element discrete Fourier transforms. If we write $N = m2^r$, then $M = \lfloor (m2^r-1)/(2m)\rfloor = 2^{r-1}-1$.

We introduce \eqref{e:z0cos} into \eqref{e:condition0}, and approximate the error in \eqref{e:condition0} by an $M$-term sine expansion:
\begin{equation}
\label{e:V-State0conditions}
\begin{split}
\operatorname{Re}\Bigg\{\Bigg(2\Omega \overline{z(\theta)} & + \frac{1}{2\pi i}\int_0^{2\pi}\frac{\overline{z(\theta) - z(\phi)}}{z(\theta)-z(\phi)}z_\phi(\phi)d\phi
    \cr
& - \frac{1}{2\pi i}\int_0^{2\pi}\frac{|z(\phi)|^2}{1-z(\theta)z(\phi)}z_\phi(\phi)d\phi\Bigg) z_\theta(\theta)\Bigg\} \approx \sum_{k = 1}^M b_k\sin(m\,k\,\theta).
\end{split}
\end{equation}

\noindent This last expression can be represented in a very compact way as
\begin{equation}
\label{e:F0}
\mathcal F_{b,\Omega}(a_1, \ldots, a_M) = (b_1, \ldots, b_M),
\end{equation}

\noindent for a certain $\mathcal F_{b,\Omega} \ : \ \mathbb{R}^{M}\to\mathbb{R}^{M}$. Remark that, for any $\Omega$ and any $b\in]0,1[$, we have trivially $\mathcal F_{b,\Omega}(\mathbf 0) = \mathbf 0$, i.e., the circumference of radius $b$ is a solution of the problem. Therefore, the obtention of a simply-connected $V$-state is reduced to finding numerically a nontrivial root $(a_1, \ldots, a_M)$ of \eqref{e:F0}. To do so, we discretize the $(M\times M)$-dimensional Jacobian matrix $\mathcal J$ of $\mathcal F_{b,\Omega}$ using first-order approximations. Fixed $|h|\ll1$ (we have chosen $h = 10^{-10}$), we have that
\begin{equation}
\label{e:derivative0}
\frac{\partial \mathcal F_{b,\Omega}(a_1, \ldots, a_M)}{\partial a_1} \approx \frac{\mathcal F_{b,\Omega}(a_1 + h, \ldots, a_M) - \mathcal F_{b,\Omega}(a_1, \ldots, a_M)}{h}.
\end{equation}

\noindent Hence, the first $M$ coefficients of the sine expansion of \eqref{e:derivative0} form the first row of $\mathcal J$, and so on. Therefore, if the $n$-th iteration is denoted by $(a_1, \ldots, a_M)^{(n)}$, then the $(n+1)$-th iteration is given by
\begin{equation*}
(a_1, \ldots, a_M)^{(n+1)} = (a_1, \ldots, a_M)^{(n)} - \mathcal F_{b,\Omega}\left((a_1, \ldots, a_M)^{(n)}\right)\cdot [\mathcal J^{(n)}]^{-1},
\end{equation*}

\noindent where $[\mathcal J^{(n)}]^{-1}$ denotes the inverse of the Jacobian matrix at $(a_1, \ldots, a_M)^{(n)}$. This iteration converges in a small number of steps to a nontrivial root for a large variety of initial data $(a_1, \ldots, a_M)^{(0)}$. In particular, it is usually enough to perturb the unit circumference by assigning a small value to ${a_1}^{(0)}$, and leave the other coefficients equal to zero. Our stopping criterion is
\begin{equation*}
\max\left|\sum_{k = 1}^M b_k\sin(m\,k\,\theta)\right| < tol,
\end{equation*}

\noindent where $tol = 10^{-13}$. For the sake of coherence, we change eventually the sign of all the coefficients $\{a_k\}$, in order that, without loss of generality, $a_1>0$.

\subsubsection{Numerical discussion}

Given $m$ and $b$, Proposition \ref{propspec1} defines the value $\lambda_m$ at which we bifurcate from the circumference of radius $b$. Let us recall that $\lambda_m = 1 - 2\Omega_m$. Although working with $\lambda$ is more convenient from an analytical point of view, we use $\Omega = (1 - \lambda) / 2$ in the graphical representations of the $V$-states that follow, because $\Omega$ is a more natural parameter from a physical point of view. Therefore, we bifurcate at $\Omega_m = (m - 1 + b^{2m}) / (2m)$.
\begin{figure}[!htb]
\center
\includegraphics[width=0.5\textwidth, clip=true]{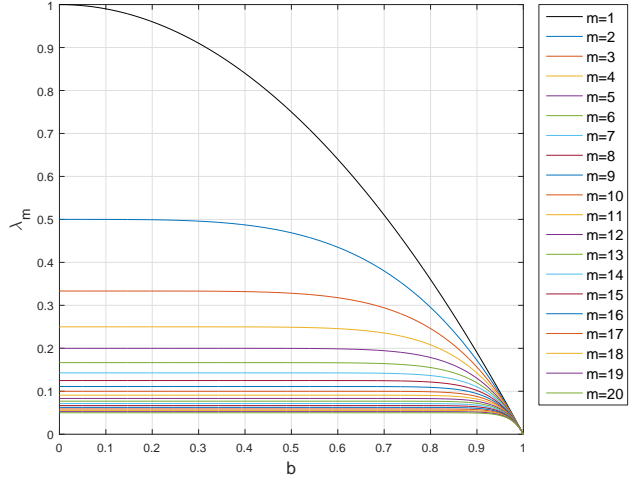}
\caption{$\lambda_m$ as a function of $b$, for $m = 1, \ldots, 20$.}
\label{f:blambdam}
\end{figure}

In Figure \ref{f:blambdam}, we have plotted $\lambda_m$ as a function of $b$, for $m = 1, \ldots, 20$. Figure \ref{f:blambdam} suggests that there are two different situations: $b$ \textit{close} to one, and $b$ \textit{not so close} to one; note that, in the latter case, the curves can be approximated by $\lambda_m \approx 1 / m$, i.e., $\Omega_m \approx (m - 1) / (2m)$, which is in agreement with \cite{DZ}.
\begin{figure}[!htb]
\center
\includegraphics[width=0.5\textwidth, clip=true]{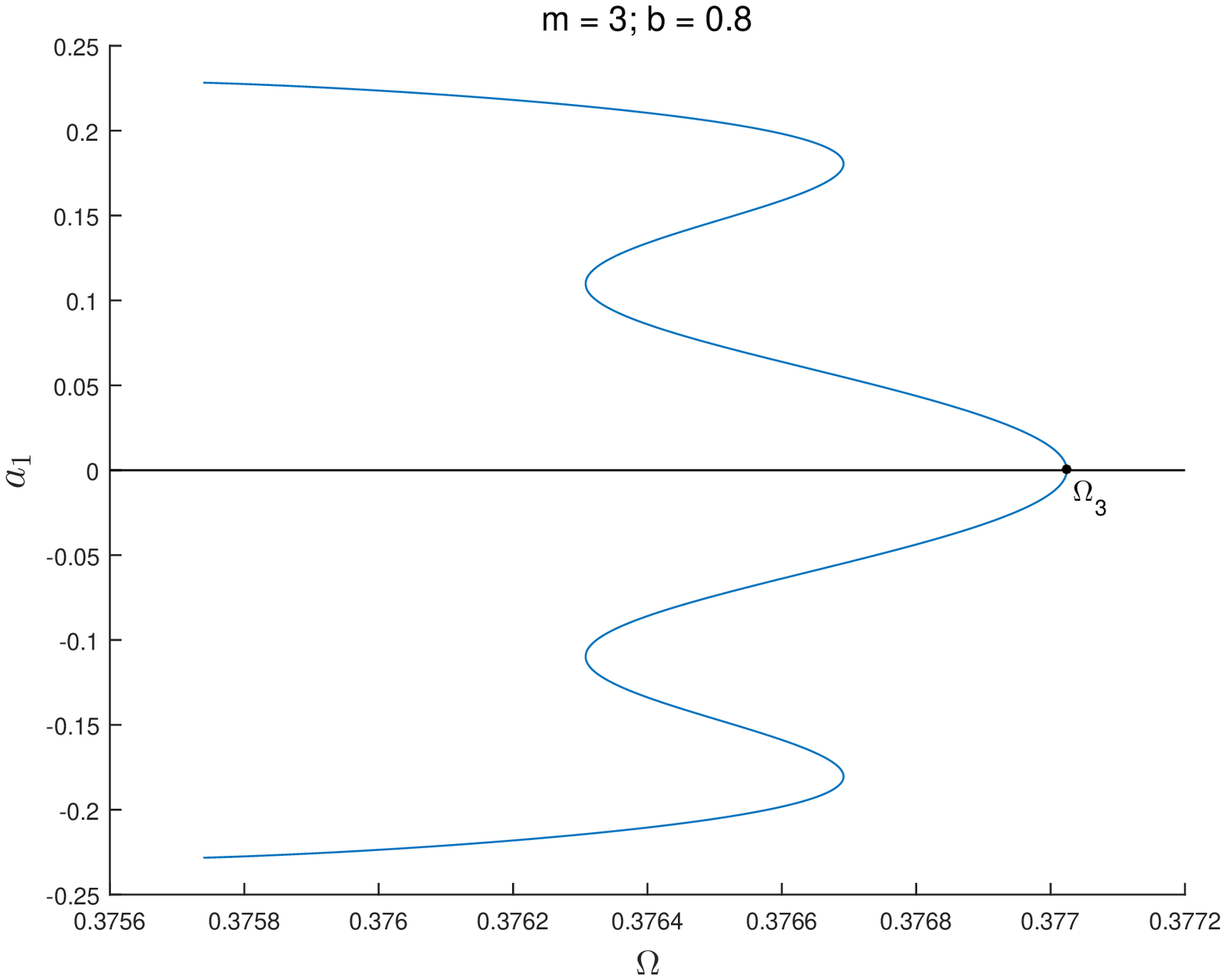}~
\includegraphics[width=0.5\textwidth, clip=true]{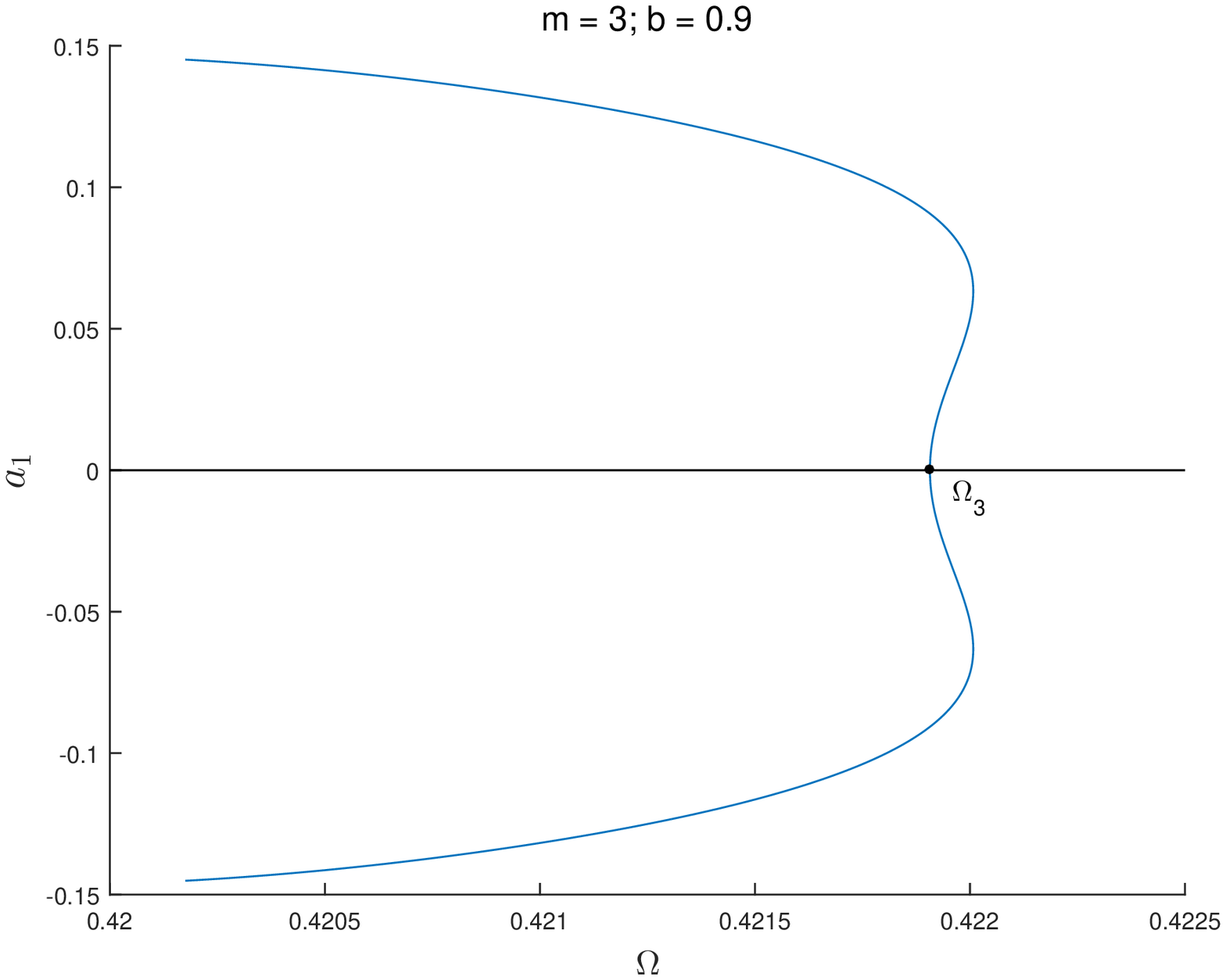}
\caption{Bifurcation diagrams corresponding to $m = 3$, $b = 0.8$ (left); and to $m = 3$, $b = 0.9$ (right). $N = 384$.}
\label{f:bifurcationm3b0_89}
\end{figure}

In order to illustrate how the shape of the simply-connected $V$-states depends on $b$, we consider the cases $1 \le m \le 4$; observe that everything said for $m = 3$ and $m = 4$ is valid for all $m \ge 3$. In general, fixed $m$ and $b$, we bifurcate  from the circumference with radius $b$ at $\Omega_m$. During the bifurcation process, there may be saddle-node bifurcation points \cite{Kil} appearing; in that case, we use the techniques described in \cite{H-H-H}. For instance, in Figure \ref{f:bifurcationm3b0_89}, we have plotted the bifurcation diagrams of the coefficient $a_1$ in \eqref{e:z0cos} against $\Omega$, for $m = 3$, $b = 0.8$ (left); and for $m = 3$, $b = 0.9$ (right). Note that, in the bifurcation diagrams, when starting to bifurcate at $\Omega_m$, we take sometimes $\Omega < \Omega_m$ (left), and other times $\Omega > \Omega_m$ (right), although the latter case may appear only when $b$ is \textit{large enough}. Note also that we may have several saddle-node bifurcation 
 points in the same bifurcation diagram, and, hence, more than two $V$-states corresponding to the same $\Omega$, and in the same bifurcation branch. For instance, the left-hand side of Figure \ref{f:bifurcationm3b0_89} tells us that there are three $V$-states corresponding to $m = 3$, $b = 0.8$, $\Omega = 0.3765$; which we have plotted in Figure \ref{f:Omega0.3765}.
\begin{figure}[!htb]
\center
\includegraphics[width=0.5\textwidth, clip=true]{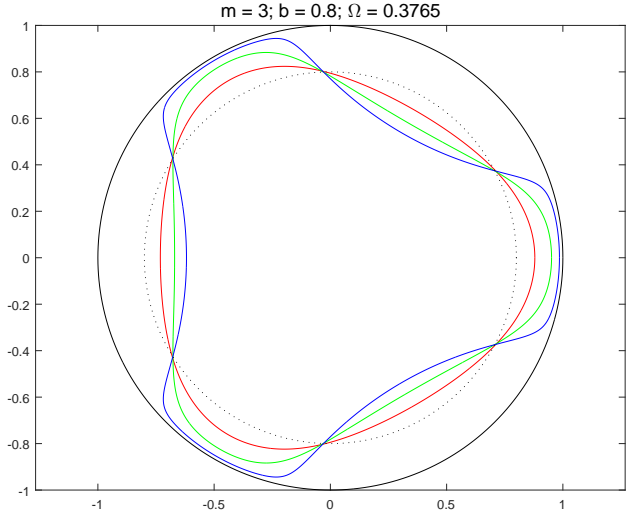}
\caption{$V$-states from the same bifurcation branch (left-hand side of Figure \ref{f:bifurcationm3b0_89}) corresponding to $m = 3$, $b = 0.8$, $\Omega = 0.3765$. $N = 768$.}
\label{f:Omega0.3765}
\end{figure}

We have approximated the limiting $V$-states occurring for $1\le m \le 4$, which are depicted in Figure \ref{f:limiting0holem1234}. Figure \ref{f:limiting0holem1234} confirms the observation on the size of $b$ made from Figure \ref{f:blambdam}. Loosely speaking, when $b$ is \textit{far enough} from one, the rigid boundary does not have any remarkable effect on the shape of the $V$-states. Take for instance the cases $m = 1$, $b = 0.4$; $m = 2$, $b = 0.4$; $m = 3$, $b = 0.6$; $m = 4$, $b = 0.7$: the approximations to the respective limiting $V$-states are clearly \textit{far away} from the unit circumference; whereas, in all the other cases, the distance to the unit circumference is smaller than $10^{-2}$. In fact, Figure \ref{f:limiting0holem1234} suggests that, from a certain $b$ on, we can obtain $V$-states arbitrarily close to the unit circumference, and that the limiting $V$-state is precisely the one whose distance to the unit circumference is zero in the limit. Moreover, as $b$
  grows towards one, the limiting $V$-states tend to cover an increasingly larger part of the unit circumference.

Continuing with Figure \ref{f:limiting0holem1234}, the cases $m = 1$ and $m = 2$ are pretty different from the other cases. Indeed, when $m \ge 3$ and $b$ is small enough, the limiting $V$-states resemble very much those in \cite{DZ}, and corner-shaped singularities seem to develop. It is remarkable that the rigid boundary only affects the shape of the $V$-states for $b$ pretty close to one; furthermore, the larger $m$ is, the larger $b$ has to be, in order that the influence of the rigid boundary becomes noticeable. On the other hand, when $m = 2$ and $b$ is small enough, the limiting $V$-states are infinity-shaped; whether some self-intersection actually occurs deserves further study. Finally, when $m = 1$ and $b$ is small enough, the limiting $V$-states seem to resemble an asymmetrical oval.
\begin{figure}[!htb]
\center
\includegraphics[width=0.5\textwidth, clip=true]{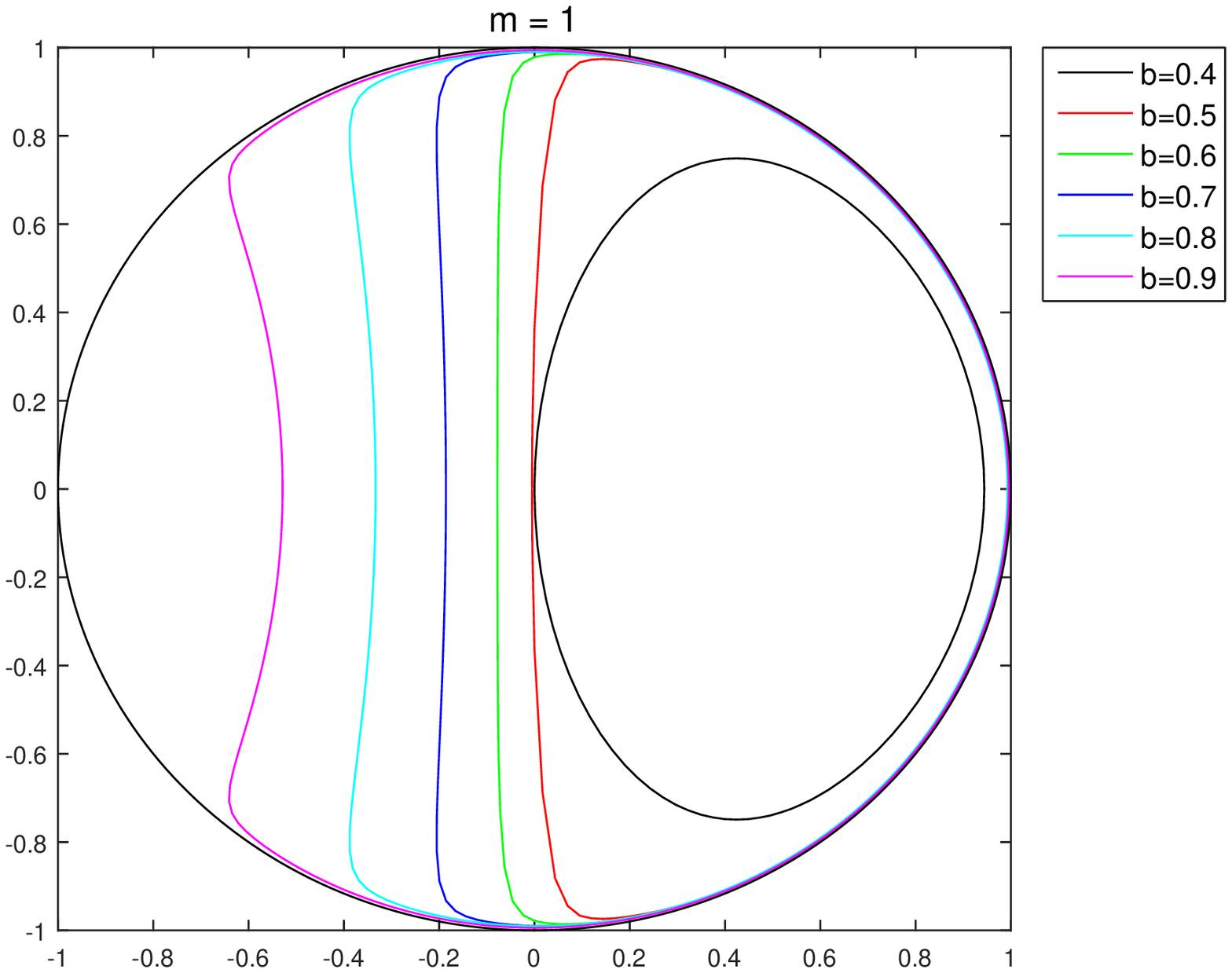}~
\includegraphics[width=0.5\textwidth, clip=true]{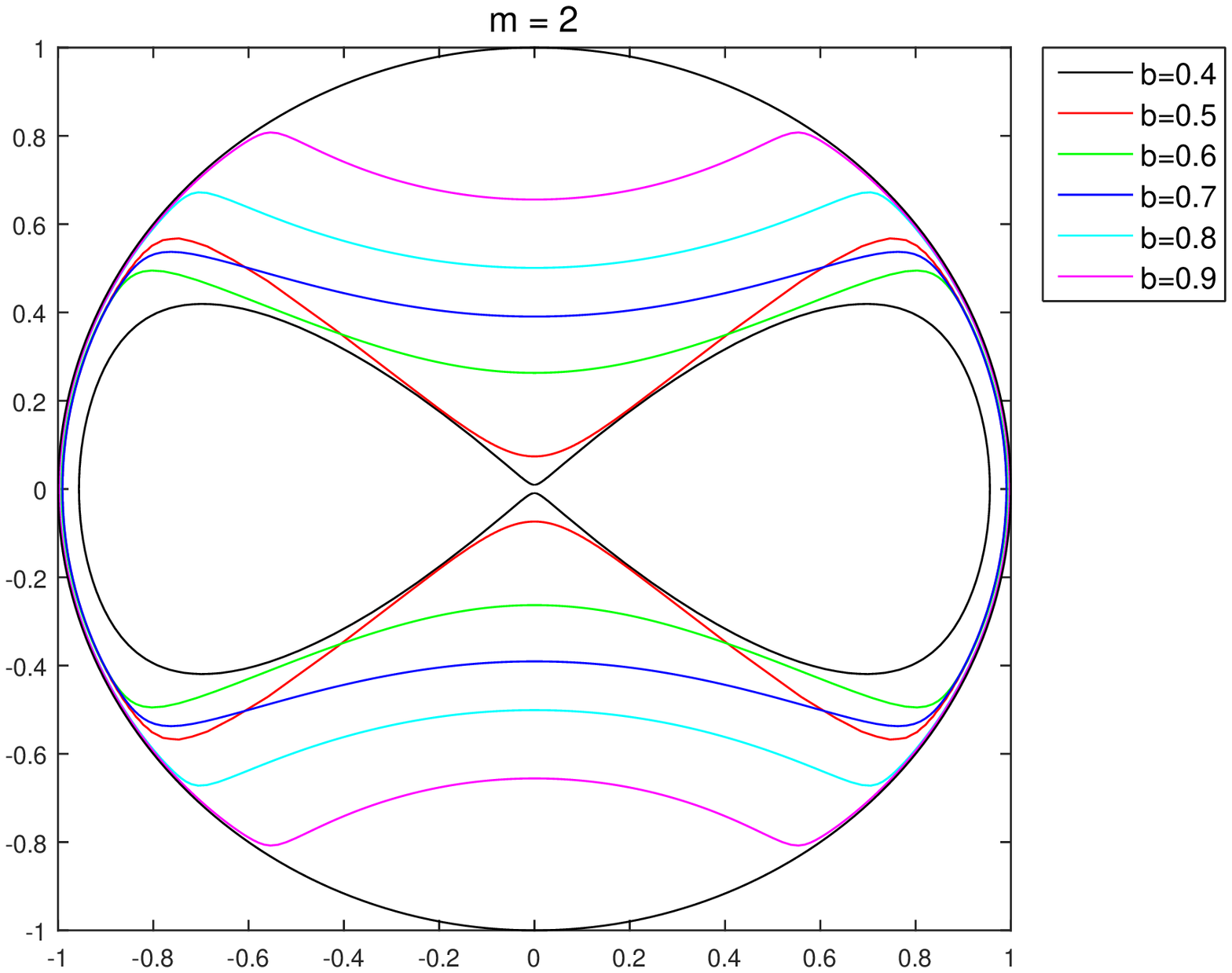}

\bigskip

\includegraphics[width=0.5\textwidth, clip=true]{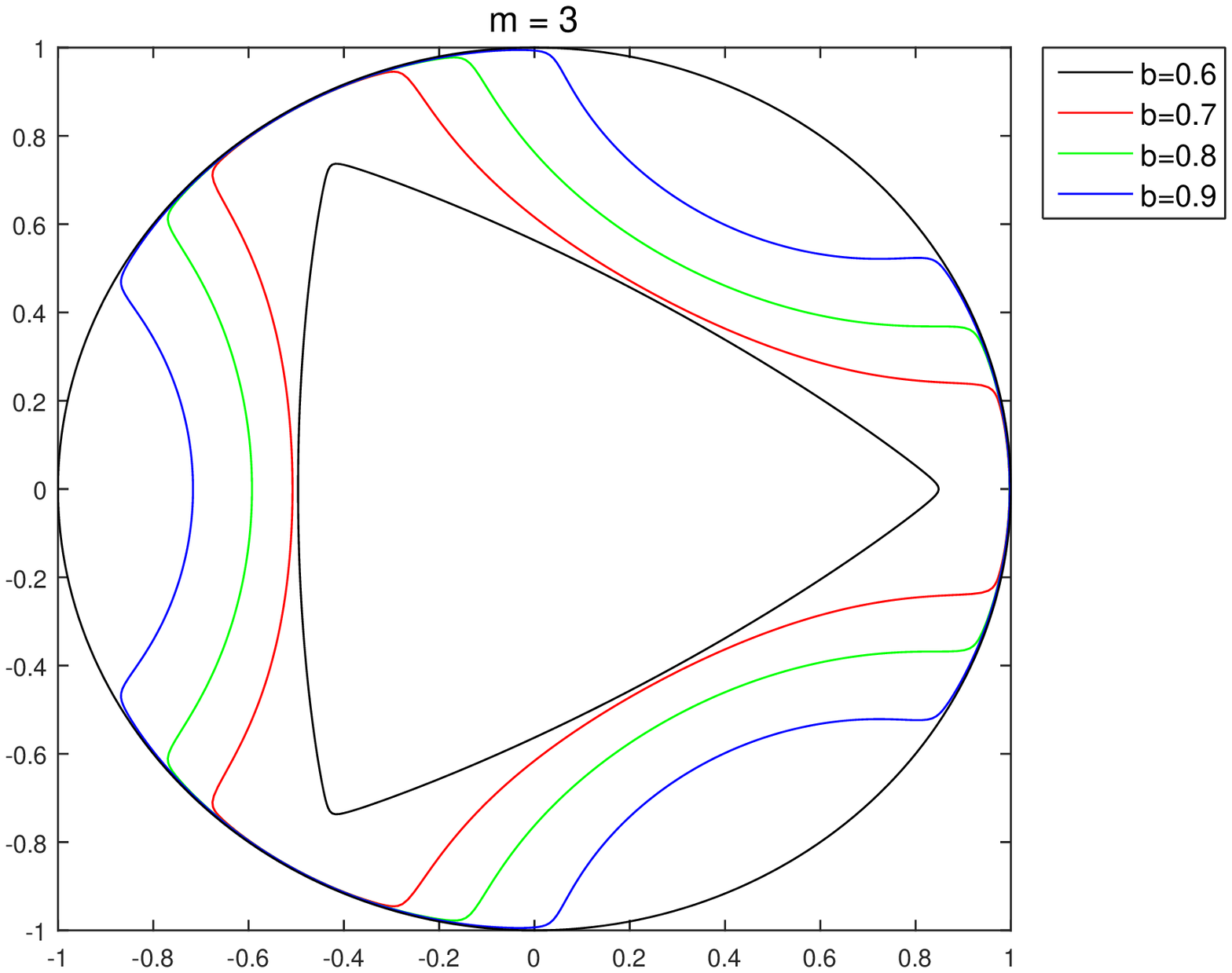}~
\includegraphics[width=0.5\textwidth, clip=true]{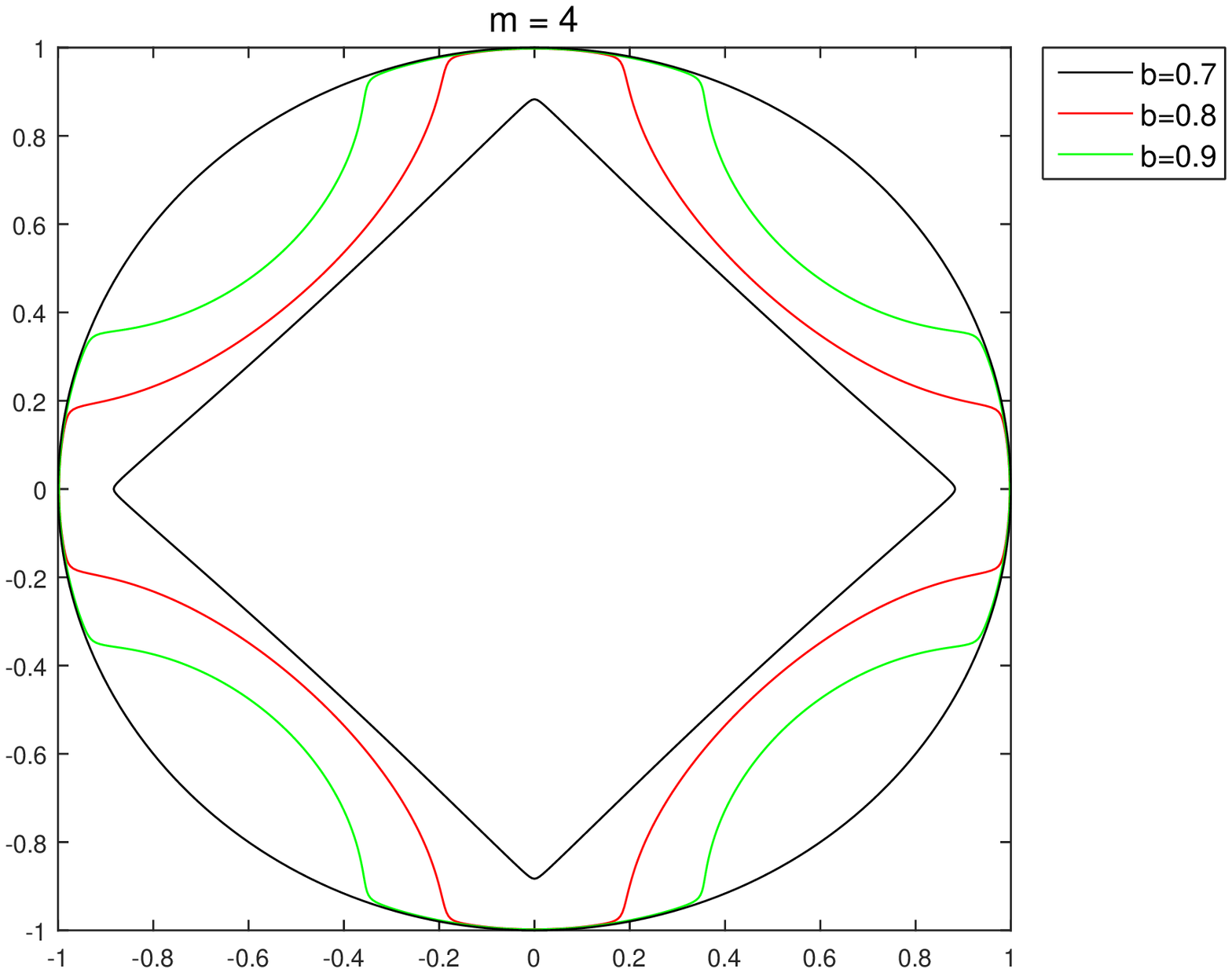}
\caption{Approximations to the limiting $V$-states corresponding to $1 \le m \le 4$, for different $b$. $N = 256\times m$. The values of $\Omega$ corresponding to the plots are given in Table \ref{t:LimitingOmegas0}.}
\label{f:limiting0holem1234}
\end{figure}

\begin{table}[!htb]
$$
\begin{tabular}{|c|c|c|c|c|}
\hline
b / $m$ & 1 & 2 & 3 & 4
    \\
\hline
0.9 & 0.3749 & 0.4057 & 0.4199 & 0.4283
    \\
\hline
0.8 & 0.3251 & 0.3589 & 0.3755 & 0.3859
    \\
\hline
0.7 & 0.2900 & 0.3163 & 0.3321 & 0.3650
    \\
\hline
0.6 & 0.2640 & 0.2731 & 0.3144 & 0.3572
    \\
\hline
0.5 & 0.2459 & 0.2363
    \\
\cline{1-3}
0.4 & 0.1964 & 0.2018
    \\
\cline{1-3}
\end{tabular}
$$
\caption{Values of $\Omega$ for the $V$-states plotted in Figure \ref{f:limiting0holem1234}.}
\label{t:LimitingOmegas0}
\end{table}

\subsection{Doubly-connected $V$-states}

\subsubsection{Numerical obtention}

Given a doubly-connected domain $D$ with outer boundary $z_1(\theta)$ and inner boundary $z_2(\theta)$, where $\theta\in[0,2\pi[$ is the Lagrangian parameter, and $z_1$ and $z_2$ are  parameterized, $D$ is a $V$-state if and only if its boundaries satisfy the following equations:
\begin{equation}
\label{e:condition1z1}
\begin{split}
\operatorname{Re}\Bigg\{\Bigg(2\Omega \overline{z_1(\theta)} & + \frac{1}{2\pi i}\int_0^{2\pi}\frac{\overline{z_1(\theta) - z_1(\phi)}}{z_1(\theta)-z_1(\phi)}z_{1,\phi}(\phi)d\phi - \frac{1}{2\pi i}\int_0^{2\pi}\frac{\overline{z_1(\theta) - z_2(\phi)}}{z_1(\theta)-z_2(\phi)}z_{2,\phi}(\phi)d\phi
    \cr
& - \frac{1}{2\pi i}\int_0^{2\pi}\frac{|z_1(\phi)|^2}{1-z_1(\theta)z_1(\phi)}z_{1,\phi}(\phi)d\phi
    \cr
& + \frac{1}{2\pi i}\int_0^{2\pi}\frac{|z_2(\phi)|^2}{1-z_1(\theta)z_2(\phi)}z_{2,\phi}(\phi)d\phi\Bigg) z_{1,\theta}(\theta)\Bigg\}=0,
\end{split}
\end{equation}

\begin{equation}
\label{e:condition1z2}
\begin{split}
\operatorname{Re}\Bigg\{\Bigg(2\Omega \overline{z_2(\theta)} & + \frac{1}{2\pi i}\int_0^{2\pi}\frac{\overline{z_2(\theta) - z_1(\phi)}}{z_2(\theta)-z_1(\phi)}z_{1,\phi}(\phi)d\phi - \frac{1}{2\pi i}\int_0^{2\pi}\frac{\overline{z_2(\theta) - z_2(\phi)}}{z_2(\theta)-z_2(\phi)}z_{2,\phi}(\phi)d\phi
    \cr
& - \frac{1}{2\pi i}\int_0^{2\pi}\frac{|z_1(\phi)|^2}{1-z_2(\theta)z_1(\phi)}z_{1,\phi}(\phi)d\phi
    \cr
& + \frac{1}{2\pi i}\int_0^{2\pi}\frac{|z_2(\phi)|^2}{1-z_2(\theta)z_2(\phi)}z_{2,\phi}(\phi)d\phi\Bigg) z_{2,\theta}(\theta)\Bigg\}=0.
\end{split}
\end{equation}

\noindent As in the simply-connected case, we use a pseudo-spectral method to find $V$-states. We discretize $\theta\in[0,2\pi[$ in $N$ equally spaced nodes $\theta_i = 2\pi i/N$, $i = 0, 1, \ldots, N-1$, where $N$ has to be large enough. Then, since $z_1$ and $z_2$ never intersect, all the integrals in \eqref{e:condition1z1} and \eqref{e:condition1z2} can be evaluated numerically with spectral accuracy at a node $\theta = \theta_i$ by means of the trapezoidal rule, exactly as in \eqref{e:trapezoidal}.

In order to obtain doubly connected $m$-fold $V$-states, we approximate $z_1$ and $z_2$ as in \eqref{e:z0cos}:
\begin{equation}
\label{e:z1z2cos}
z_1(\theta) = e^{i\theta}\left[b_1 + \sum_{k = 1}^M a_{1,k}\cos(m\,k\,\theta)\right], \qquad
z_2(\theta) = e^{i\theta}\left[b_2 + \sum_{k = 1}^M a_{2,k}\cos(m\,k\,\theta)\right],
\end{equation}

\noindent where the mean outer and inner radii are respectively $b_1$ and $b_2$; and we are imposing that $z_1(-\theta) = \bar z_1(\theta)$ and $z_2(-\theta) = \bar z_2(\theta)$, i.e., we are looking for $V$-states symmetric with respect to the $x$-axis. Again, if we choose $N$ of the form $N = m2^r$, then $M = \lfloor (m2^r-1)/(2m)\rfloor = 2^{r-1}-1$.

We introduce \eqref{e:z1z2cos} into \eqref{e:condition1z1} and \eqref{e:condition1z2}, and, as in \eqref{e:V-State0conditions}, we approximate the errors in \eqref{e:condition1z1} and \eqref{e:condition1z2} by their $M$-term sine expansions, which are respectively $\sum_{k = 1}^M b_{1,k}\sin(m\,k\,\theta)$ and $\sum_{k = 1}^M b_{2,k}\sin(m\,k\,\theta)$. Then, as in \eqref{e:F0}, the resulting systems of equations can be represented in a very compact way as
\begin{equation}
\label{e:F1}
\mathcal F_{b_1,b_2,\Omega}(a_{1,1}, \ldots, a_{1,M}\,,\, a_{2,1}, \ldots, a_{2,M}) = (b_{1,1}, \ldots, b_{1,M}\,,\, b_{2,1}, \ldots, b_{2,M}),
\end{equation}

\noindent for a certain $\mathcal F_{b_1,b_2,\Omega}\ : \ \mathbb{R}^{2M}\to\mathbb{R}^{2M}$. Remark that, for any $\Omega$, and any $0 < b_2 < b_1 < 1$, we have trivially $\mathcal F_{b_1b_2,\Omega}(\mathbf 0) = \mathbf 0$, i.e., any circular annulus is a solution of the problem. Therefore, the obtention of a doubly-connected $V$-state is reduced to finding numerically $\{a_{1,k}\}$ and $\{a_{2,k}\}$, such that $(a_{1,1}, \ldots, a_{1,M}\,,\, a_{2,1}, \ldots, a_{2,M})$ is a nontrivial root of \eqref{e:F1}. To do so, we discretize the $(2M\times2M)$-dimensional Jacobian matrix $\mathcal J$ of $\mathcal F_{b_1,b_2,\Omega}$ as in \eqref{e:derivative0}, taking $h = 10^{-9}$:
\begin{equation}
\label{e:derivative}
\begin{split}
& \frac{\partial \mathcal F_{b_1,b_2,\Omega}(a_{1,1}, \ldots, a_{1,M}\,,\, a_{2,1}, \ldots, a_{2,M})}{\partial a_{1,1}}
    \cr
& \approx \frac{\mathcal F_{b_1,b_2,\Omega}(a_{1,1} + h, a_{1,2}, \ldots, a_{1,M}\,,\, a_{2,1}, \ldots, a_{2,M}) - \mathcal F_{b_1,b_2,\Omega}(a_{1,1}, \ldots, a_{1,M}\,,\, a_{2,1}, \ldots, a_{2,M})}{h},
\end{split}
\end{equation}

\noindent Then, the sine expansion of \eqref{e:derivative} gives us the first row of $\mathcal J$, and so on. Hence, if the $n$-th iteration is denoted by $(a_{1,1}, \ldots, a_{1,M}\,,\, a_{2,1}, \ldots, a_{2,M})^{(n)}$, then the $(n+1)$-th iteration is given by
\begin{equation*}
\begin{split}
& (a_{1,1}, \ldots, a_{1,M}\,,\, a_{2,1}, \ldots, a_{2,M})^{(n+1)}
    \cr
& = (a_{1,1}, \ldots, a_{1,M}\,,\, a_{2,1}, \ldots, a_{2,M})^{(n)} - \mathcal F_{b_1,b_2,\Omega}\left((a_{1,1}, \ldots, a_{1,M}\,,\, a_{2,1}, \ldots, a_{2,M})^{(n)}\right)\cdot [\mathcal J^{(n)}]^{-1},
\end{split}
\end{equation*}

\noindent where $[\mathcal J^{(n)}]^{-1}$ denotes the inverse of the Jacobian matrix at $(a_{1,1}, \ldots, a_{1,M}\,,\, a_{2,1}, \ldots, a_{2,M})^{(n)}$. To make this iteration converge, it is usually enough to perturb the annulus by assigning a small value to ${a_{1,1}}^{(0)}$ or ${a_{2,1}}^{(0)}$, and leave the other coefficients equal to zero. Our stopping criterion is
\begin{equation*}
\max\left|\sum_{k = 1}^M b_{1,k}\sin(m\,k\,\theta)\right| < tol \quad \wedge \quad \max\left|\sum_{k = 1}^M b_{2,k}\sin(m\,k\,\theta)\right| < tol,
\end{equation*}

\noindent where $tol = 10^{-13}$. As in \cite{H-F-M-V} and \cite{H-H-H}, $a_{1,1}\cdot a_{2,1}<0$, so, for the sake of coherence, we change eventually the sign of all the coefficients $\{a_{1,k}\}$ and $\{a_{2,k}\}$, in order that, without loss of generality, $a_{1,1}>0$ and $a_{2,1}<0$.

\subsubsection{Numerical discussion}

Proposition \ref{asympbeh} states that, given $b_1\in]0,1[$ and $m\ge 2$, there is a certain $b_m^\star$, such that $b_2\in[0, b_m^\star]$. Let us recall that $b_m^\star$ is the only solution of
\begin{equation*}
m  = \frac{2+2(x/b_1)^m-(b_1^m+x^m)^2}{1-(x/b_1)^2}.
\end{equation*}

\noindent In Figure \ref{f:bmstar}, we have plotted $b_m^\star$ as a function of $b_1$, for $m = 2, \ldots, 20$.

\begin{figure}[!htb]
\center
\includegraphics[width=0.5\textwidth, clip=true]{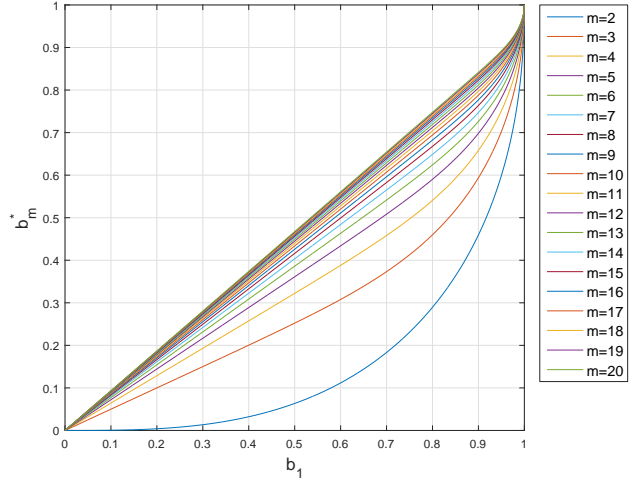}
\caption{$b_m^\star$ as a function of $b_1$, for $m = 2, \ldots, 20$.}
\label{f:bmstar}
\end{figure}

If we make $b_2 = b_m^\star$, then the discriminant $\Delta_m$ defined in Theorem \ref{thmV02} is equal to zero; and, in that case, $\Omega_m^+ = \Omega_m^-$, or, equivalently, $\lambda_m^+ = \lambda_m^-$.  Note that the relation between $\Omega^\pm$ and  $\lambda_m^\pm$ is given by
$$
\Omega_m^\pm=\frac12(1-\lambda_m^\mp).
$$In Figure \ref{f:lambda0.250.50.750.99}, we have plotted $\lambda_m^\pm$ as a function of $b_2\in[0, b_m^\star]$, for $m = 2, \ldots, 20$, and $b_1 \in \{0.25, 0.5, 0.75, 0.99\}$. 
We have also plotted in black the special case $m = 1$, where $b_2\in[0, b_1]$, $\lambda_1^+=1+b_2^2-b_1^2$, and $\lambda_1^-=(b_2/b_1)^2$. Observe that, whereas the curves $\lambda^+_m$ and $\lambda^-_m$ are disjoint for $m\ge 2$; $\lambda_1^+$ may intersect $\lambda_m^+$ or $\lambda_m^-$. It is particularly interesting to see what happens when $b_1$ is close to one; indeed, when $b_1 = 0.99$, the curves $\lambda_m^-$ become practically indistinguishable.
%
\begin{figure}[!htb]
\center
\includegraphics[width=0.5\textwidth, clip=true]{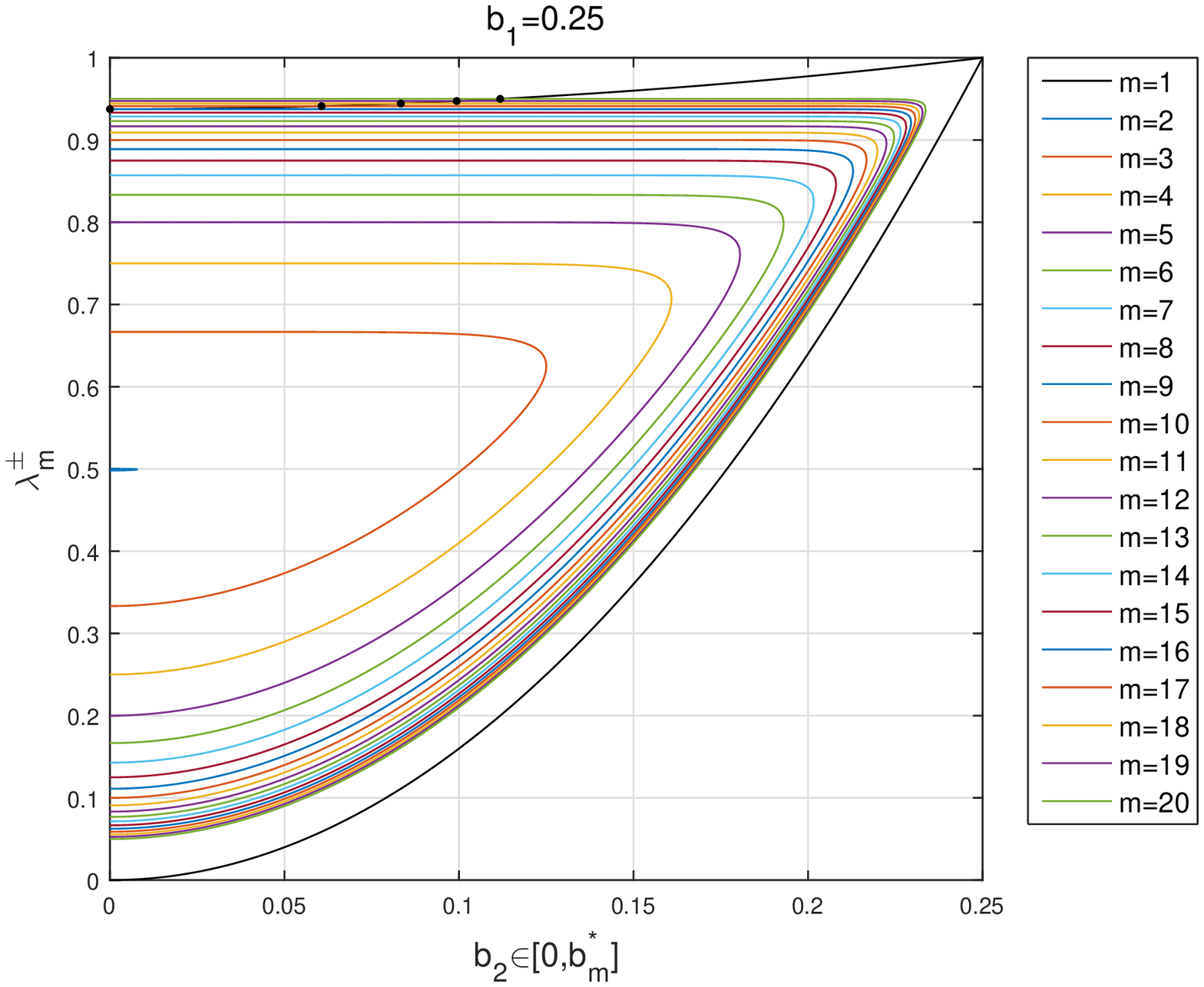}~
\includegraphics[width=0.5\textwidth, clip=true]{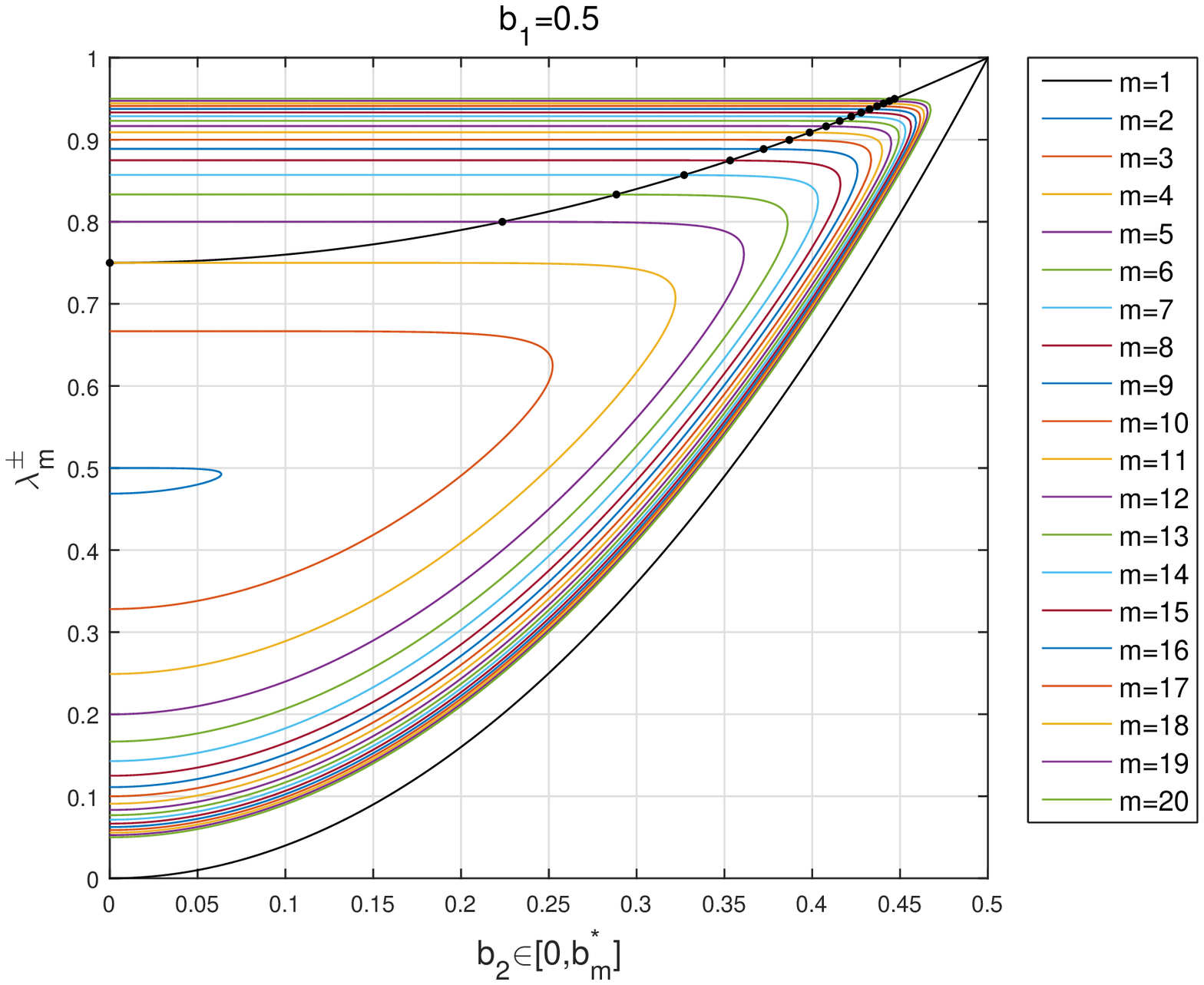}

\bigskip

\includegraphics[width=0.5\textwidth, clip=true]{T3.eps}~
\includegraphics[width=0.5\textwidth, clip=true]{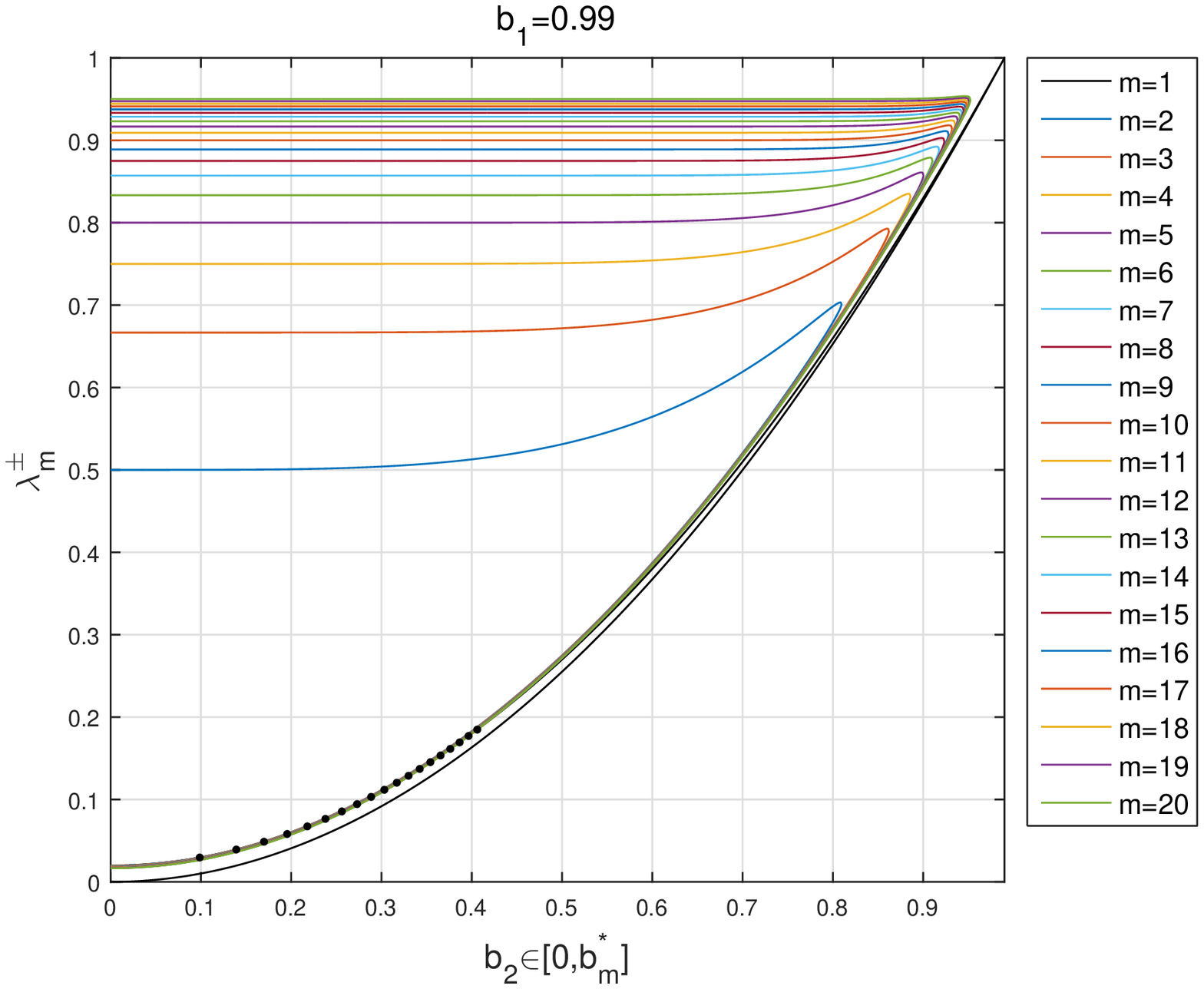}
\caption{$\lambda_m^\pm$ as a function of $b_2\in[0, b_m^\star]$, for $m = 2, \ldots, 20$, together with the case $m = 1$ (black), for $b_1 \in \{0.25, 0.5, 0.75, 0.99\}$. We have marked with a small black dot the intersections happening between the case $m = 1$ and the other cases.}
\label{f:lambda0.250.50.750.99}
\end{figure}

Although Figure \ref{f:lambda0.250.50.750.99} gives a fairly good idea of the structure of $\lambda_m^\pm$, it may be clarifying to show globally how the curves in Figure \ref{f:lambda0.250.50.750.99} behave as $b_1$ changes, for a fixed $m$. In Figure \ref{f:lambda3D}, we have plotted $\lambda_m^\pm$ as a function of $b_2\in[0, b_m^\star]$, for $m = 2, 3, 4$, and for all $b_1\in]0,1[$; in such a way that, for a given $b_1$, the intersection between $z = b_1$ and the resulting surfaces yields curves equivalent to those in Figure \ref{f:lambda3D}. In general, the surfaces corresponding to $m\ge 3$ are very similar. On the other hand, Figure \ref{f:lambda3D} shows that, when $m = 2$, and $b_1$ is \textit{not too large}, the size of the curves $(b_2, \lambda_2^\pm)$ is \textit{very small}; indeed, in Figure \ref{f:lambda0.250.50.750.99}, $(b_2, \lambda_2^\pm)$ is hardly visible, when $b_1 = 0.25$. A similar observation can be done with respect to the case $m = 2$ in \mbox{Figure \ref{f:bmstar},} which is markedly different from the others.
\begin{figure}[!htb]
\center
\includegraphics[width=0.3333\textwidth, clip=true]{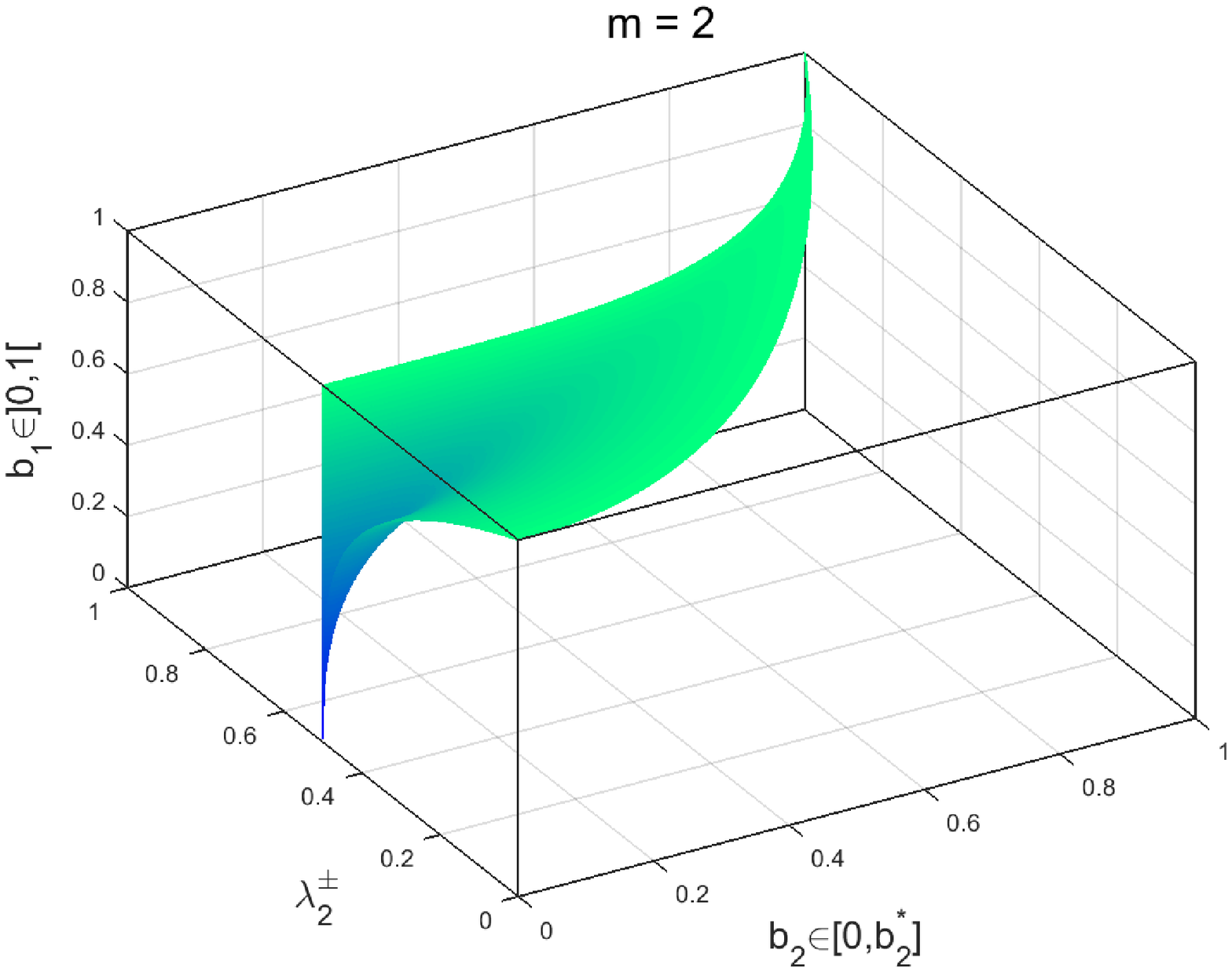}~
\includegraphics[width=0.3333\textwidth, clip=true]{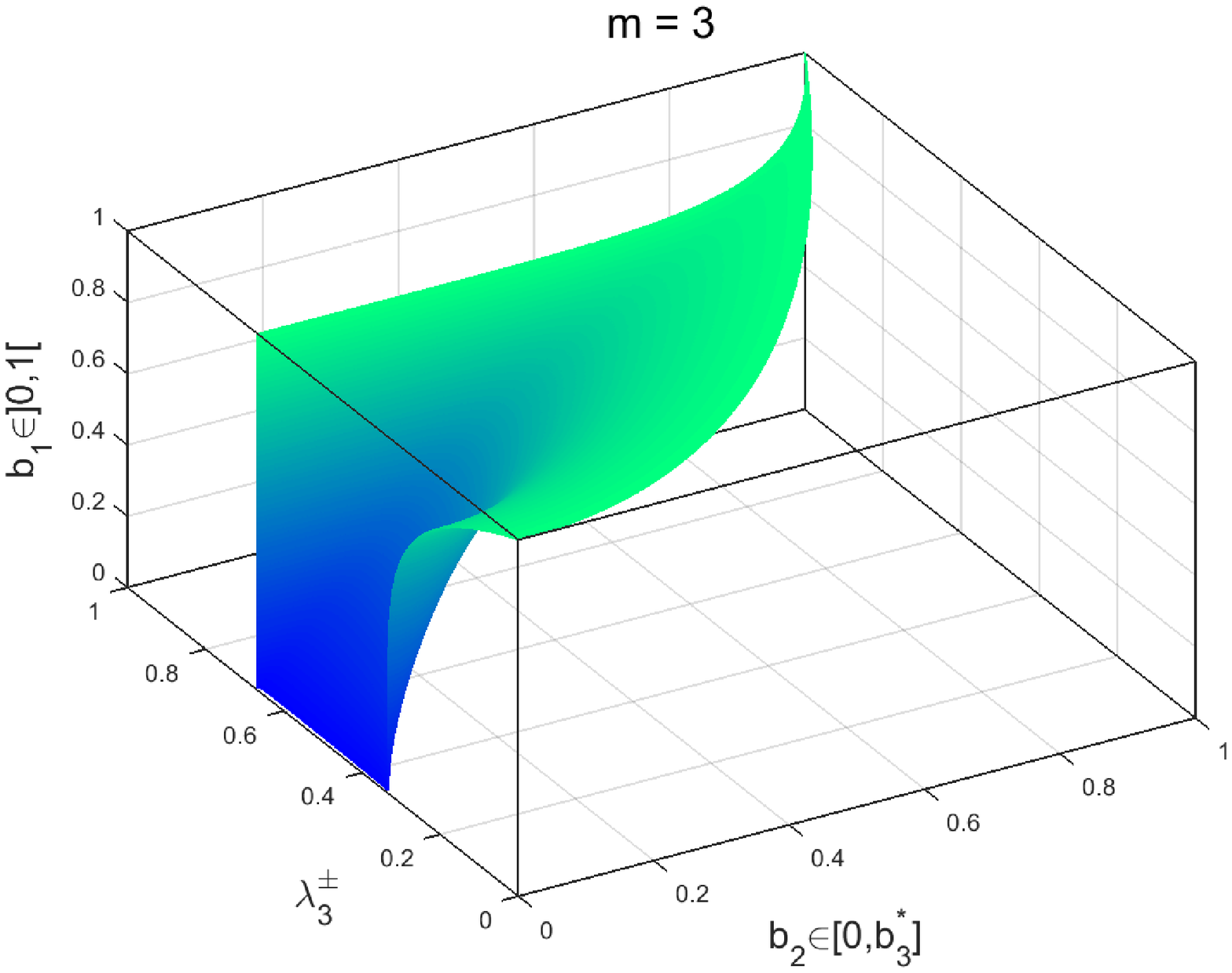}~
\includegraphics[width=0.3333\textwidth, clip=true]{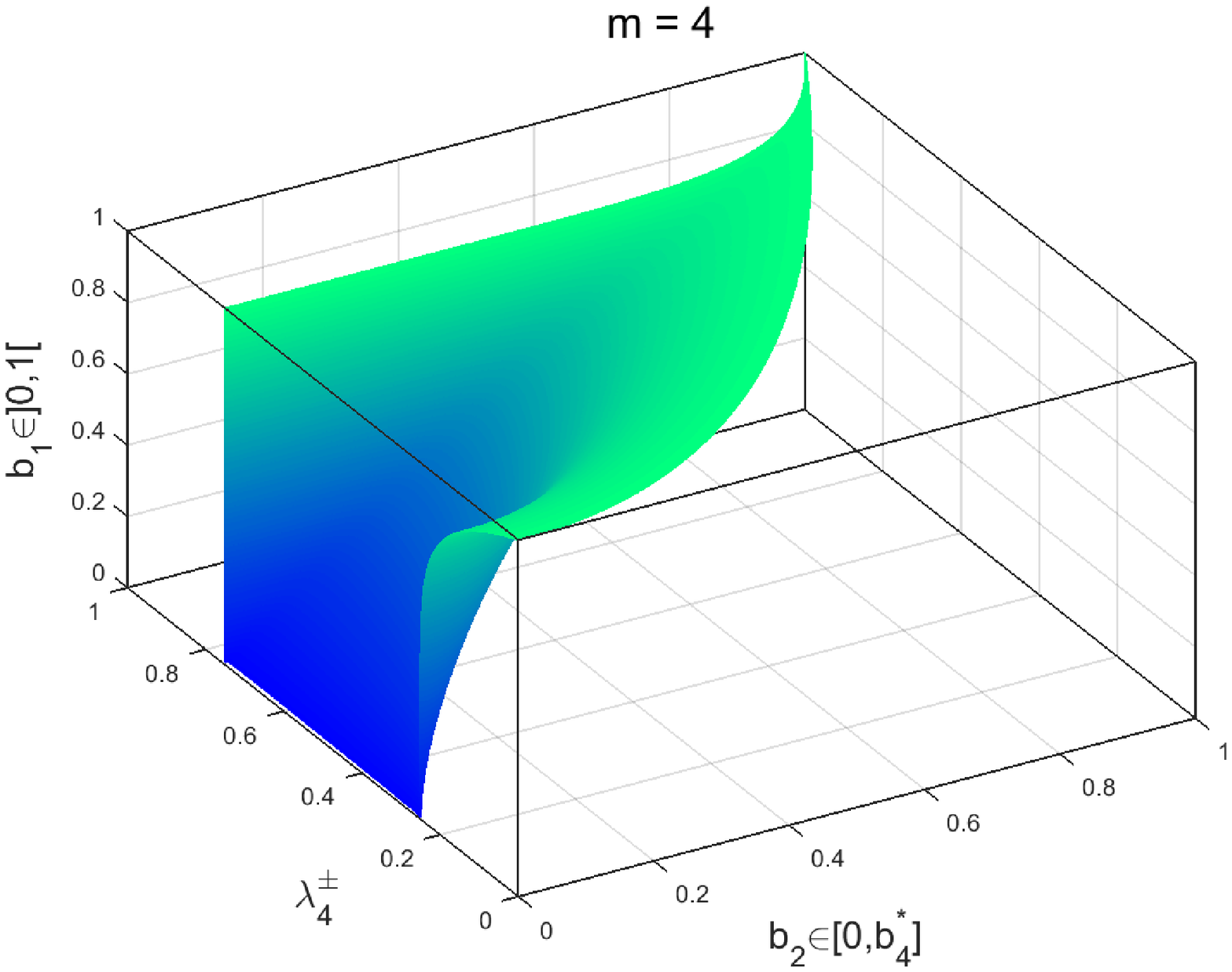}~
\caption{$\lambda_m^\pm$ as a function of $b_2\in[0, b_m^\star]$, for $m = 2, 3, 4$, and for all $b_1\in]0,1[$.}
\label{f:lambda3D}
\end{figure}

As in the simply-connected case, we use $\Omega = (1 - \lambda) / 2$ as our bifurcation parameter. In order to treat the saddle-node bifurcation points \cite{Kil} that may appear during the bifurcation process, we use again the techniques described in \cite{H-H-H}.

Before illustrating the shape of the doubly-connected $V$-states, let us mention that the situation is much more involved than in the simply-connected case, where there were roughly two situations for all $m$: $b$ \textit{close} to one, and $b$ \textit{not so close} to one. Indeed, we have to play now with both the proximity of $b_1$ to one, and that of $b_2$ to $b_m^\star$. Furthermore, we can start the bifurcation from the annulus of radii $b_1$ and $b_2$ at two different values of $\Omega$, i.e., $\Omega_m^+$ and $\Omega_m^-$. Finally, the case $m = 1$ needs to be studied individually. All in all, we have detected the following scenarios.

When $m \ge 3$, there are roughly three cases, when starting to bifurcate at $\Omega_m^+$; and two cases, when starting to bifurcate at $\Omega_m^-$. More precisely, if we start to bifurcate at $\Omega_m^+$, we have to distinguish whether:
\begin{itemize}
\item $b_2$ is \textit{very close} to $b_m^\star$. In that case, it seems possible to obtain $V$-states for all $\Omega\in]\Omega_m^-,\Omega_m^+[$, very much like in \cite{H-F-M-V}, irrespectively of the size of $b_1$. For example, in Figure \ref{f:VState1holen4b10.8b20.53}, we have calculated the $V$-states corresponding to $m = 4$, $b_1 = 0.8$, $b_2 = 0.53$. Observe that $b_4^\star = 0.5407\ldots$, i.e., we have chosen $b_2$ \textit{close enough} to $b_4^\star$. On the right-hand side, we have plotted the bifurcation diagram of the coefficients $a_{1,1}$ and $a_{2,1}$ in \eqref{e:z1z2cos} against $\Omega$, which shows that there is indeed a continuous bifurcation branch that joins $\Omega_m^-$ and $\Omega_m^+$, where $\Omega_4^- = 0.1335\ldots$, $\Omega_4^+ = 0.1671\ldots$. On the left-hand side, we have plotted $V$-states for four different values of $\Omega\in]\Omega_m^-,\Omega_m^+[$.
\begin{figure}[!htb]
\center
\includegraphics[width=0.5\textwidth, clip=true]{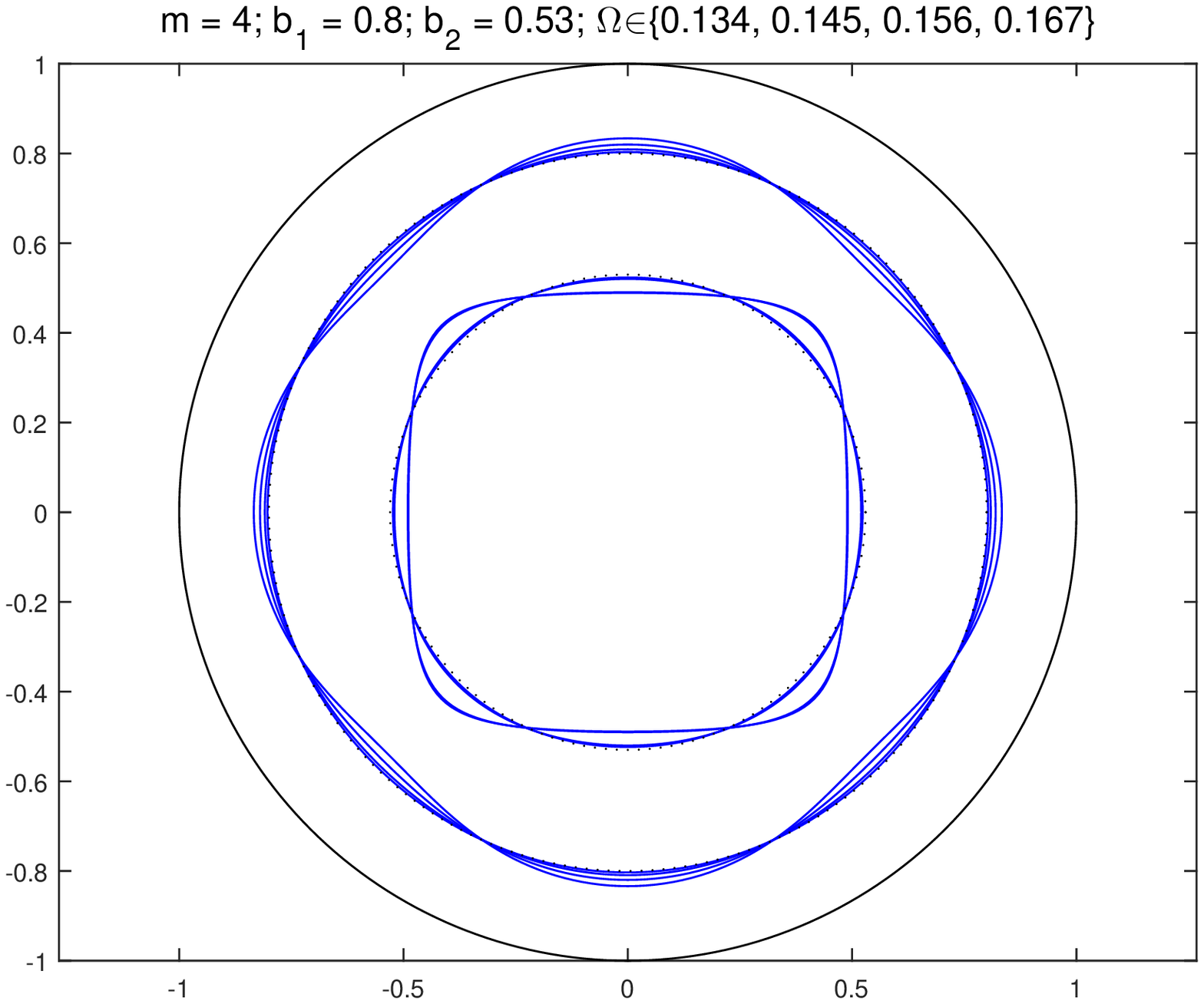}~
\includegraphics[width=0.5\textwidth, clip=true]{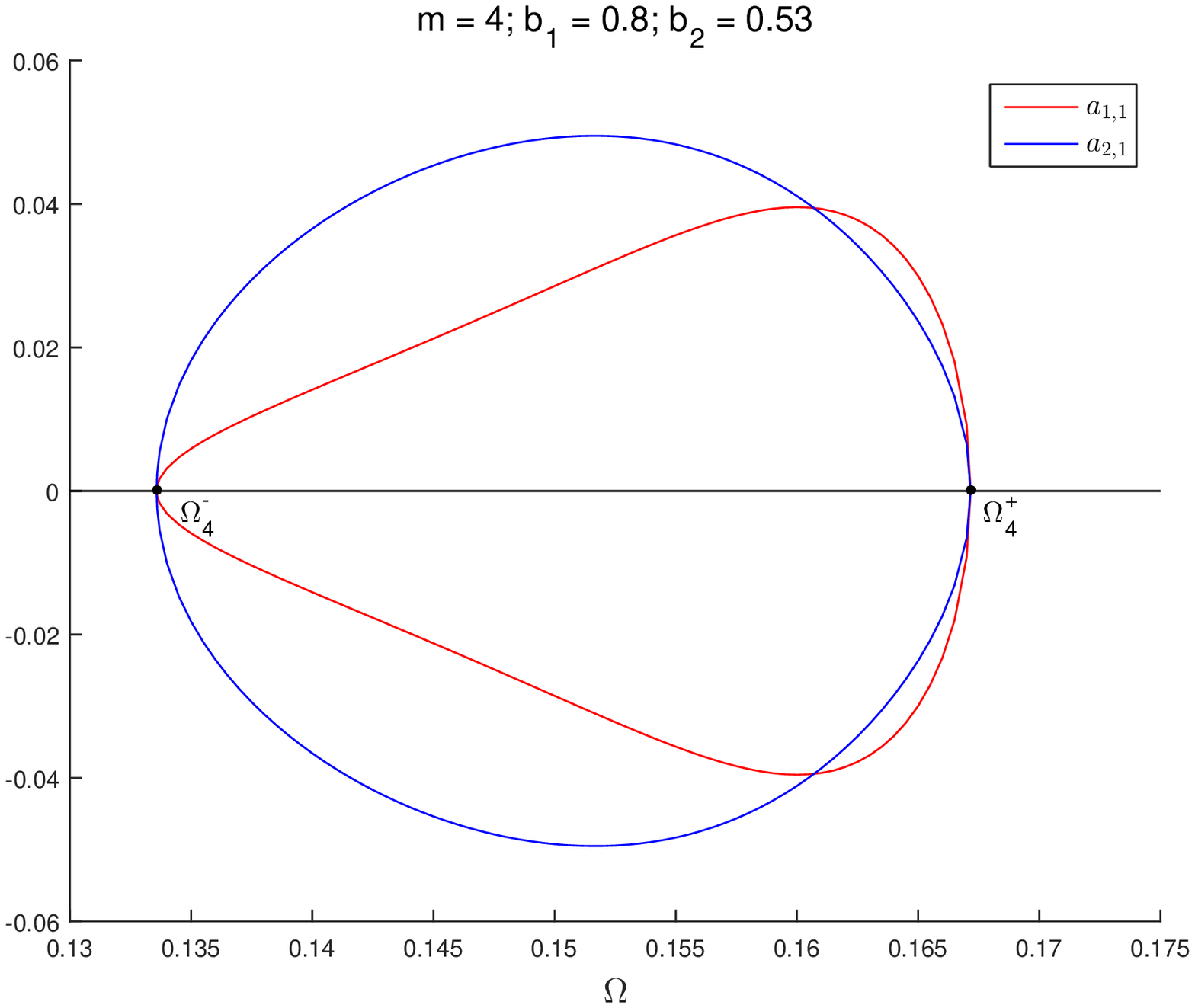}
\caption{Left: $V$-states corresponding to $m = 4$, $b_1 = 0.8$, $b_2 = 0.53$, and several values of $\Omega$. Right: bifurcation diagram. $N = 256$.
}
\label{f:VState1holen4b10.8b20.53}
\end{figure}

\item $b_1$ is \textit{close} to one, and $b_2$ is \textit{small enough}, there are limiting $V$-states, for which the distance between the outer boundary $z_1$ and the unit circumference tends to zero; but the inner boundary $z_2$ does not deviate greatly from the circumference of radius $b_2$. On the left-hand side of Figure \ref{f:limitn4b10_8b20_3}, we have approximated the limiting $V$-state corresponding to $m = 4$, $b_1 = 0.8$, $b_2 = 0.3$. The shape of $z_1$ is not very far from the case $m = 4$, $b = 0.8$ of Figure \ref{f:limiting0holem1234}.

\begin{figure}[!htb]
\center
\includegraphics[width=0.5\textwidth, clip=true]{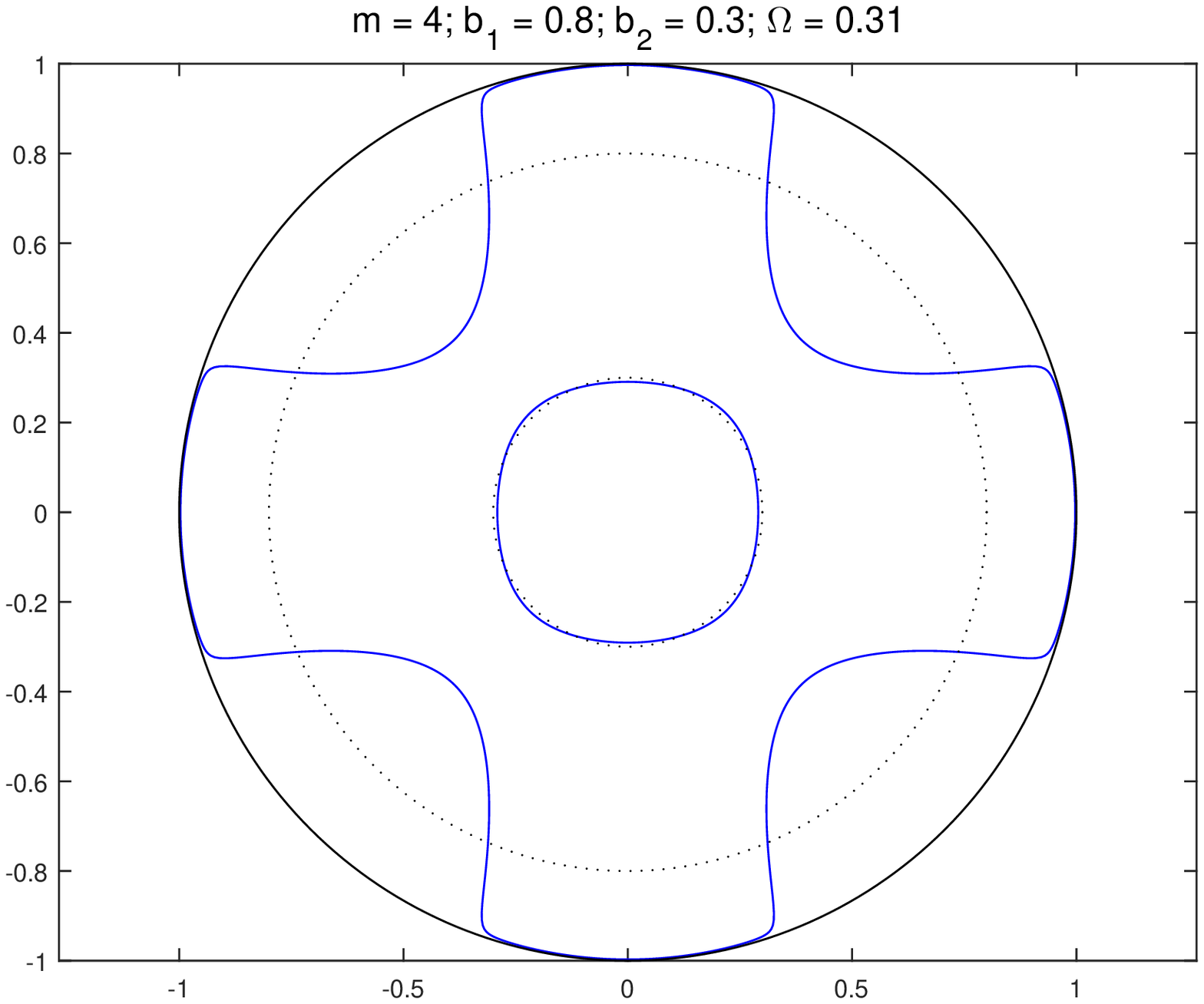}~
\includegraphics[width=0.5\textwidth, clip=true]{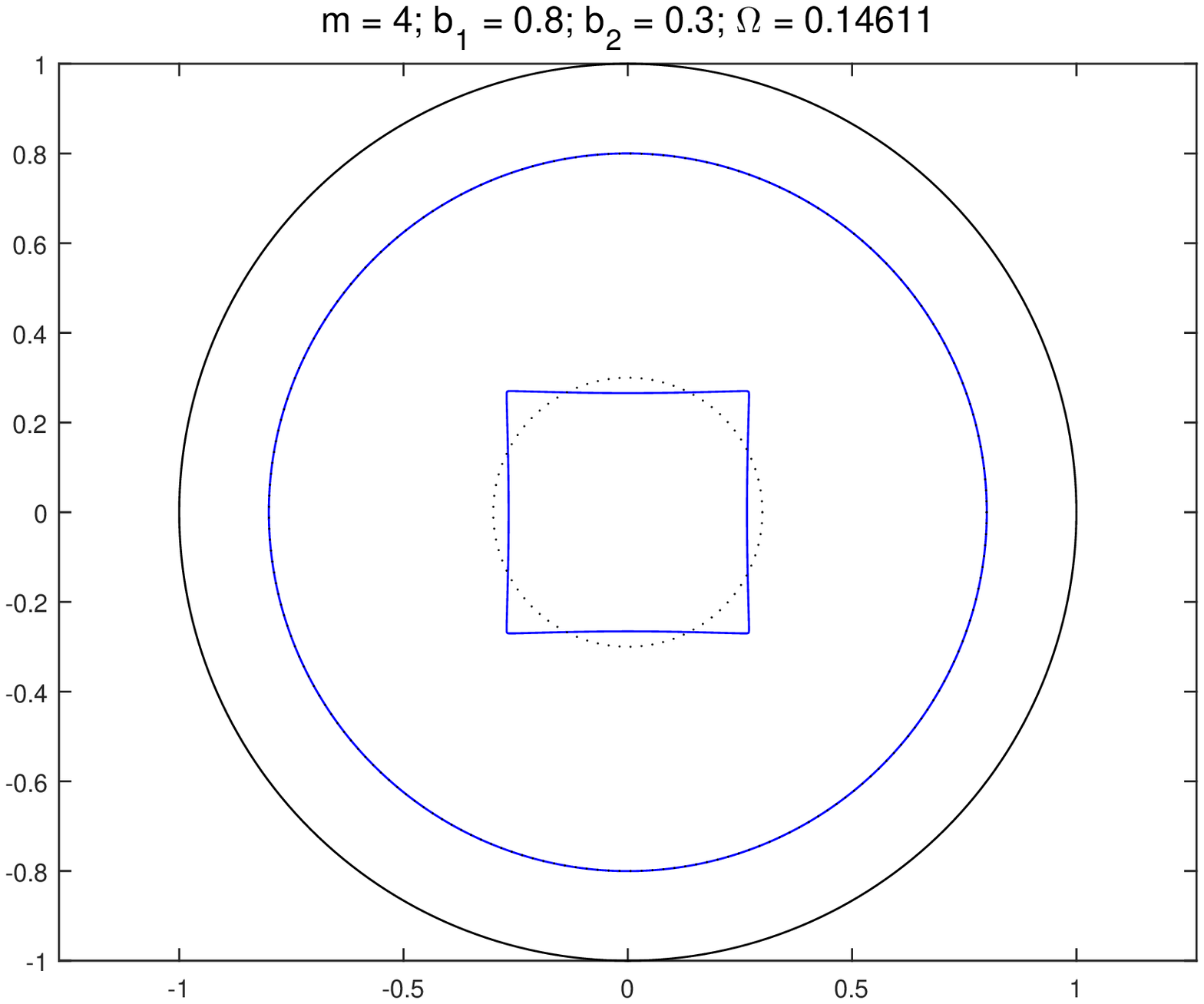}

\caption{Approximation to the limiting $V$-states corresponding to $m = 4$, $b_1 = 0.8$, $b_2 = 0.3$. Left: we have started to bifurcate at $\Omega_4^+ = 0.3256\ldots$, taking $\Omega < \Omega_4^+$. Right: we have started to bifurcate at $\Omega_4^- = 0.1250\ldots$, taking $\Omega > \Omega_4^-$. $N = 1024$.
}
\label{f:limitn4b10_8b20_3}
\end{figure}

\item $b_1$ and $b_2$ do not fit in the previous two cases. In that case, there are also limiting $V$-states, characterized by the appearance of corner-shaped singularities in $z_1$ or $z_2$. In Figure \ref{f:limitn4others}, we have approximated the limiting $V$-states corresponding to $m = 4$, $b_1 = 0.8$, $b_2 = 0.4$ (left); and to $m = 4$, $b_1 = 0.6$, $b_2 = 0.3$ (right). Observe that the influence of the rigid boundary seems less perceptible in the second example, which, accordingly, does not differ too much from those in \cite{H-F-M-V}.

\begin{figure}[!htb]
\center
\includegraphics[width=0.5\textwidth, clip=true]{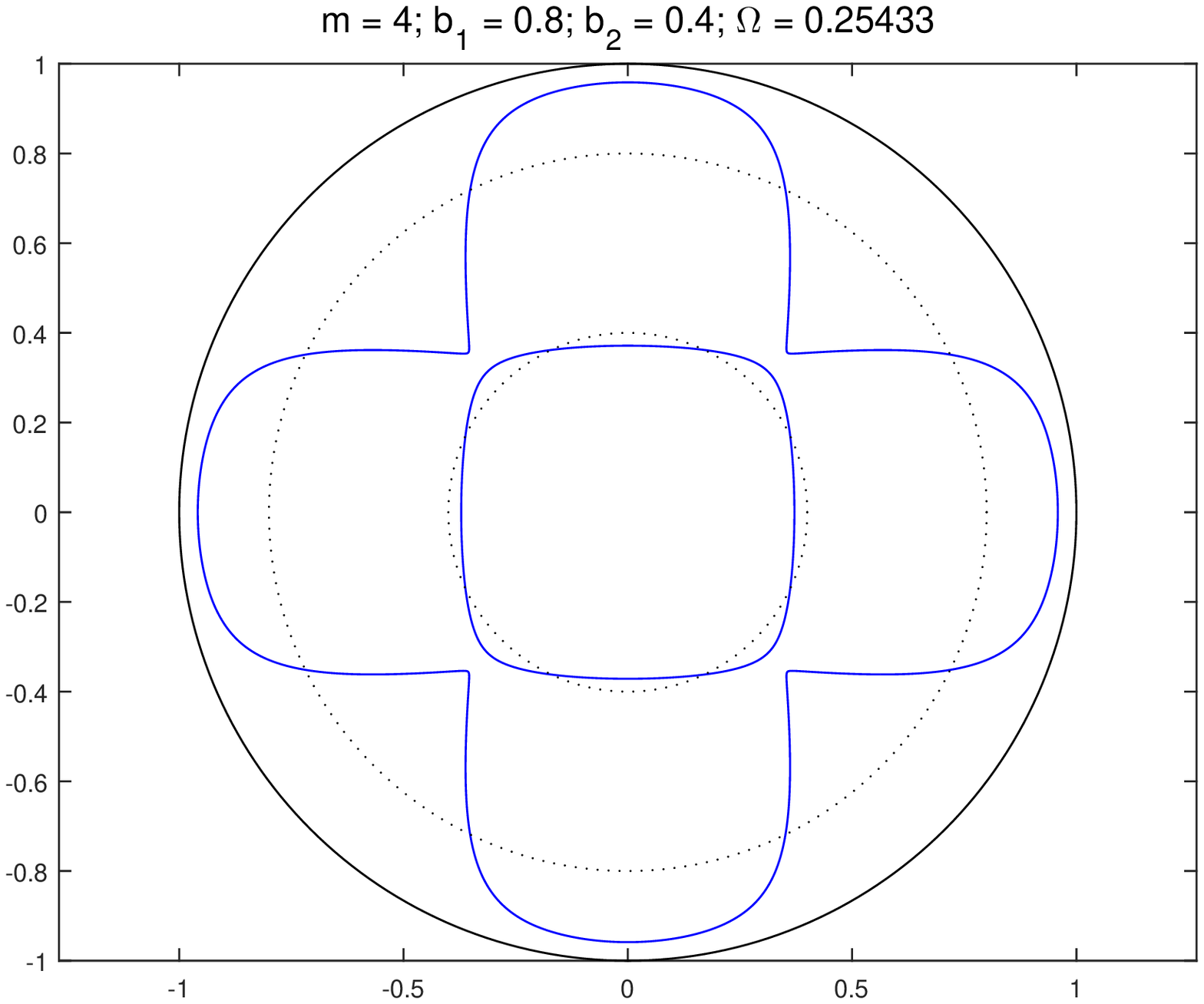}~
\includegraphics[width=0.5\textwidth, clip=true]{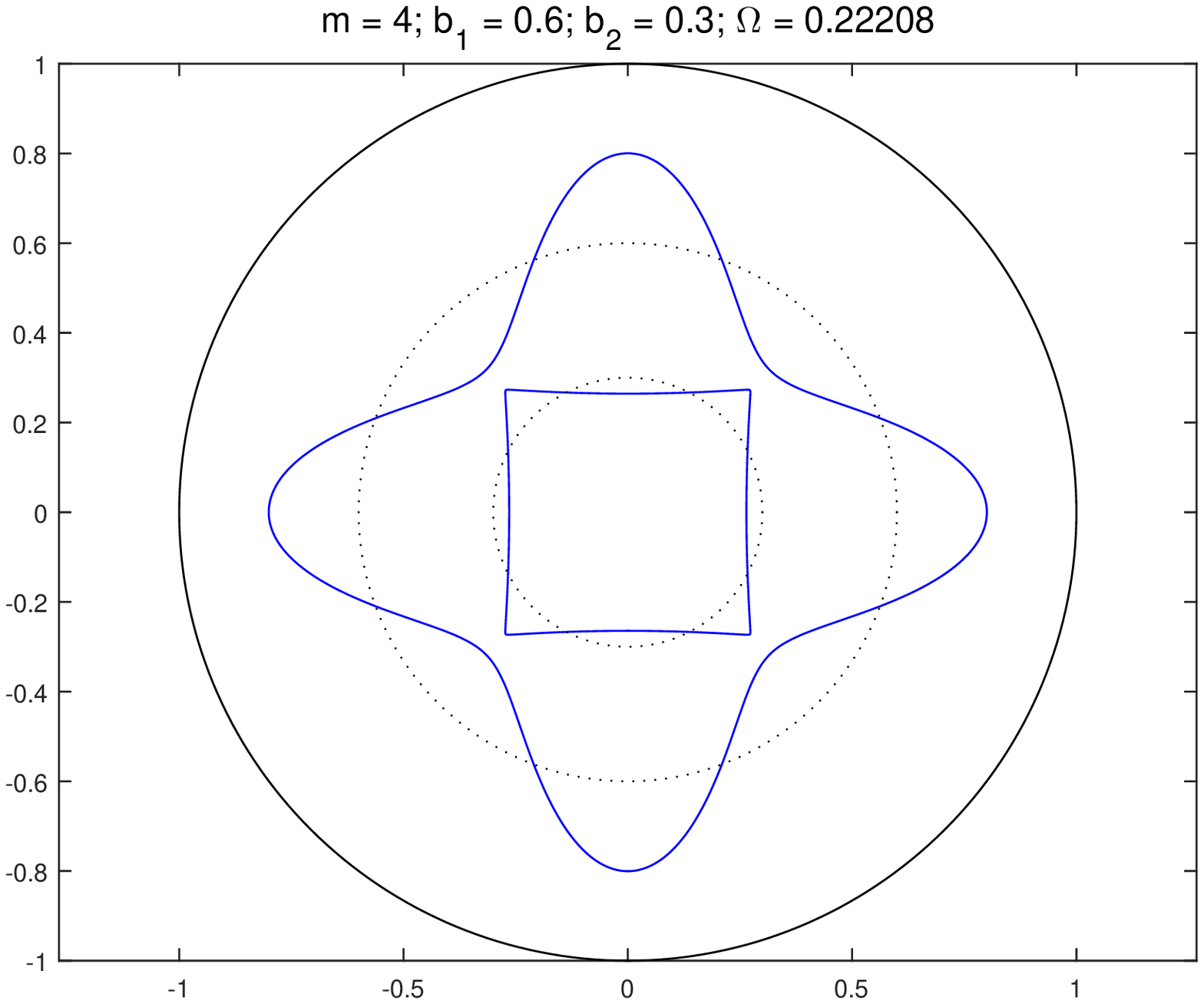}

\caption{Left: Approximation to the limiting $V$-state corresponding to $m = 4$, $b_1 = 0.8$, $b_2 = 0.4$, starting to bifurcate at $\Omega_4^+ = 0.2706$, taking $\Omega < \Omega_4^+$. Right: Approximation to the limiting $V$-state corresponding to $m = 4$, $b_1 = 0.6$, $b_2 = 0.3$, starting to bifurcate at $\Omega_4^+ = 0.2516$, taking $\Omega < \Omega_4^+$. $N = 1024$.
}
\label{f:limitn4others}
\end{figure}

    Although the distance between $z_1$ and the unit circumference is always strictly positive; the distance between $z_1$ and $z_2$ is sometimes very small, and we can not exclude in advance the existence of limiting $V$-states where $z_1$ and $z_2$ actually touch each other. For instance, after playing with the values of $b_1$ and $b_2$, we have found that the choice of $b_1 = 0.72$, $b_2 = 0.32$ enables us to find a $V$-state, such that the distance between $z_1$ and $z_2$ is of about $7\times10^{-3}$. This $V$-state is plotted in Figure \ref{f:limitn4selfintersection}, together with a zoom of one apparent intersection of the boundaries, that shows that there is really no intersection, and that the nodal resolution is adequate.

\begin{figure}[!htb]
\center
\includegraphics[width=0.5\textwidth, clip=true]{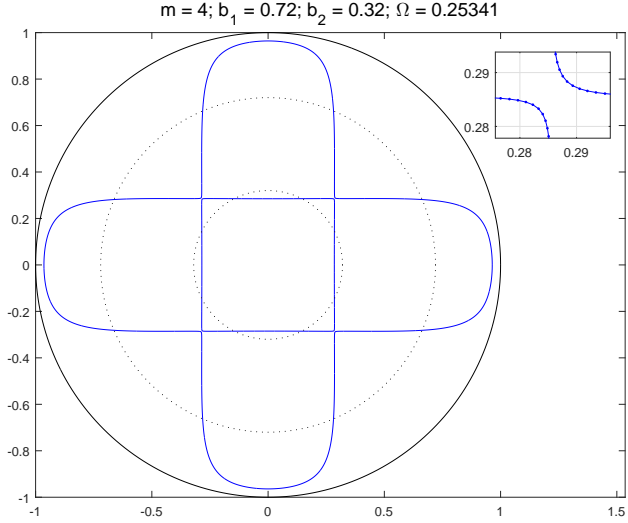}

\caption{Approximation to the limiting $V$-state corresponding to $m = 4$, $b_1 = 0.72$, $b_2 = 0.32$, starting to bifurcate at $\Omega_4^+ = 0.2851$, taking $\Omega < \Omega_4^+$. $N = 2048$. The zoom shows that that the boundaries are very close from each other, but there is no intersection.
}
\label{f:limitn4selfintersection}
\end{figure}

\end{itemize}

On the other hand, if we start to bifurcate at $\Omega_m^-$, we have to distinguish whether:
\begin{itemize}

\item $b_2$ is \textit{very close} to $b_m^\star$. This case has been explained above. In fact, it is irrelevant whether we start to bifurcate at $\Omega_m^-$ or at $\Omega_m^+$.

\item $b_2$ is not \textit{close enough} to $b_m^\star$. In that case, there are limiting $V$-states, characterized by the appearance of corner-shaped singularities in $z_2$, whereas the outer boundary $z_1$ does not deviate greatly from the circumference of radius $b_1$. On the right-hand side of Figure \ref{f:limitn4b10_8b20_3}, we have approximated the limiting $V$-state corresponding to $m = 4$, $b_1 = 0.8$, $b_2 = 0.3$. We have not bothered to plot the $V$-states corresponding to those in Figures \ref{f:limitn4others} and \ref{f:limitn4selfintersection}, but starting to bifurcate at $\Omega_m^-$, because they are virtually identical, up to a scaling of $z_2$. This case closely matches that in \cite{H-F-M-V}, and the inner boundary resembles the simply-connected $V$-states in \cite{DZ}.

\end{itemize}

Summarizing, if we compare the doubly-connected $V$-states just described, with those in \cite{H-F-M-V}, we conclude that the truly unique case here is when $b_1$ is \textit{close} to one, and $b_2$ is \textit{small enough}. In what regards the case $m = 2$, everything said above is applicable. For example, in Figure \ref{f:limitn2b10_9b20_2}, we have taken $b_1 = 0.9$ and $b_1 = 0.2$, i.e., a value of $b_1$ \textit{close} to one and a value of $b_2$ \textit{small enough}. On the left-hand side, we show an approximation to the limiting $V$-state appearing when starting to bifurcate at $\Omega_2^+$; note the clear parallelism with the case $m = 2$, $b = 0.9$ of Figure \ref{f:limiting0holem1234}, and with the left-hand side of Figure \ref{f:limitn4b10_8b20_3}. On the right-hand side, we show an approximation to the limiting $V$-state appearing when starting to bifurcate at $\Omega_2^-$; as in the right-hand side of Figure \ref{f:limitn4b10_8b20_3}, corner-shaped singularities seem to d
 evelop in $z_2$, whereas $z_1$ has barely deviated from a circumference.
\begin{figure}[!htb]
\center
\includegraphics[width=0.5\textwidth, clip=true]{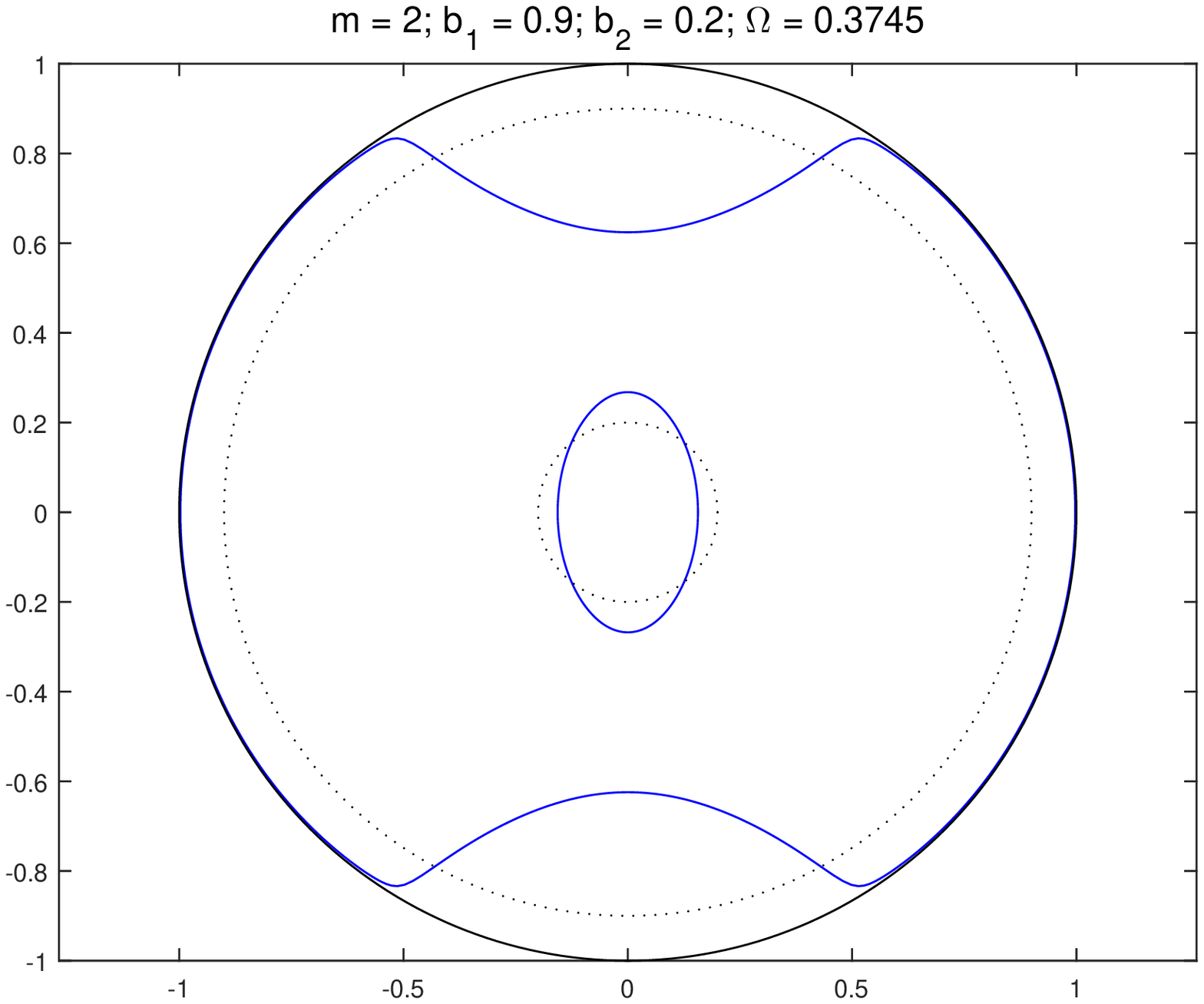}~
\includegraphics[width=0.5\textwidth, clip=true]{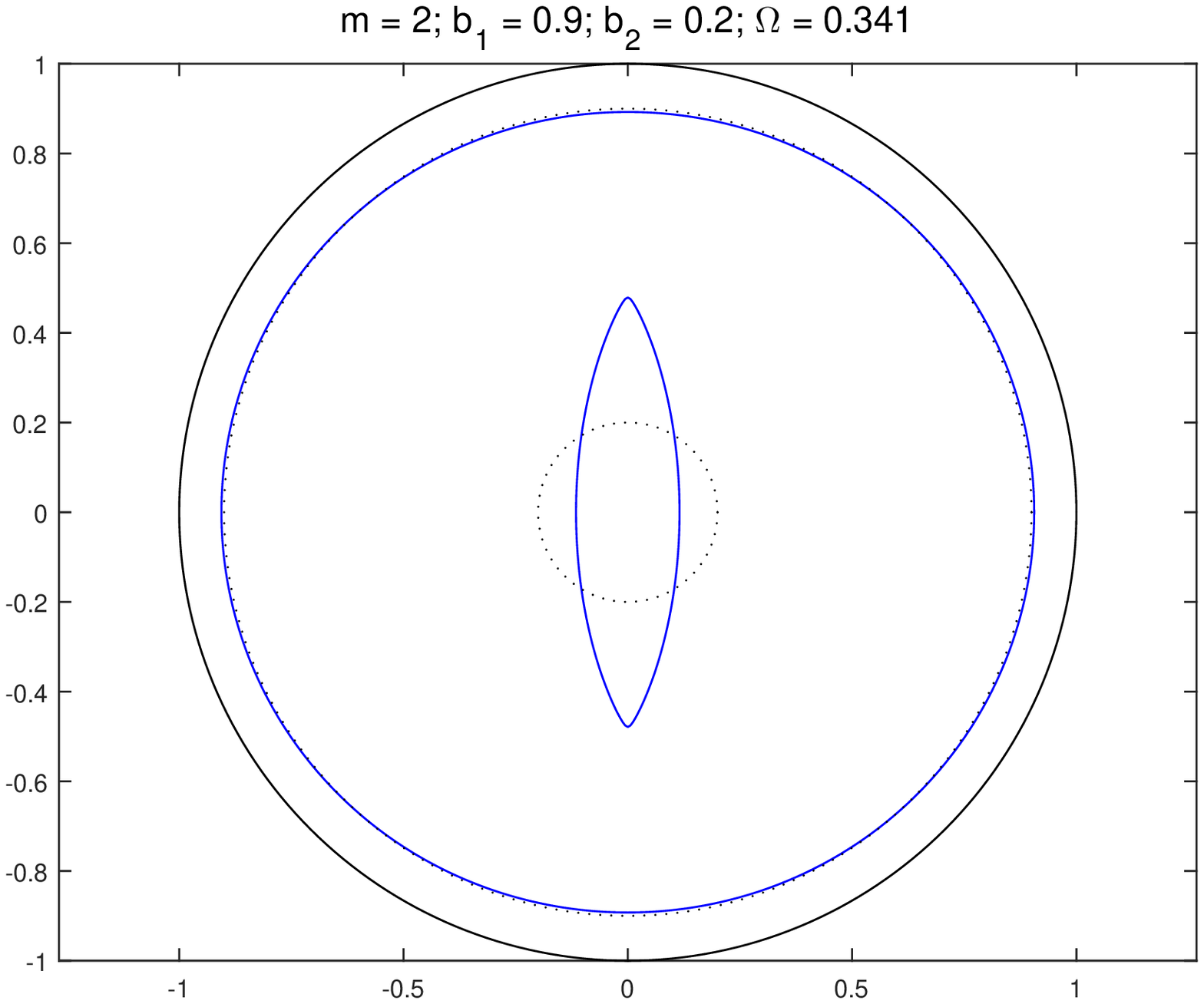}

\caption{Left: Approximation to the limiting $V$-states corresponding to $m = 2$, $b_1 = 0.9$, $b_2 = 0.2$, starting to bifurcate at $\Omega_2^+ = 0.3892\ldots$, taking $\Omega < \Omega_2^+$. Right: we have started to bifurcate at $\Omega_2^- = 0.2497\ldots$, taking $\Omega > \Omega_2^-$. $N = 512$.
}
\label{f:limitn2b10_9b20_2}
\end{figure}

The case $m = 1$ deserves also some comment. In Figure \ref{f:limitn1b10_9b20_3}, we have approximated the limiting $V$-states corresponding to $m = 1$, taking again a value of $b_1$ \textit{close} to one and a value of $b_2$ \textit{small enough}, more precisely, $b_1 = 0.9$, $b_2 = 0.3$. On the left-hand side, we have started to bifurcate at $\Omega_1^+$; and on the right-hand side, we have started to bifurcate at $\Omega_1^-$. It is remarkable that, in both cases, the distance of $z_1$ to the unit circumference is smaller than $10^{-2}$. Moreover, even if the $V$-state on the left-hand side is roughly in agreement with Figure \ref{f:limiting0holem1234}, and with the left-hand sides of Figures \ref{f:limitn4b10_8b20_3} and \ref{f:limitn2b10_9b20_2}; the $V$-state on the right-hand side exhibits a completely different, unexpected behavior.

\begin{figure}[!htb]
\center
\includegraphics[width=0.5\textwidth, clip=true]{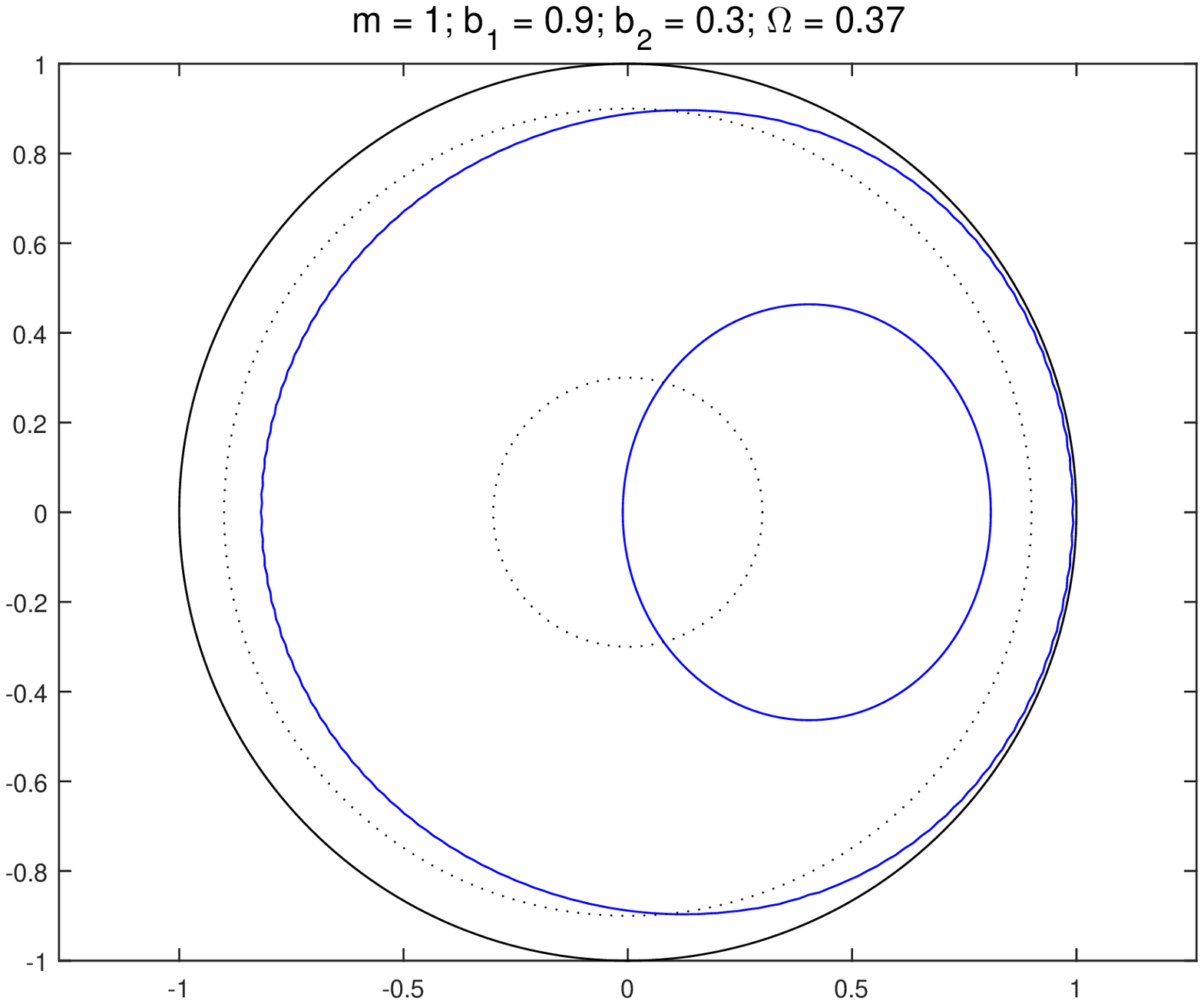}~
\includegraphics[width=0.5\textwidth, clip=true]{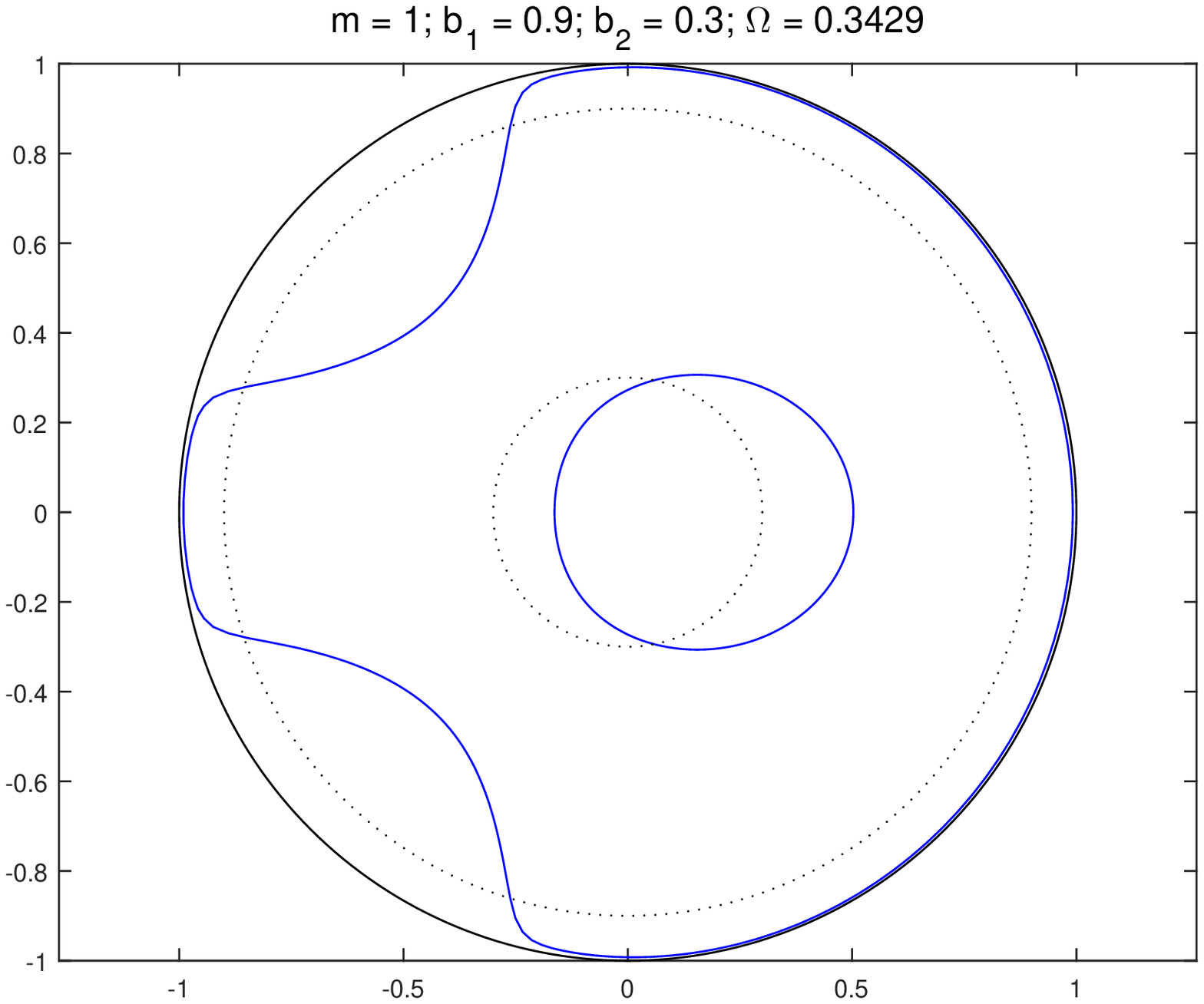}

\caption{Approximation to the limiting $V$-states corresponding to $m = 1$, $b_1 = 0.9$, $b_2 = 0.3$. Left: we have started to bifurcate at $\Omega_1^+ = 4 / 9$, taking $\Omega < \Omega_1^+$. Right: we have started to bifurcate at $\Omega_1^- = 0.36$, taking $\Omega > \Omega_1^-$. $N = 256$.
}
\label{f:limitn1b10_9b20_3}
\end{figure}








\begin{ackname}
Francisco de la Hoz was supported by the Basque Government, through the project IT641-13, and by the Spanish Ministry of Economy and Competitiveness, through the project MTM2014-53145-P. Taoufik Hmidi was partially supported by the ANR project Dyficolti ANR-13-BS01-0003- 01. Joan Mateu 
 was partially supported by the grants of Generalitat de Catalunya 2014SGR7, Ministerio de Econom{\'\i}a y Competitividad MTM 2013-4469,
ECPP7- Marie Curie project MAnET.
\end{ackname}

\end{document}